\documentclass[11pt]{amsart}
\usepackage{amsmath}
\usepackage{amsfonts}
\usepackage{amsthm}
\usepackage{amssymb}
\usepackage{mathrsfs}
\usepackage{graphicx}
\usepackage{enumerate}

\usepackage{ifpdf}
\ifpdf
\usepackage[pdftex,plainpages=false,hypertexnames=false,pdfpagelabels]{hyperref}
\else
\usepackage[dvips,plainpages=false,hypertexnames=false]{hyperref}
\fi

\usepackage{tikz}

\newcommand{\mathfig}[2]{\ensuremath{\hspace{-3pt}\begin{array}{c}%
  \raisebox{-2.5pt}{\includegraphics[width=#1\textwidth]{#2}}%
\end{array}\hspace{-3pt}}}

\swapnumbers

\newtheorem{theorem}{Theorem}[subsection]

\newtheorem{df}[theorem]{Definition}
\newtheorem{lemma}[theorem]{Lemma}
\newtheorem{proposition}[theorem]{Proposition}
\newtheorem{corr}[theorem]{Corollary}
\newtheorem{example}[theorem]{Example}
\newtheorem{remark}[theorem]{Remark}

\def\Ss{R}

\def\Q{\mathbf{Q}}
\def\Z{\mathbf{Z}}
\def\Gal{\mathrm{Gal}}
\def\Qbar{\overline{\Q}}
\def\R{\mathbf{R}}
\def\C{\mathbf{C}}
\def\M{\mathscr{M}}
\def\LL{\mathscr{L}}
\def\Num{\mathscr{N}}
\def\epl{\frac{\sqrt{3}+\sqrt{7}}{2}}
\def\ep{\frac{1}{2} (\sqrt{3} + \sqrt{7})}

\def\Cb{\Phi}
\def\mut{\rho}

\newcommand{\Cat}{\mathcal C}
\newcommand{\A}{\mathcal A}
\newcommand{\B}{\mathcal B}
\newcommand{\E}{\mathcal E}
\newcommand{\Center}{\mathcal Z}
\newcommand{\cO}{\mathcal O}
\newcommand{\V}{\mathcal V}
\newcommand{\IX}{{\mathcal I}{\mathcal X}}
\newcommand{\BC}{\mathbf C}
\newcommand{\BZ}{\mathbf Z}
\newcommand{\BQ}{\mathbf Q}
\newcommand{\ot}{\otimes}
\newcommand{\FP}{\mbox{FP}}
\newcommand{\Rep}{\mbox{Rep}}
\newcommand{\Hom}{\mbox{Hom}}
\newcommand{\id}{\mbox{id}}
\newcommand{\Ve}{\mbox{Vec}}
\newcommand{\be}{{\bf{1}}}
\newcommand{\bX}{{\bf{X}}}
\newcommand{\bV}{{\bf{V}}}
\newcommand{\bM}{{\bf{M}}}
\newcommand{\bg}{{\bf{g}}}
\newcommand{\g}{\mathfrak g}

\def\tareesidedbox#1{\setbox0=\hbox{$#1$}\dimen0=\wd0 \advance\dimen0 by3pt\rlap{\hbox{\vrule height9pt width.4pt
 depth2pt \kern-.4pt\vrule height9.4pt width\dimen0 depth-9pt\kern-.4pt \vrule height9pt width.4pt depth2pt}}
 \relax \hbox to\dimen0{\hss$#1$\hss}}
\def\ho#1{\tareesidedbox#1}

\def\tareesidedboy#1{\setbox0=\hbox{$#1$}\dimen0=\wd0 \advance\dimen0 by3pt\rlap{\hbox{\vrule height9.8pt width.4pt
 depth2pt \kern-.4pt\vrule height10pt width\dimen0 depth-9.6pt\kern-.4pt \vrule height9.8pt width.4pt depth2pt}}
 \relax \hbox to\dimen0{\hss$#1$\hss}}
\def\hr#1{\tareesidedboy#1}

\begin{document}

\title{cyclotomic integers, fusion categories, and
subfactors.}
\author{Frank Calegari \and Kim Morrison \and Noah Snyder}

\begin{abstract}
Dimensions of objects in fusion categories are cyclotomic integers, hence number theoretic results have implications in the study of fusion categories and finite depth subfactors.   We give two such applications.  The first application is determining a complete list of numbers in the interval $(2, 76/33)$ which can occur as the Frobenius-Perron dimension of an object in a fusion category.  The smallest number on this list is realized in a new fusion category which is constructed in the appendix written by V. Ostrik, while the others are all realized by known examples.  The second application proves that in any family of graphs obtained by adding a $2$-valent tree to a fixed graph, either only finitely many graphs are principal graphs of subfactors or the family consists of the $A_n$ or $D_n$ Dynkin diagrams.  This result is effective, and we apply it to several families arising in the classification of subfactors of index less then $5$.
\end{abstract}

\maketitle


\section{Introduction}

Let $C$ be a fusion category and $f$  any ring map from the Grothendieck ring
$K(C)$ to $\C$. If $X$ is an object in $C$, then Etingof--Nikshych--Ostrik proved in \cite{MR2183279} that $f([X])$ is a cyclotomic integer.  This result allows for applications of algebraic number theory to fusion categories and subfactors.  
The first such application was given by Asaeda and Yasuda \cite{MR2307421, MR2472028} who excluded a certain infinite family of graphs as possible principal graphs of subfactors.  We prove two main results, one a classification of small Frobenius--Perron dimensions of objects in fusion categories, and the other a generalization of Asaeda--Yasuda's result to arbitrary families of the same form.

\begin{theorem} Let $X$ be an object in a fusion category \label{theorem:fusion}
whose Frobenius--Perron dimension satisfying $2 < \FP(X) \leq 76/33 = 2.303030\ldots$ 
then $\FP(X)$ is equal to one of the following algebraic integers:
$$\begin{aligned}
\frac{\sqrt{7} + \sqrt{3}}{2} = & \ 2.188901059\ldots \\
\sqrt{5} = &  \ 2.236067977\ldots \\
1 + 2 \cos(2 \pi/7) = & \ 2.246979603\ldots \\
\frac{1 + \sqrt{5}}{\sqrt{2}} = 2 \cos(\pi/20) + 2 \cos(9 \pi/20) = & \
2.288245611\ldots \\
\frac{1 + \sqrt{13}}{2} = & \ 2.302775637\ldots \end{aligned}$$
 \end{theorem}
 
\begin{remark}
\emph{Each of the numbers in Theorem~\ref{theorem:fusion} can be realized as the Frobenius--Perron dimension of an object in a fusion category.  See \S \ref{subsec:realizing} and Appendix~\ref{appendix} (written by Ostrik).}
\end{remark}

\begin{theorem} \label{theorem:graphs}
Let $\Gamma$ be a connected
graph with $|\Gamma|$ vertices. Fix a vertex $v$ of $\Gamma$, and
let $\Gamma_n$ denote the  sequence of graphs obtained by adding a $2$-valent tree
of length $n - |\Gamma|$ to $\Gamma$ at $v$ (see Figure \ref{fig:gamma-family}).
For any fixed $\Gamma$, there exists an effective constant $N$ such that for all $n \ge N$, either:
\begin{enumerate}
\item  $\Gamma_n$ is the  Dynkin diagram $A_n$ or $D_n$.
\item  $\Gamma_n$ is not the principal graph of a subfactor.
\end{enumerate}
\end{theorem}
\begin{figure}[!ht]
$$\mathfig{0.4}{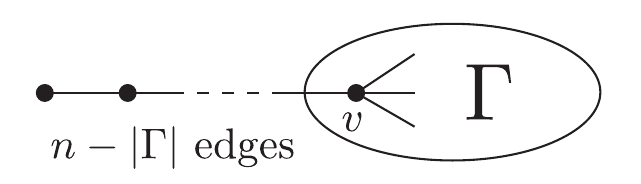}$$
\caption{The family of graphs $\Gamma_n$.}
\label{fig:gamma-family}
\end{figure}

\begin{remark} \emph{The main theorem of Asaeda--Yasuda \cite{MR2472028} is the particular case where
$\Gamma$ is the Dynkin diagram $A_7$ and $v$ is the central vertex.  See Example \ref{ex:threecases} to see our results applied to this case and two others arising the classification of subfactors of small index.}
\end{remark}

Perhaps surprisingly, both Theorem~\ref{theorem:fusion} and Theorem~\ref{theorem:graphs} can  be deduced purely from arithmetic considerations.  

The first main result follows immediately from the following theorem.

\begin{theorem} Let $\beta \in \Q(\zeta)$ be a real \label{theorem:main}
algebraic integer in some 
cyclotomic extension of the rationals. Let $\ho{\beta}$ denote the largest absolute value
of all conjugates of $\beta$. If $\ho{\beta} \le 2$ then 
$\ho{\beta} = 2 \cos(\pi/n)$  for some integer $n$. If $2 < 
\ho{\beta} < 76/33$, then
$\ho{\beta}$ is one of the five numbers occurring
in Theorem~\ref{theorem:fusion}.
\end{theorem}

The second main result is a consequence of the following theorem, combined with the fact that the even part of a finite depth subfactor is a fusion category.

\begin{theorem} For any $\Gamma$, there exists an effective constant
$N$ such that for all $n \ge N$, either: \label{theorem:galoisintro}
\begin{enumerate}
\item  All the eigenvalues of the adjacency matrix $M_n$ are of the form
$\zeta + \zeta^{-1}$ for some root of unity $\zeta$, and the graphs $\Gamma_n$
are the  Dynkin diagrams $A_n$ or $D_n$.
\item The largest eigenvalue $\lambda$ is greater than $2$,
and the field $\Q(\lambda^2)$ is not abelian.
\end{enumerate}
\end{theorem}

Although Theorem \ref{theorem:galoisintro} is, in principle,  effective, it is difficult to apply in practice.  We also give a logically weaker but more effective version of Theorem \ref{theorem:galoisintro} which is sufficient to prove Theorem \ref{theorem:graphs} and is practical for many examples.

We briefly summarize the main ideas in the proofs of these arithmetic theorems.  A key idea of Cassels~\cite{MR0246852} is to study elements with small normalized trace $\M(\beta) = \frac{1}{\deg \beta} \mathrm{Tr}(\beta \cdot \overline{\beta}) \in \Q$ rather than work directly with bounds on $\ho{\beta}$.  A key principle, made rigorous by Loxton \cite{MR0309896}, says that if $\beta$ is a cyclotomic integer and $\M(\beta)$ is small, then $\beta$ can be written as a sum of a small number of roots of unity.  This principle was first applied by Cassels to study cyclotomic integers of small norm~\cite{MR0246852}.
 In fact, Theorem~\ref{theorem:main} (at least for $\ho{\beta} \le \sqrt{5}$) is a consequence of the main theorem
of Cassels \emph{with finitely many exceptions}.

A careful study of Cassels' analysis shows that any exceptions must lie in the field $\Q(\zeta_N)$ with
 $$N = 4692838820715366441120 = 2^5 \cdot 3^3  \cdot 5 \cdot 7 \cdot 11 \cdot 13 \cdot
 17 \cdot 19 \cdot 23 \cdot 29 \cdot 31 \cdot 37 \cdot 41 \cdot 47
 \cdot 53.$$
Given that the  problem of finding small vectors inside a lattice (say, of algebraic integers) is NP-complete, this is not immediately useful.  We improve on Cassels argument in three main ways.  First, we show that $\ho{\beta}<76/33$ implies that $\M(\beta)<23/6$ (which improves substantially on the obvious bound of $(76/33)^2$).  Second, we systematically exploit the condition that $\beta$ is real (an assumption that Cassels did not make).  In particular, we adapt techniques of A.~J.~Jones~\cite{MR0224587} and Conway and A.~J.~Jones~\cite{MR0422149} for classifying small sums of three roots of unity to understand {\em real} sums of five roots of unity.  Finally, we engage in a detailed case-by-case analysis to complete the argument and remove all exceptions.

We now sketch the ideas of the proof of Theorem \ref{theorem:galoisintro}.  Let $\lambda_n$ be the  Frobenius-Perron eigenvalue of the graph $\Gamma_n$.  The average $\frac{1}{n} \sum_\mu |\mu^2-2 |^2$  over all eigenvalues $\mu$ of the adjacency matrix of $\Gamma_n$ can
be shown to converge to $2$ as $n$ increases without bound.  Since all Galois conjugates of $\lambda_n$ are  eigenvalues of the adjacency matrix, this suggests that $\M(\lambda_n^2-2)$ should also be small.  By the Cassels-Loxton principle, 
if $\lambda^2_n$ is cyclotomic, one  would expect
that $\lambda_n^2-2$ should be a sum of a small number of roots of unity. 
Explicitly, we deduce for all $n$ greater than some explicit bound (depending on
$\Gamma$)  that
either $\lambda^2_n$ is \emph{not} cyclotomic or $\lambda^2_n - 2$ is the sum
of at most two roots of unity. The latter case can only occur if $|\lambda_n| \le 2$, in which
case the characteristic polynomial of $\Gamma_n$ is  a Chebyshev polynomial, and
$\Gamma_n$ is necessarily an extended Dynkin diagram.
  In order to make this argument rigorous, one needs to understand the relationship between all eigenvalues and the subset of eigenvalues conjugate to $\lambda_n$.  We do this in two different ways.  First, we use the result of Etingof-Nikshych-Ostrik to show that all non-repeating eigenvalues are cyclotomic integers.  In light of this result, we need only control the repeated eigenvalues and the eigenvalues of the form $\zeta+\zeta^{-1}$ for roots of unity $\zeta$.  This can be done using techniques of Gross-Hironaka-McMullen~\cite{MR2516970}.  To finish the argument, we use a much easier version of Theorem~\ref{theorem:main} to get a contradiction.  For the second proof, we use some height inequalities to show that the degree of $\lambda_n$ grows linearly in $n$.  Again this is enough to get bound on $\M(\lambda_n^2-2)$, as well
  as bounds on $\M(P(\lambda^2_n))$ for  other polynomials in $\lambda_n^2$.   The desired contradiction
  then follows from  Loxton's result applied to a particular polynomial in $\lambda_n^2$.
  
  \begin{remark} \emph{The methods used in our proof of Theorem \ref{theorem:fusion} can certainly be extended further
than $76/33$, at the cost of a certain amount of combinatorial explosion. However, there
do exist  limit points of the set of possible $\ho{\beta}$,
including at  $2 \sqrt{2} = \displaystyle{\lim_{\longrightarrow} 2 \sqrt{2} \cos(\pi/n)}$ and
$3 =  \displaystyle{\lim_{\longrightarrow}  1 + 2 \cos(2 \pi/n)}$.  The best general ``sparseness" result we have is Theorem \ref{theorem:closedsubset} which states that the set of values of $\M(\beta)$ for $\beta$ a cyclotomic integer is a closed subset of $\Q$.}
\end{remark}

Theorem~\ref{theorem:fusion} is similar in spirit to Haagerup's classification of all subfactors of index less than $3+\sqrt{3} = 4.73205...$ \cite{MR1317352}.  In fact, a version of Theorem~\ref{theorem:fusion} follows from Haagerup's classification, for example, ``if $X$ is an object in a \emph{unitary} tensor category with duals then the dimension of $X$ does not lie in the interval 
$(2,\sqrt{\frac{5+\sqrt{13}}{2}}) = (2,2.074313\ldots)$."  
Our result is weaker in that we assume finiteness, but stronger in that it does not assume unitarity and applies to larger dimensions.  

In the other direction, one might wonder if purely arithmetic considerations have implications
for  finite depth subfactors of small index larger than $4$. Indeed, using only arithmetic we can prove the following result.

\begin{theorem} Suppose that $\displaystyle{4 < \alpha < 4 + 10/33 = 4.303030\ldots}$ 
is the index of a finite \label{theorem:submain}
depth subfactor. Then either $\alpha = 3 + 2 \cos(2 \pi/7)$, or
$\displaystyle{\alpha = \frac{5 + \sqrt{13}}{2}}$.
\end{theorem}

  \subsection{Detailed summary}
The proof of Theorem \ref{theorem:main} proceeds in several steps. We first prove the theorem
for those $\beta$ which can be written as the sum of at most $5$ roots of unity
(Theorem~\ref{theorem:use}).
This argument requires some preliminary analysis of vanishing sums of roots
of unity, which we undertake in \S\ref{section:vanishingsums}.
Having done this, we prove
Theorem~\ref{theorem:firstbound}, which shows that any
exception to Theorem~\ref{theorem:main} lies in $\Q(\zeta_N)$ with
$N = 420$. 
A  useful technical tool is provided by
Lemma~\ref{lemma:bound}, which allows us to reduce our search to  $\beta$ satisfying
$\M(\beta) < 23/6$ rather than $\M(\beta) < 5$ as in Cassels.
In Lemma~\ref{lemma:bound84} and Corollary~\ref{corr:cor84}, we prove
Theorem~\ref{theorem:main} for $\beta \in \Q(\zeta_{84})$.
In \S\ref{section:finalreduction} we make the final  step
of showing that any counterexample $\beta \in \Q(\zeta_{420})$ must actually
lie in $\Q(\zeta_{84})$. There is a certain amount of combinatorial
explosion in this section which we control as much as possible with
various tricks. Although our paper is written to be independent,
it would probably be useful to the reader to consult A.~J.~Jones~\cite{MR0224587}
when reading \S\ref{section:jones}, and Cassels~\cite{MR0246852} when reading
\S\S\ref{section:cassels1}--\ref{section:cassels4}.

In \S\ref{section:discrete}, we prove an easier version of Theorem \ref{theorem:main} which will be used to prove the effective version of Theorem \ref{theorem:graphs}.  In this section, we also prove that the values of $\M(\beta)$ for $\beta$ a cyclotomic integer are a closed subset of $\Q$.   We then prove an effective version of Theorem \ref{theorem:graphs} in \S\ref{section:graphs} and give applications to several families which appear in the classification of small index subfactors.  In \S\ref{section:graphs2}, we prove Theorem \ref{theorem:galoisintro} which is logically stronger but less effective than the result in the previous section.  A reader mainly interested in the applications to subfactors may wish to skip directly to \S\ref{section:graphs}~\&~\S\ref{section:graphs2}.

\subsection{Acknowledgements}
We would like to thank MathOverflow where this collaboration began (see ``Number theoretic spectral properties of random graphs"  \url{http://mathoverflow.net/questions/5994/}).  We would also like to thank Feng Xu for helpful conversations, and Victor Ostrik for writing the appendix. Frank Calegari was supported by NSF Career Grant DMS-0846285,  NSF Grant DMS-0701048, and a Sloan Foundation Fellowship, Kim Morrison was at the Miller Institute for Basic Research at UC Berkeley, and Noah Snyder was supported by an NSF Postdoctoral Fellowship.

\section{Definitions and preliminaries}

If $N$ is an integer, let $\zeta_{N}$ denote $\exp(2 \pi i/N)$.  Having fixed this choice for all $N$, there is  no ambiguity when writing  expressions such as $\zeta_{12} + \zeta_{20}$ --- \emph{a priori}, such an expression is not even well
defined up to conjugation.

Suppose that $\Q(\beta)$ is an abelian extension.
 By the Kronecker--Weber theorem,
$\beta$ is contained inside some minimal cyclotomic field $\Q(\zeta_N)$.
($N$ is the \emph{conductor} of $\Q(\beta)$.)
If  $\beta \in \Q(\zeta_N)$ is an algebraic integer, we shall consider several
invariants attached to  $\beta$:

\begin{df} For a cyclotomic integer $\beta$, we denote by $\Num(\beta)$ the size of the smallest set $S$ such that
$\beta = \sum_{S} \xi_i$ for $\xi_i$ a root of unity.
\end{df}

\begin{df} If $\beta$ is any algebraic integer, we let $\ho{\beta}$ denote the maximum of the absolute values of all the conjugates
of $\beta$, and let $\M(\beta)$ denote the
average value of the real numbers $|\sigma \beta|^2$ where $\sigma \beta$ runs
over all conjugates of $\beta$.
\end{df}

\begin{remark} \emph{If $\beta \in K$ where $K$ is Galois and $G = \Gal(K/\Q)$, then
$\M(\beta)$ is well behaved whenever complex conjugation is central in $G$,
since then $|\sigma \beta|^2 = \sigma |\beta|^2$, and 
$[K:\Q] \M(\beta) = \mathrm{Tr}(|\beta^2|)$. This is the case, for example,
whenever $K$ is totally real or abelian. In particular, in these cases, 
$\M(\beta) \in \Q$.}
\end{remark}

There are inequalities $\Num(\beta) \ge \ho{\beta}$, which follows from
the triangle inequality, and
$\ho{\beta}^{\kern+0.1em{2}}  \ge \M(\beta) 
\ge |N_{K/\Q}(\beta)|^{1/[K:\Q]}$, which is $\ge 1$ if $\beta$ is non-zero.
Note that $\ho{{\alpha + \beta}} \ne \ho{\alpha} + \ho{\beta}$ in general.

\begin{example} \emph{
Suppose that $\beta$ is a totally real algebraic integer and \label{example:cos}
that $\ho{\beta} \le 2$. If $\alpha + \alpha^{-1} = \beta$, then all the conjugates
of $\alpha$ have absolute value $1$. A theorem of Kronecker~\cite{an:053.1389cj} implies that
$\alpha$ is a root of unity, and then an easy computation shows that
$\ho{\beta} = 2 \cos(\pi/n)$ for some integer $n$.
}
\end{example}

This example shows that the values $\ho{\beta}$ are discrete in $[0,\theta]$
for any $\theta < 2$. On the other hand, it follows from 
Theorem~1 of~\cite{MR736460} that the values of $\ho{\beta}$ for totally real
algebraic integers $\beta$ are  dense in $[2,\infty)$.
Thus, the discreteness implicit in Theorem~\ref{theorem:main} reflects a special property of cyclotomic integers. It also follows from Theorem~1 of~\cite{MR736460} that
the values $\M(\beta)$ (for totally real $\beta$) are dense in $[2,\infty)$. On the other hand,  a classical theorem of Siegel~\cite{MR0012092}  says that   $\M(\beta) \ge 3/2$ for any totally
real algebraic integer $\beta$ of degree $\ge 2$, the minimum value occurring for $\beta =
\displaystyle{\frac{1+\sqrt{5}}{2}}$, and,  furthermore, the values of $\M(\beta)$
are discrete in $[0,\theta]$ for any $\theta < \lambda = 1.733610\ldots$
In the cyclotomic case, we once more see a limit point of $\M(\beta)$ at $2$ followed
by a region beyond $2$ where $\M(\beta)$ is discrete (Theorem~\ref{theorem:lazy}).
Moreover, the closure of $\M(\beta)$ on $[0,\infty)$ is, in fact, a closed
subset of $\Q$ (Theorem~\ref{theorem:closure}).

\section{Background on fusion categories and subfactors}

In this section, we rapidly review some notions about fusion categories and subfactors,
and collect a few remarks and examples.  Although the applications of our main results are to fusion categories and subfactors, their proofs are purely arithmetic and can be read independently from this section.

A \emph{fusion category} $C$ over a field $k$ is an abelian, $k$-linear, semisimple, rigid, monoidal category with finitely many isomorphism classes of simple objects.  In this paper, all fusion categories are over the complex numbers.

A \emph{subfactor} is an inclusion $A<B$ of von Neumann algebras with trivial centers.  We will only consider subfactors in this paper which are irreducible ($B$ is an irreducible $A$-$B$ bimodule) and type $I\kern-0.12em{I_1}$ (there exists a unique normalized trace).  A subfactor is called \emph{finite depth} if only finitely many isomorphism classes of simple bimodules appear as summands of tensor powers of ${}_A B_A$.  In particular, to every finite depth subfactor there is an associated fusion category $C$, called the \emph{principal even part} which is the full subcategory of the category of $A$-$A$ bimodules whose objects are summands of tensor powers of ${}_A B_A$.

The \emph{principal graph} of a subfactor is a bipartite graph whose even vertices are the simple $A$-$A$ bimodules which occur as summands of tensor powers of ${}_A B_A$, whose odd vertices are the simple $A$-$B$ bimodules which occur as summands of tensor powers of ${}_A B_A$ tensored with ${}_A B_B$, and where $X$ and $Y$ are connected by $\dim (X \otimes {}_A B_B, Y)$ edges.

\begin{remark}
\emph{Usually included in the data of a principal graph is the choice of a fixed leaf which corresponds to the monoidal unit ${}_A A_A$.   All the techniques in our paper  which eliminate a graph $\Gamma$ as a possible principal graph eliminate the graph for any choice of leaf.  Nonetheless, techniques in other papers  often depend on the choice of fixed leaf. }

\emph{In particular, the families in Haagerup's list of potential principal graphs of small index \cite{MR1317352} have modularity restrictions on the length of the degree $2$ tree which depend on the choice of leaf.  Strictly speaking, our main result when applied to $\Gamma = A_7$ with $v$ the middle vertex is stronger than the result in \cite{MR2472028} where they only check noncyclotomicity after assuming Haagerup's modularity conditions.  Nonetheless, we will often elide this issue, and when we say a paper eliminated a family of potential principal graphs we will mean that they eliminated the principal graphs in that family which had not already been eliminated by Haagerup.}
\end{remark}

A \emph{dimension function} on a fusion category $C$ is a ring map $f: K(C)\rightarrow \C$ where $K(C)$ is the Grothendieck group thought of as a ring with the product induced by the tensor product.  We often abuse notation by applying $f$ directly to objects in $C$.  There exists a unique dimension function $\FP$ called the \emph{Frobenius--Perron dimension} which assigns a \emph{positive} real number to each simple object \cite[\S 8]{MR2183279}.  The Frobenius--Perron dimension of $X \in C$ is given by the unique largest eigenvalue of left multiplication by $[X]$ in $K(C) \otimes \C$.  The Frobenius--Perron dimension of ${}_A B_A$ in $C$ is the  index of $A<B$ which is denoted $[B:A]$.  The index of $A<B$ is the square of the largest eigenvalue of the adjacency matrix of the principal graph.

\medskip

For the applications in our paper,  we need the following strong arithmetic condition on dimensions.

\begin{theorem}\cite[Corollary 8.53]{MR2183279}
If $C$ is a fusion category, \label{theorem:ENO}
$X$ is an object in $C$, and $f$ is a dimension function, then $\Q(f(X))$ is abelian.
\end{theorem}

We will also want a version of this result that more easily applies to principal graphs:

\begin{lemma}
If $\Gamma$ is the principal graph of a finite depth subfactor $A<B$  \label{lemma:graphcyclo}
and $\lambda$ is an eigenvalue of $M(\Gamma)$ of multiplicity one, then $\Q(\lambda^2)$ is abelian.
\end{lemma}
\begin{proof}
Let $C$ be the fusion category which is the principal even part of the subfactor.  Let $X$ be the object ${}_A B_A$ inside $C$.  From the definition of the principal graph it follows that  $\lambda^2$ is a multiplicity $1$ eigenvalue for left multiplication by $[X]$ in the base extended Grothendieck group $K(C) \otimes \mathbf{C}$.  Decompose $K(C) \otimes \mathbf{C}$ as a product of matrix algebras $\prod \mathrm{End}(V_i)$.  An element of $\mathrm{End}(V_i)$ can be thought of as acting by left multiplication on itself or as acting on $V_i$. The eigenvalues of the former action are exactly the eigenvalues of the latter action but each repeated $\dim V_i$ times.  In particular, if $x$ is an element of a multi-matrix algebra then any multiplicity one eigenvalue of $x$ acting on the algebra by left multiplication must be a component of $x$ in one of the $1$-dimensional matrix summands.  In particular, we see that there is a map of rings $f: K(C) \otimes \mathbf{C}\rightarrow \mathbf{C}$ such that $\lambda^2 = f(X)$.  Our result now follows immediately from Theorem~\ref{theorem:ENO}.
\end{proof}

The  following well-known arithmetic arguments proving two versions of the V.~Jones index restriction  \cite{MR696688} are baby examples of the main idea of this paper:

\begin{lemma}
If $X$ is an object in a fusion category with $\FP(X) \le  2$ then $\FP(X) = 2 \cos(\pi/n)$ for some integer \label{lemma:joe1} $n$.
\end{lemma}

\begin{lemma}
If $A<B$ is a finite depth \label{lemma:joe2} subfactor with index $[B:A] \le 4$, then
$[B:A] = 4 \cos(\pi/n)^2 = 2+2 \cos(2 \pi/n)$.
\end{lemma}

\begin{proof}[Proofs] In light of Theorem~\ref{theorem:ENO},
Lemma~\ref{lemma:joe1} follows directly from Example~\ref{example:cos}.
In light of Lemma~\ref{lemma:graphcyclo},
Lemma~\ref{lemma:joe2} follows either from applying Example~\ref{example:cos}
to $\lambda$, where $\lambda^2 = [B:A]$, or to $\lambda^2 - 2$.
\end{proof}

\begin{remark}
\emph{This is weaker than the V.~Jones index restriction since we are
making a finite depth assumption.  Indeed, all our results in this paper about
subfactors and monoidal categories depend crucially on finiteness assumptions.}
\end{remark}

\subsection{Realizing the possible dimensions} \label{subsec:realizing}
As mentioned in the introduction, each of the numbers in Theorem \ref{theorem:main} can in fact be realized as the dimension of an object in a fusion category.  Nonetheless, we do not 
necessarily expect that every number of the form $\ho{x}$ for $x$ a real cyclotomic integer can be realized as a dimension of an object in a fusion category. We quickly summarize how each of these numbers can be realized.  The dimension $(\sqrt{3}+\sqrt{7})/2$ occurs in a fusion category constructed by Ostrik in the appendix based on an unpublished construction via a conformal inclusion (due to Xu \cite{Xu}) of a subfactor originally constructed by Izumi \cite{MR1832764}.  The dimension $\sqrt{5}$ can be achieved by a Tambara--Yamigami category associated to $\mathbb{Z}/5\mathbb{Z}$ \cite{MR1659954}.  The dimension $1 + 2 \cos(2 \pi/7)$ occurs as a dimension of an object in quantum $\text{SU}(2)$ at a $14$th root of unity.  The dimension $(1 + \sqrt{5})/\sqrt{2}$ occurs in the Deligne tensor product of quantum $\text{SU}(2)$ at a $10$th root of unity and quantum $\text{SU}(2)$ at an $8$th root of unity.  Finally, $(1 + \sqrt{13})/2$ occurs as the dimension of an object in the dual even part of the Haagerup subfactor \cite{MR1686551}.

\subsection{Deduction of Theorem~\ref{theorem:submain} from 
Lemma~\ref{lemma:graphcyclo}} 
Suppose that $\alpha$
 is the index of a finite 
depth subfactor and $\displaystyle{4 < \alpha < 4 + 10/33 = 4.303030\ldots}$
 Then $\alpha$ is a cyclotomic integer by Lemma~\ref{lemma:graphcyclo}, and $\alpha = \lambda^2$ for a totally real algebraic integer $\lambda$
which is the Perron--Frobenius eigenvalue of the principal graph.
Thus $$-2 \le (\sigma \lambda)^2 - 2 \le 76/33$$
for every conjugate $\sigma \lambda$ of $\lambda$. In particular, if $\beta = \alpha - 2$,
then $2 < \ho{\beta} < 76/33$,
and by Theorem~\ref{theorem:main}, we deduce that $\beta$ is one of the five numbers
occurring in Theorem~\ref{theorem:fusion}. 
On the other hand, for three of these five numbers $\beta$ has a conjugate smaller than $ - 2$, and hence the corresponding field
$\Q(\lambda)$  is not totally real.
Thus, either $\alpha = 3 + 2 \cos(2 \pi/7)$ or
$\displaystyle{\alpha = \frac{5 + \sqrt{13}}{2}}$. \qed

\section{The case when $\beta$ is a sum
of at most $5$ roots of unity} \label{section:vanishingsums}

The goal of this section is to prove Theorem~\ref{theorem:main}
in the case that $\Num(\beta) \le 5$ (see Theorem~\ref{theorem:useful1}).
The outline of this argument is that we first use the Conway--A.~J.~Jones classification of small vanishing sums of roots of unity in order to show that, outside a few exceptional cases, any real sum of five roots of unity is of the obvious form (with pairs of complex conjugate terms).  We then make a more in depth analysis of small sums of the form $\zeta_N^a + \zeta_N^{-a}+\zeta_N^b+\zeta_N^{-b}$.

\subsection{Vanishing Sums}

Consider a vanishing sum:
$$\sum_{S} \xi_i = 0,$$
where the $\xi_i$ are roots of unity.
Such a sum is called \emph{primitive} if no proper subsum vanishes.  We say that such a sum has $|S|$ terms.
We may normalize any such sum up to a finite ambiguity by insisting that one of the summands be $1$.
 
\begin{theorem}[Conway--A.~J.~Jones~\cite{MR0422149}] For every $|S|$, there are only finitely many primitive
normalized vanishing sums
$\displaystyle{\sum_{i \in S} \xi_i= 0}$.
\end{theorem}

The Conway and A.~J.~Jones result is more precise, in that they give explicit bounds on the conductor of the cyclotomic field generated by the $\xi_i$ in a vanishing sum with a fixed number of terms. For our purposes, it will be useful to have a more explicit description of the primitive normalized
vanishing sums for small $|S|$. The following result is a small extension of Theorem~6 of~\cite{MR0422149} which can be found in Table $1$ of ~\cite{MR1612877}.   

\begin{theorem}[Conway--A.~J.~Jones, Poonen--Rubinstein] The primitive vanishing sums with $|S|$ even and $|S| \le 10$ are as follows: \label{theorem:poo}
\begin{itemize}
\item $|S| = 2$:
$$1 + (-1) = 0.$$
\item $|S| = 6$:
$$ \zeta_6 + \zeta^5_6 + \zeta_5 + \zeta^2_5 + \zeta^3_5 + \zeta^4_5  = 0.$$
\item $|S| = 8$:
$$\zeta_6 + \zeta^5_6 + \zeta_7 + \zeta^2_7 + \zeta^3_7 + \zeta^4_7 + \zeta^5_7 + \zeta^6_7 = 0.$$
$$\zeta_6 + \zeta^5_6 + \zeta^4_{30} + \zeta^{10}_{30} + \zeta^{11}_{30} + \zeta^{17}_{30} +
\zeta^{23}_{30} + \zeta^{24}_{30} = 0.$$
$$\zeta_6 + \zeta^5_6 + \zeta_{30} + \zeta^{2}_{30} + \zeta^{12}_{30} + \zeta^{13}_{30} +
\zeta^{19}_{30} + \zeta^{20}_{30} = 0.$$
\item $|S|=10$:
$$ \zeta_7 + \zeta^2_7 + \zeta^3_7 + \zeta^4_7 + \zeta^5_7 + \zeta^6_7 +\zeta_{10}+\zeta_{10}^3+\zeta_{10}^7+\zeta_{10}^9 = 0.$$
\begin{align*}
1 + \zeta_{3} + \zeta_{7} + \zeta_{7}^{2} + \zeta_{21}^{10} + \zeta_{21}^{13} + \zeta_{42} + \zeta_{42}^{25} + \zeta_{42}^{31} + \zeta_{42}^{37} & = 0. \\
1 + \zeta_{3} + \zeta_{7} + \zeta_{7}^{3} + \zeta_{21}^{10} + \zeta_{21}^{16} + \zeta_{42} + \zeta_{42}^{19} + \zeta_{42}^{31} + \zeta_{42}^{37} & = 0. \\
1 + \zeta_{3} + \zeta_{7} + \zeta_{7}^{4} + \zeta_{21}^{10} + \zeta_{21}^{19} + \zeta_{42} + \zeta_{42}^{19} + \zeta_{42}^{25} + \zeta_{42}^{37} & = 0. \\
1 + \zeta_{3} + \zeta_{7} + \zeta_{7}^{5} + \zeta_{21} + \zeta_{21}^{10} + \zeta_{42} + \zeta_{42}^{19} + \zeta_{42}^{25} + \zeta_{42}^{31} & = 0. \\
1 + \zeta_{3} + \zeta_{7}^{2} + \zeta_{7}^{4} + \zeta_{21}^{13} + \zeta_{21}^{19} + \zeta_{42} + \zeta_{42}^{13} + \zeta_{42}^{25} + \zeta_{42}^{37} & = 0.
\end{align*}

\end{itemize}

In particular, there do not exist any vanishing sum with $|S| = 4$.
\end{theorem}

Note that any vanishing sum of roots of unity with $|S|$ terms decomposes as a sum of primitive
vanishing sums each with $|S_i|$ terms, where $|S| = \sum |S_i|$ is  a partition of $|S|$.

We are interested in cyclotomic integers $\beta$ that are totally real.

\begin{lemma} Suppose that $\Num(\beta) \le 5$, and that $\beta \ne 0$ is   real.
Then there exists integers $a$, $b$, and a root of unity $\zeta$ such that, up to sign, one of the following holds: \label{lemma:84}
\begin{enumerate}
\item $\Num(\beta) = 1$, and $\beta =  1$,
\item $\Num(\beta) = 2$ and $\beta = \zeta^a + \zeta^{-a}$,
\item $\Num(\beta) = 3$, and $\beta = \zeta^a + \zeta^{-a} + 1$,
\item $\Num(\beta) = 4$, and $\beta = \zeta^a + \zeta^{-a} + \zeta^{b} + \zeta^{-b}$,
\item $\Num(\beta) = 5$, and $\beta =  \zeta^a + \zeta^{-a} + \zeta^{b} + \zeta^{-b} + 1$,
\item $\Num(\beta) = 3$, and $\beta$ is Galois conjugate to $\zeta_{12} + \zeta_{20} + \zeta_{20}^{17}.$
\item $\Num(\beta) = 4$, and $\beta$ is Galois conjugate to one of 
\begin{enumerate}
\item $\zeta_{84}^{-9} + \zeta_{84}^{-7} + \zeta_{84}^{3} + \zeta_{84}^{15},$
\item $\zeta_{84}^{-9} + \zeta_{84}^{-7} + \zeta_{84}^{3} + \zeta_{84}^{27},$
\item $1+\zeta_{12} + \zeta_{20} + \zeta_{20}^{17}.$
\end{enumerate}
\item $\Num(\beta) = 5$, and $\beta$ is Galois conjugate to one of 
\begin{enumerate}
\item $\zeta_{12} + \zeta_{20} + \zeta_{20}^{17} + \zeta^a +\zeta^{-a}$ for some root of unity $\zeta$
\item $1+\zeta_{84}^{-9} + \zeta_{84}^{-7} + \zeta_{84}^{3} + \zeta_{84}^{15}$
\item $1+\zeta_{84}^{-9} + \zeta_{84}^{-7} + \zeta_{84}^{3} + \zeta_{84}^{27}$
\item $\zeta_{84}^{-9} + \zeta_{84}^{-7} + \zeta_{84} + \zeta_{84}^3 + \zeta_{84}^{13}$
\item $\zeta_{84}^{-9} + \zeta_{84}^{-7} + \zeta_{84}^{15} + \zeta_{84}^{25} + \zeta_{84}^{73}$
\end{enumerate}
\end{enumerate}
 \end{lemma}

\begin{proof} Let $I$ denote a set of size $\Num(\beta)$ such that
$\beta = \sum_{I} \xi_i$. 
Note that $-1$ is a root of unity.
If $\beta$ is real, then we have a vanishing sum
$$\beta - \overline{\beta} = \sum_{I} \xi_i  + \sum_{I} - \xi^{-1}_i = 0$$
with $2 \Num(\beta) \le 10$ terms. This sum can be decomposed into primitive sums whose number of terms sum
to $2 \Num(\beta)$. Write such a primitive vanishing sum as
$$\sum_{A} \xi_i  + \sum_{B} - \xi^{-1}_i = 0,$$
where $A$ and $B$ are disjoint subsets of $I$. Suppose that $|A|+|B|$ is odd. 
Since the sum is invariant under complex
conjugation, we may assume that $|A| > |B|$. It follows that
$$\beta = \sum_{I} \xi_i = \sum_{I \setminus A} \xi_i + \sum_{A} \xi_i = \sum_{I \setminus A} \xi_i + \sum_{B} \xi^{-1}_i,$$
and hence $\Num(\beta) \le |I| - |A| + |B| < |I|$, a contradiction.
Thus, every such vanishing subsum must have an even number of terms. 

Suppose that there is a vanishing subsum with $2$ terms.  Then we have the following options:
\begin{enumerate}
\item If $\xi_i + \xi_j = 0$, then $\displaystyle{\beta = \sum_{I - \{i,j\}} \xi_i}$, and hence $\Num(\beta) \le |I| - 2$, a contradiction.
\item If $\xi_{i} - \xi^{-1}_i = 0$, then $\xi_{i} = \pm 1$. Let $\gamma = \beta - \xi_{i}$. Then $\gamma$
  real and satisfies $\Num(\gamma) = \Num(\beta) - 1$.
\item If $\xi_{i} - \xi^{-1}_j = 0$,  let $\gamma = \beta - \xi_{i} - \xi^{-1}_i$. 
Then $\gamma$ is real and  $\Num(\gamma) = \Num(\beta) - 2$.
\end{enumerate}
In all these cases, the result follows by induction on $\Num(\beta)$. So we may assume that there are no vanishing subsums  with $2$ terms.

Since there exists no primitive vanishing sum with $4$ terms, and since $10 < 6 + 6$, it follows that $\sum_{I} \xi_i  + \sum_{I} - \xi^{-1}_i$ is itself primitive.

Suppose that $2|I|=6$ and our sum is proportional to a primitive vanishing sum with $6$ terms.  Hence our sum is proportional to $\zeta_6 + \zeta^5_6 + \zeta_5 + \zeta^2_5 + \zeta^3_5 + \zeta^4_5$.   By construction, there exists a decomposition of the sum $\sum_{I} \xi_i + \sum_{I} - \xi^{-1}_i$ into pairs with product $-1$.  Rescaling, we have a decomposition of the sum $\zeta_6 + \zeta^5_6 + \zeta_5 + \zeta^2_5 + \zeta^3_5 + \zeta^4_5$ into pairs with constant product.  Since the product of all of these numbers is $1$, this constant product must be a third root of unity.  But at least one pair consists only of fifth roots of unity, so the product of this pair is a fifth root of unity.  Hence the product of each pair must be $1$.  It follows that the constant of proportionality is $\zeta_4^{\pm 1}$.  Hence, up to sign and Galois conjugation, $\beta = \zeta_4 \zeta_6 + \zeta_4 \zeta_5 + \zeta_4 \zeta_5^2$, which is Galois conjugate to 
$\eta:=\zeta_{12} + \zeta_{20} + \zeta_{20}^{17}$.  The minimal polynomial of this number is $x^8-8x^6+14x^4-7x^2+1$, and its largest Galois conjugate is $2.40487\ldots$ We note in passing that
$$\eta = 2 \cos(\pi/30) + 2 \cos(13 \pi/30), \ \text{and} \ 
\eta^2 = \frac{4 + \sqrt{5} + \sqrt{15 + 6 \sqrt{5}}}{2}.$$

Suppose that $2|I|=8$ and our sum is proportional to a primitive vanishing sum with $8$ terms.  First suppose that the vanishing sum is proportional to  $\zeta_6 + \zeta^5_6 + \zeta_7 + \zeta^2_7 + \zeta^3_7 + \zeta^4_7 + \zeta^5_7 + \zeta^6_7$.  Again, we look for a way of decomposing this sum into four pairs with a fixed constant product.  Since the product of all the terms is $1$, the constant must be a fourth root of unity.  However, at least one pair has product which is a seventh root of unity.  Hence, the constant product must be $1$.  Hence, $\beta$ must be a sum of four elements, each consisting of one term from each of the pairs $(\zeta^{}_6 \kern+0.2em,\zeta^{-1}_6)$, $(\zeta^{}_7 \kern+0.2em,\zeta^{-1}_7)$, $(\zeta^2_7 \kern+0.2em,\zeta^{-2}_7)$, $(\zeta^3_7 \kern+0.2em,\zeta^{-3}_7)$ all scaled by a fixed
primitive $4$th root of unity. This leads to sixteen possibilities, which fall under two Galois orbits. One Galois orbit consists of the twelve Galois conjugates of $\zeta_{84}^{-9} + \zeta_{84}^{-7} + \zeta_{84}^{3} + \zeta_{84}^{15}$, which has minimal polynomial
$$x^{12} -15 x^{10} + 64 x^{8} - 113 x^6 + 85 x^4 - 22 x^2 + 1,$$
and largest root $\beta = 3.056668\ldots$.  The other orbit consists of the four conjugates of $\zeta_{84}^{-9} + \zeta_{84}^{-7} + \zeta_{84}^{3} + \zeta_{84}^{27}$, which
has minimal polynomial $x^4 - 5x^2 + 1$.  We have
$$\displaystyle{\zeta_{84}^{-9} + \zeta_{84}^{-7} + \zeta_{84}^{3} + \zeta_{84}^{27} = \frac{\sqrt{3} + \sqrt{7}}{2}
= 2.188901\ldots}$$

Now suppose that the vanishing sum is proportional to a sum of the form $$\zeta_5^{a_1} + \zeta_5^{a_2} -\zeta_5^{a_3}(\zeta_3+\zeta_3^2)-\zeta_5^{a_4}(\zeta_3+\zeta_3^2)-\zeta_5^{a_5}(\zeta_3+\zeta_3^2)$$ where the $a_i$ are some permutation of $\{0,1,2,3,4\}$.  This includes the last two vanishing sums with $8$ terms.  Here the product of all the terms is $\zeta_5^{a_3+a_4+a_5}$.  Rescaling the sum by a fifth root of unity, we may assume that the product of all the terms is $1$. So the product of each pair must be a fourth root of unity.  Furthermore, at least one pair consists of two $30$th roots of unity, hence the product of each pair must be a $30$th root of unity.  Hence the product of each pair must be $\pm 1$.   Since the fifth roots of unity must then pair with each other the product must be $1$.  However, at most one other pair multiplies to $1$ (if $a_i = 0$ for $i = 3,4,5$).  Hence there are no $\beta$ that yield this vanishing sum.

Suppose that $2|I|=10$ and our sum is proportional to a primitive vanishing sum with $10$ terms.  First suppose our vanishing sum is proportional to the first $10$-term vanishing sum: $\zeta_7 +\zeta_7^2+\zeta_7^3+\zeta_7^4+\zeta_7^5+\zeta_7^6 - (\zeta_5+\zeta_5^2+\zeta_5^3+\zeta_5^4)$.
The product of all of the terms in this sum is $1$, hence the product of each pair must be a $5$th root of unity.  However, at least one pair consists only of $7$th roots of unity, hence the product of each pair must be $1$.  Hence, $\beta=\zeta_4^{\pm 1} (\zeta_7^{\pm 1}+\zeta_7^{\pm 2} + \zeta_7^{\pm 3} -\zeta_5^{\pm 1} - \zeta_5^{\pm 2})$.  Up to Galois conjugation there are two numbers of this form.  Their minimal polynomials are 
{\footnotesize
$$x^{24} - 36 x^{22} + 506 x^{20} - 3713 x^{18} + 15825 x^{16} - 40916 x^{14} + 64917 x^{12} - 62642 x^{10} + 35684 x^8 - 11253 x^6 + 1717 x^4 - 90 x^2 + 1$$} and $x^8 - 12 x^6 + 34 x^4 - 23 x^2 + 1$.
 The largest roots of these are $3.7294849\ldots$ and $2.861717\ldots$ respectively.

Now suppose our vanishing sum is proportional to a sum of the form $$\zeta_7^{a_1}+\zeta_7^{a_2}+\zeta_7^{a_3}+\zeta_7^{a_4}-\zeta_7^{a_5}(\zeta_3+\zeta_3^2)-\zeta_7^{a_6}(\zeta_3+\zeta_3^2)-\zeta_7^{a_7}(\zeta_3+\zeta_3^2)$$ where the $a_i$ are a permutation of the numbers $\{0, \ldots, 6\}$.  This form includes the remaining $5$ vanishing sums.  After possibly rescaling by a $7$th root of unity, the product of all the terms is $1$, and hence the product of each pair is a $5$th root of unity.  Since the only fifth root of unity that appears as a product of two terms is $1$, the product of each pair must be $1$.  Hence, without loss of generality the pairs must be
$$\{\zeta_7^{a_1}, \zeta_7^{-a_1}\}, \{\zeta_7^{a_3}, \zeta_7^{-a_3}\},\{- \zeta_3, - \zeta_3^2\},  \{-\zeta_7^{a_6} \zeta_3, -\zeta_7^{-a_6} \zeta_3^2\},\{-\zeta_7^{-a_6} \zeta_3, -\zeta_7^{a_6} \zeta_3^2\}.$$

Thus $\beta$ is Galois conjugate to something of the form $\zeta_4(\zeta_7 + \zeta_7^x -\zeta_3^{\pm 1} -\zeta_7^y \zeta_3  - (\zeta_7^{y} \zeta_3^{2})^{\pm 1})$.  If the last sign is positive then $\beta$ can be rewritten, using $\zeta_3 +\zeta_3^2 = -1$, as a sum of $4$ terms.  Hence the last sign is negative.  Now, if the first sign is positive we can also rewrite $\beta$ as a sum of four roots of unity.  Namely, we see that
\begin{align*}
\zeta_4(\zeta_7 + \zeta_7^x -\zeta_3-\zeta_7^y \zeta_3  - \zeta_7^{-y} \zeta_3) &= -\zeta_4 \zeta_3(-\zeta_3^{-1} \zeta_7 - \zeta_3^{-1} \zeta_7^x +1+\zeta_7^y  +\zeta_7^{-y})\\
&= -\zeta_4 \zeta_3(-\zeta_7^{a}-\zeta_7^b+\zeta_3 \zeta_7 + \zeta_3 \zeta_7^x),
\end{align*}
where $a$, $b$, $x$, $\pm y$ are a permutation of $2, \ldots, 6$.  The relation that we used is $\zeta_7^{a_1} + \zeta_7^{a_2} + \zeta_7^{a_3} + \zeta_7^{a_4} + \zeta_7^{a_5} - (\zeta_3-\zeta_3^{-1})\zeta_7^{a_6} - (\zeta_3-\zeta_3^{-1})\zeta_7^{a_7}$ where the $a_i$ are a premutation of $0, \ldots, 6$.

Hence $\beta$ is Galois conjugate to $$\zeta_4(\zeta_7 + \zeta_7^x -\zeta_3^2-\zeta_7^y \zeta_3  - \zeta_7^{-y} \zeta_3)$$ where $x$ and $y$ are each one of $\{2, 3, 4, 5\}$ such that $x$ is not congruent to $\pm y$ modulo $7$.

There are two different Galois orbits of that form.
The roots of $x^{12} - 16 x^{10} + 60 x^8 - 78 x^6 + 44 x^4 - 11 x^2 + 1$, the largest of which is approximately $3.354753\ldots$; and
the roots of $x^{12} - 22 x^{10} + 85 x^8 - 113 x^6 + 64 x^4 - 15x^2 + 1$, the largest of which is approximately $4.183308\ldots$
These correspond to Galois conjugates of the roots occuring in $(8d)$ and $(8e)$
in the statement of the theorem. Note the curious identies (of sums of real numbers):
$$(\zeta_{84}^{-9} + \zeta_{84}^{-7} + \zeta_{84} + \zeta_{84}^3 + \zeta_{84}^{13})
= (\zeta_{84}^{-9} + \zeta_{84}^{-7} + \zeta_{84}^{3} + \zeta_{84}^{15}) + 
(\zeta_{84} + \zeta_{84}^{13} - \zeta_{84}^{15}),$$
$$(\zeta_{84}^{-9} + \zeta_{84}^{-7} + \zeta_{84}^{15} + \zeta_{84}^{25} + \zeta_{84}^{73})
= (\zeta_{84}^{-9} + \zeta_{84}^{-7} + \zeta_{84}^{3} + \zeta_{84}^{15}) + 
(\zeta_{84}^{25} + \zeta_{84}^{73} - \zeta_{84}^{3}).$$
Here the ``exotic'' real numbers $\zeta_{84} + \zeta_{84}^{13} - \zeta_{84}^{15}$
and $\zeta_{84}^{25} + \zeta_{84}^{73} - \zeta_{84}^{3}$
are equal to $2 \cos(13 \pi/84)$ and $2 \cos(25 \pi/84)$ respectively, and so can actually be written as the sum of two roots of unity.
\end{proof}

\subsection{Estimates}  \label{section:jones}
In this section, we analyze in more detail sums of the form $\beta=\zeta_N^{a}+\zeta_N^{-a} + \zeta_N^b +\zeta_N^{-b}$.  We wish to find all such sums which have $\ho{\beta} < 4 \cos(2 \pi /7)$.   Our argument in this section closely follows the paper of A.~J.~Jones~\cite{MR0224587},
 who studies expressions of the form $\beta = 1 + \zeta_N^a + \zeta_N^b$ with $\ho{\beta}$ small.  In outline, this argument uses the geometry of numbers, as follows. The Galois conjugates of $\zeta_N^{a}+\zeta_N^{-a} + \zeta_N^b +\zeta_N^{-b}$ are all of the form $\zeta_N^{ak}+\zeta_N^{-ak} + \zeta_N^{bk} +\zeta_N^{-bk}$ for $(k,N) = 1$. Using Minkowski's theorem, we can find a $k$
 such that all four roots of unity are all ``close'' to one, and thus the expression above is large.
 However, it is not immediately apparent that one can choose such a $k$ co-prime to $N$.
 Instead, using certain estimates involving the Jacobsthal function, we show that \emph{either} there exists a suitable
 $k$ co-prime to $N$ \emph{or} the integers $(a,b)$ satisfy a linear relation $ax - by = 0 \mod N$ with $(x,y)$ one of a small
 explicit finite set of integers ($(2,2)$, $(3,3)$, or $(2,4)$). In the latter case, we may study $\beta$ directly.
 A much simpler $1$-dimensional argument, also using estimates on the Jacobsthal function, gives a description of all $\beta = \zeta_{12} + \zeta_{20} + \zeta^{17}_{20} + \zeta_N^{a}+\zeta_N^{-a}$ such that $\ho{\beta} < 4 \cos(2 \pi /7)$.
 
\begin{df}
The Jacobsthal function
$j(N)$ is defined to be the smallest $m$ with the following property: In every arithmetic progression with
at least one integer co-prime to $N$, every $m$ consecutive terms contains an element co-prime to $N$. 
\end{df}

\begin{lemma} Suppose that $M|N$ has one fewer distinct prime factors
than $N$. Then \label{lemma:jacob}
 there is 
an inequality $j(M)^2 \le N/11$ for $N > 210$ and 
$N \ne 330, 390$.
\end{lemma}

\begin{proof} 
A result  of Kanold~\cite{MR0209247}
shows that $j(N) \le 2^{\omega(N)}$, where $\omega(N)$ is the number
of distinct primes dividing $N$.
Note that $j(N)$ only depends on the square-free part of $N$.
Suppose that $N$ has  at least $d \ge 7$
prime factors.
Then
$j(M) \le 2^{d-1}$, whereas
$$N/11 \ge (11)^{-1} \prod_{n = 1}^{d} p_n
\ge 2 \cdot 3 \cdot 5 \cdot 7  \cdot 13 \cdot 4^{d-6}
=\frac{1365}{512} \cdot 4^{d-1} \ge j(M)^2.$$
For smaller $d$, we note the following bounds on $j(M)$, noting that $M$
has one less distinct prime divisor than $N$.
These bounds were computed by Jacobsthal~\cite{MR0125047}:
$$\begin{aligned}
\text{if $d$ = 2}, & \text{ then $j(M) \le 2$,} \\
\text{if $d$ = 3}, & \text{ then $j(M) \le 4$,} \\
\text{if $d$ = 4}, & \text{ then $j(M) \le 6$ and} \\
\text{if $d$ = 5}, & \text{ then $j(M) \le 10$.} \end{aligned}$$
Thus, if $N$ has $5$ prime factors, we are done if
$N \ge 1100$, if $N$ has $4$ prime factors, we are done if
$N \ge 396$, and if $N$ has less than three prime factors, we are done
if $N \ge 176$. Yet if $N$ has $5$ prime factors, then $N \ge 2310$, and if
$N$ has $4$ prime factors, then $N \ge 396$ unless $N = 210$, $330$, or $390$.
\end{proof}

\begin{lemma} \label{lemma:jacob2}
$j(M) \le 2M/5 - 1$ for all $M$ except $M \in \{1,2,3,4,5,6,7,10,12\}$.
\end{lemma}
\begin{proof}
As in  Lemma~\ref{lemma:jacob} we use the result of Kanold to see that this theorem is true for all $M$ which is divisible by $3$ or more primes.  If $M$ is a product of two primes, then $j(M) \leq 4$, so the inequality follows so long as $M > 12$.  If $M$ is prime then $j(M) \leq 2$, so the inequality follows so long as $M>7$.
\end{proof}

\begin{remark} \emph{The known asymptotic bounds for $j(N)$ are much better,
see, for example, Iwaniec~\cite{MR0296043}.}
\end{remark}

Now we apply these bounds on Jacobsthal functions to finding small sums of roots of unity.

\begin{lemma} \label{lemma:etacase}
If $\beta = \zeta_{12} + \zeta_{20} + \zeta^{17}_{20} + \zeta^a_N + \zeta^{-a}_N$, then $\ho{\beta} =  2\cos(2\pi/60)$, or $\ho{\beta} =\zeta_{12} + \zeta_{20} + \zeta^{17}_{20}$, or $\ho{\beta} \geq 4 \cos(2 \pi /7)$.
\end{lemma}
\begin{proof}
Let $\eta =  \zeta_{12} + \zeta_{20} + \zeta^{17}_{20}$.  Write $N = AM$ where $A = (N,60)$.  We see that $\beta$ is conjugate to $\eta + \zeta^b_N + \zeta^{-b}_N$ for any $(b,N) = 1$ such that $b \equiv a \mod A$.
If there exists such a $b$ satisfying $b/N \in [-1/5,1/5]$, then $$\ho{\beta} \ge \eta + 2 \cos(2 \pi/5) = 3.022901\ldots$$
which is certainly greater than $4 \cos(2 \pi/7)$. To  guarantee the existence of such a $b$, we need to ensure that
at least one term of the arithmetic progression
of integers congruent to $a \mod A$ in the range
$[-N/5,N/5]$ is co-prime to $M$ (it is automatically
co-prime to $A$). 
The length of this arithmetic progression is at least
$2N/5A - 1 = 2M/5 - 1$. Such a $b$ always exists provided that
$j(M) \le 2M/5 - 1$, where $j(M)$ denotes the Jacobsthal function. By Lemma~\ref{lemma:jacob2}, this inequality holds for all $M$ except $M \in \{1,2,3,4,5,6,7,10,12\}$. This leaves a finite
number of possible $N$ and $\beta$ to consider, which we can explicitly compute. In particular, we look at 
\begin{align*}
N \in \{ 1,2,&3,4,5,6,7,8,9,10,12,14,15,16,18,20,21,24,25,28,30,35,36,40,42,45,48,50,60,\\70,72,
& 75,80,84,90,100,105,120,140,144,150,180,200,210,240,300,360,420,600,720\}.
\end{align*}
Indeed, in this range, the smallest largest conjugate
of $\beta$ is $2\cos(2\pi/60)$ (with, e.g., $N=60, a=17$), the second smallest is $\eta  = 2.40487\ldots$ (with, e.g., $N=4, a=1$), and the next smallest is
$$\eta + 2 \cos\left(2 \pi \frac{19}{60}\right) = 2.71559\ldots > 4 \cos(2 \pi /7).$$
\end{proof}

\begin{theorem} Suppose that $N > 230$ (we will need a slightly higher bound than the 210 of the Lemma~\ref{lemma:jacob}), and $N \ne 330$, or $390$. \label{theorem:useful1}
Let $\beta$ be a number of the form $\zeta_N^{a} + \zeta_N^{-a} +\zeta_N^{b}+\zeta_N^{-b}$ where $a$ and $b$ relatively prime,
then either:
\begin{enumerate}
\item $\beta$ is the sum of at most two roots of unity, and
thus $\ho{\beta} \le 2$,
\item $\beta$ is conjugate to $(1 + \sqrt{5})/\sqrt{2}$ or $\sqrt{6}$,
\item $\beta$ has a positive conjugate whose absolute value is bigger than $4 \cos(2 \pi/7)$, in particular,
$\ho{\beta} \ge 4 \cos(2 \pi/7)$.
\end{enumerate}
 \end{theorem}

Before proving this theorem we prove several lemmas.  Let us fix once and for all the constant $K = 2/49$.

\begin{lemma} Let $x,y \in \R$. 
Suppose that $x^2 + y^2 < K$. Then \label{lemma:estimate}
$2 \cos(2 \pi x) + 2 \cos(2 \pi y) > 4 \cos(2 \pi/7)$. 
\end{lemma}

\begin{proof} The minimum value of $2 \cos(2 \pi x) + 2 \cos(2 \pi y)$ occurs when
$x = y = 1/7$.
\end{proof}

\begin{df}
Denote by $\Lambda_{a,b,N} \subset \Z^2$ the set of integer vectors
$x$ such that $x.(a,-b) \equiv 0 \mod N$.
\end{df}

The determinant of the lattice $\Lambda_{a,b,N}$ is $N$. We may describe it
explicitly as follows. Let $u = (b,a)$, and
fix a vector $v$ such that $v.(a,-b) = 1$.
Then $\Z^2$ is generated by $u$ and $v$, and  $\Lambda_{a,b,N}$ is generated
by $u$ and $Nv$.
Up to scalar, there is a canonical map $\phi: \Lambda_{a,b,N} \rightarrow \Z/N \Z$ obtained
by reduction modulo $N$. We say that a vector $\lambda \in \Lambda_{a,b,N}$
is co-prime to $N$ if and only if the image $\phi(\lambda)$ of $\lambda$ in $\Z/N\Z$ lands
in $(\Z/N\Z)^{\times}$.
Denote by $Q$ the quadratic form $Q(x,y) = x^2 + y^2$ on $\Z^2$ restricted to
$\Lambda_{a,b,N}$; it has
 discriminant $-4 N^2$. 

\begin{lemma} If $\lambda$ is co-prime to $N$, \label{lemma:ineq}
and $Q(\lambda) \leq K \cdot N^2$, then $\beta = \zeta_N^a+\zeta_N^{-a}+\zeta_N^b+\zeta_N^{-b}$ has $\ho{\beta} \geq 4 \cos(2 \pi/7)$.
\end{lemma}

\begin{proof} We may write $(r,s) = \lambda = k(b,a) \mod N$, for some $k$ co-prime to $N$.
Replacing $\zeta$ by $\zeta^k$  is thus an automorphism of $\Q(\zeta)$, which has
the effect of replacing $\beta$ by
 $$\zeta^{ka} + \zeta^{-ka} + \zeta^{kb} + \zeta^{-kb}
 = 2 \cos(2 \pi r/N) +  2 \cos(2 \pi s/N) \le 4 \cos(2 \pi/7).$$
 Hence, from Lemma~\ref{lemma:estimate}, we deduce the result.
 \end{proof}

 \begin{proof}[Proof of Theorem~\ref{theorem:useful1}] It suffices to assume that $\ho{\beta} < 4 \cos(2 \pi/7)$ and derive a contradiction.
Note that the Galois conjugates of $\beta$ can be obtained by replacing $\zeta_N$ by $\zeta_N^k$ for some integer $k$ such that $(k,N) = 1$.  Hence the Galois conjugates are exactly the numbers of the form $\zeta_N^{a'} +\zeta_N^{-a'} + \zeta_N^{b'}+\zeta_N^{-b'}$ for $(a',b') \in \Lambda_{a,b,N}$ which is relatively prime to $N$.

 By reduction theory for quadratic forms, there exists
a basis $\mu$, $\nu$ of $\Lambda_{a,b,N}$ for which
$$Q(x \cdot \mu + y \cdot \nu) = A x^2 + B x y + C y^2, \qquad |B| \le A \le C,
\quad \Delta: = B^2 - 4AC = - 4 N^2.$$
 Now
$\displaystyle{A^2 \le AC \le AC +  \frac{1}{3}(AC - B^2) = -\frac{4}{3} \Delta = 3 N^2
\le K^2 \cdot N^4}$,
providing that $N > 43$. 
Hence $Q(\mu) < K \cdot N^2$, and
thus, by Lemma~\ref{lemma:ineq}, $\mu$ is not co-prime to $N$.
Since $\phi: \Lambda_{a,b,N} \rightarrow \Z/N \Z$ is surjective,
 there exists an integer $k$ such that
$k \mu + \nu$ is co-prime to $N$.
By assumption, $N$ has a prime factor $q$ that divides $\mu$.  The terms in this sequence must all be automatically co-prime to $q$.  Let $M$ be $N$ divided by the highest power of $q$ dividing $N$.  In order to find something of the form $k \mu + \nu$ is co-prime to $N$, it suffices to find one that is co-prime to $M$.  By definition of the Jacobsthal function, 
it follows that we may take a 
$$k \in \left[\frac{-j(M)}{2} + \frac{B}{2A},\frac{j(M)}{2} + \frac{B}{2A}\right]$$
such that
$k \mu+ \nu$ is co-prime to $N$, and hence
$Q(k \mu + \nu) > K \cdot N^2$. Yet
$$Q(k \mu + \nu)
= A k^2 + B k + C  = A (k - B/2A)^2 + (4AC - B^2)/4A
\le   j(M)^2 A/4 + N^2/A,$$
and thus
$$j(M)^2 A^2/4 + N^2 \ge K N^2 A.$$

Since this inequality holds for $A=0$, and since $j(M)^2 > 0$, we see that the inequality holds exactly on the complement of some (possibly empty) interval.  Using the assumption that $N \geq 230$  and Lemma~\ref{lemma:jacob}, we see that the inequality does not hold for $A= \sqrt{3} N$.  Namely,
$$\frac{3}{4} j(M)^2  N^2 + N^2 \leq \frac{3}{44}  N^3 + N^2 < K N^2 A $$
Similarly, using that $N \ge 28$, we also see that the inequality does not hold for $A=25$.  Namely, 
$$\left(\frac{25}{4} j(M)^2 + N^2/25 \le \frac{N}{44} + \frac{N^2}{25}\right) A < K \cdot N^2 \cdot A.$$  
Hence the inequality does not hold for any $A$ in the interval $[25, \sqrt{3}N]$.  Since $A$ is positive and $A^2 \leq 3 N^2$ it follows that $A \leq 24$, and hence $Q(\mu) \leq 24$.

Write $\mu = (x,y)$. Then $x^2 + y^2 \le 24$, and $ax - by \equiv 0 \mod N$.
Recall that $\mu$ is not co-prime to $N$, and thus $x$ must not be co-prime to $y$.
It follows that $(x,y)$, up to sign and ordering, is one of the pairs
$(2,2), (3,3)$ or  $(2,4).$
We consider each of these in turn.
\begin{enumerate}
\item $(x,y) = (2,2)$.
It follows that $(a,b) = (a,a)$ or $(a,a+N/2)$.
In the first case, the maximum absolute value of any conjugate of $\beta$
is of the form $4 \cos(\pi/M)$ for some $M$. In the second case, $\zeta^a = - \zeta^b$,
so $\beta = 0$.
\item $(x,y) = (3,3)$, either $(a,b) = (a,a)$, or, after making an appropriate
permutation, $(a,b) = (a,a+N/3)$. In this case, with $\omega^3 = 1$,
$$\beta = \zeta^a + \zeta^{-a} + \zeta^{a} \omega + \zeta^{-a} \omega^{-1}
= - \omega^{-1} \zeta^{a} - \omega \zeta^{a}$$
is a sum of two roots of unity.
\item $(x,y) = (2,4)$. The only new possibility is $(a,b) = (a,a + N/4)$. Letting
$i^4 = 1$, we find that
$$\beta = \zeta^a (1 + i) + \zeta^{-a} (1 - i)
= \sqrt{2} (\zeta^a \zeta_8 + \zeta^{-a} \zeta^{-1}_8)
=  \sqrt{2}(\zeta' + \zeta^{'-1}),$$
and hence $\ho{\beta} = 2 \sqrt{2} \cos(\pi/M)$
for some $M$. The only numbers of this kind between
$2$ and $4 \cos(\pi/7)$ occur for $M = 5$ and $6$, for which we
obtain the values 
$(1 + \sqrt{5})/\sqrt{2}$ and $\sqrt{6}$.
\end{enumerate}
This completes the proof of the theorem.
\end{proof}

\begin{theorem} Let $\beta$ be totally real, and suppose that $\Num(\beta) \le 5$. \label{theorem:use}
Then either
\begin{enumerate}
\item $\beta$ is a sum of at most two roots of unity.
\item $\displaystyle{\ho{\beta}   \ge 4 \cos(2\pi/7)}.$
\item A conjugate of $\beta$ is one of the following numbers, listed in
increasing order:
$$
\begin{aligned}
\frac{\sqrt{3} + \sqrt{7}}{2} & = 2.18890105931673 \ldots \\
\sqrt{5} & = 2.23606797749978\ldots \\
1 + 2 \cos(2 \pi/7) = 2 \cos(\pi/7) + 2\cos(3 \pi/7) & =2.24697960371746\ldots \\
\frac{1 + \sqrt{5}}{\sqrt{2}} = 2 \cos(\pi/20) + 2 \cos(9\pi/20) &  = 2.28824561127073 \ldots \\
1 + 2 \cos(4 \pi/13) + 2 \cos(6 \pi/13) & = 2.37720285397295\ldots \\
1 + 2 \cos(2 \pi/11) + 2 \cos(6 \pi/11) & = 2.39787738911579\ldots \\
2 \cos(\pi/30) + 2 \cos(13 \pi/30) & =  2.40486717237206\ldots \\
1 + \sqrt{2} & =2.41421356237309\ldots \\
\sqrt{6} = 2 \cos(\pi/12) + 2 \cos(5 \pi/12) & = 2.44948974278317 \ldots \\
2 \cos(11 \pi/42) + 2 \cos(13 \pi/42) & = 2.48698559166908\ldots \end{aligned}$$
\end{enumerate}
\end{theorem}

\begin{proof}  We split into cases using the classification of Lemma~\ref{lemma:84}.
If $\Num(\beta)=3$ and $\beta = 1 + \zeta^a + \zeta^{-a}$, then the largest conjugate of 
$\beta$ is $1 + 2 \cos(2 \pi/N)$. For $N$ less than $7$ we could rewrite this as a sum of fewer than three terms.   If $N = 7$, then $\beta = 1 + 2 \cos(2 \pi/7)$. If $N = 8$, then  $\beta = 1 + \sqrt{2}$. If $N \ge 9$, then $\beta > 4 \cos(2 \pi/7)$.

If $\Num(\beta) = 4$,  and $\beta = \zeta^a + \zeta^{-a} + \zeta^{b} + \zeta^{-b}$ then the previous theorem applies if $N > 230$ and $N \ne 330$ or $390$.

If $\Num(\beta) = 5$, and $\beta = 1+\zeta^a + \zeta^{-a} + \zeta^{b} + \zeta^{-b}$ then the previous theorem applies to $\beta-1$ if $N > 230$ and $N \ne 330$ or $390$.  If $\beta-1$ has a positive conjugate whose absolute value is larger than $4 \cos(2 \pi/7)$, it follows that $\ho{\beta}>1+4 \cos(2 \pi/7)$

If $\Num(\beta) = 5$ and $\beta = \zeta_{12} + \zeta_{20} + \zeta^{17}_{20} + \zeta^a_N + \zeta^{-a}_N$ then we apply Lemma~\ref{lemma:etacase}.

Hence we need only consider finitely many remaining numbers.  First,  we may have that $N \le 230$ or $N=330$ or $N=390$.  Second,  we may be in one of the finitely many exceptional cases in Lemma~\ref{lemma:84}.  In the former case, we compute directly that the largest conjugates
all have absolute value at least $4 \cos(2 \pi/7)$, except for the exceptions listed above.
For the latter case, only one of the exception numbers, $(\sqrt{3}+\sqrt{7})/2$, has $\ho{\beta}$ small enough.

\end{proof}

\begin{remark} \emph{It is a consequence of
this computation and Theorem~\ref{theorem:main}  that
the smallest largest conjugate of a real cyclotomic integer  which is
\emph{not} a sum of $5$ or fewer roots of unity is
$$\frac{1 + \sqrt{13}}{2} = 
- \left(\zeta^2 + \zeta^{-2} + \zeta^6 + \zeta^{-6} + \zeta^{8} + \zeta^{-8}\right)
=  2.30277\ldots$$
where $\zeta$ is a $13$th root of unity.}
 \end{remark}

We shall use the following result, which follows directly
from Theorem~\ref{theorem:use}.

\begin{corr} Let $\beta$ be a real cyclotomic integer such that \label{corr:useful}
$3 \leq \Num(\beta) \leq  5$. Then either
$\beta$ is conjugate to $\ep$, $\sqrt{5}$, $1 + 2 \cos(2 \pi/7)$,
$(1 + \sqrt{5})/\sqrt{2}$,  or
$\ho{\beta} \ge 76/33$.
\end{corr}

\section{The normalized trace}
 
The goal of the next two sections is to prove that 

\begin{theorem} If $\beta$ is a cyclotomic integer such that \label{theorem:firstbound}
$\beta$ is real, $\displaystyle{\ho{\beta} < 76/33}$,  and
$\Num(\beta) \ge 3$, 
then either $\ho{\beta} = (1 + \sqrt{13})/2$, or
$\beta \in \Q(\zeta_{N})$, where
$$N = 4 \cdot 3 \cdot 5 \cdot 7 = 420.$$
 \end{theorem}
 
 So suppose that $\beta$ is real, that $\beta \in \Q(\zeta_N)$ with $N$ minimal, that $\Num(\beta) \ge 3$, and $\displaystyle{\ho{\beta} < 76/33}$.  First we prove a lemma which allows us to reduce to studying $\beta$ with $\M(\beta) < 23/6$.  Second, we show that if $p^k \mid N$ with $k>1$ then $p^k = 4$, this argument uses techniques developed by Cassels \cite{MR0246852}.  In the next section, we will show that if $p>7$ then either $p \nmid N$ or $p=13$ and $\beta=(1+\sqrt{13})/2$.  Again this argument will use techniques generalizing those of Cassels.
 
 \subsection{Relationship between $\ho{\beta}$ and $\M(\beta)$}


The following Lemma allows us to reduce to considering $\beta$ with $\M(\beta)$ small.
\begin{lemma} Let $\beta  \in \Qbar$ be a totally real algebraic 
integer, \label{lemma:bound}
   and suppose
 that $\ho{\beta} < 76/33 = 2.303030\ldots$ \ 
 Then either $\beta^2 = 4$ or $5$, or $\M(\beta) < 23/6 = 3.833333\ldots$
 \end{lemma}
\begin{proof} Let $\kappa = 23/6$, and let $\alpha = \displaystyle{\frac{1+\sqrt{13}}{2}}$.
Since 
$$\M(\alpha) = \frac{1}{2} \left(\frac{7 + \sqrt{13}}{2} + \frac{7 - \sqrt{13}}{2} \right) = 
\frac{7}{2} < \frac{23}{6},$$
we may assume that $\beta$ is not a conjugate of $\alpha$. Similarly
$\M(\sqrt{3}) = 3$, and so we may assume that $\beta^2 \ne 3$.

The inequality
$$\Theta(x) = 120 (\kappa  - x)  - \left(36 \log |x - 4| + 160 \log |x - 5| 
+ 9 \log |x - 3| + 2 \log|x^2 - 7x + 9| \right) > 0$$
for $x \in [0,(76/33)^2] = [0,5.303948\ldots]$ is an easy calculus exercise.
(Note that the roots of the polynomial $x^2 - 7x + 9$ are the conjugates of $\alpha^2$.)
The critical points are the roots of $-40200 + 68381x - 44376 x^2 + 13814 x^3 - 2071 x^4 + 
   120 x^5 $.  The absolute minimum value in this range occurs at approximately
$x = 3.320758\ldots$ where $\Theta$ obtains its minimum of
roughly $0.394415\ldots$
 
 \begin{figure}[!h]
\begin{center}
\begin{tikzpicture}[x=2cm]
\draw[smooth,samples=1000,domain=0:5.305] plot function{log(20*23-120*x-(36*log(abs(x-4))+160*log(abs(x-5))+9*log(abs(x-3))+2*log(abs(x**2-7*x+9)))+1)};

 \draw[->] (0,0) -- (5.5,0) node[right] {$x$};
 \draw[->] (0,-0.2) -- (0,7.5) node[above] {$\Theta(x)+1$};

  \foreach \x/\xtext in {1/1, 2/2, 3/3, 4/4, 5/5, 5.305/5.305}
   \draw[shift={(\x,0)}] (0pt,2pt) -- (0pt,-2pt) node[below] {$\xtext$};

{0, 2.30259, 4.60517, 6.90776, 9.21034, 11.5129}

 \foreach \y/\ytext in {0/1, 2.30259/10, 4.60517/100, 6.90776/1000}
   \draw[shift={(0,\y)}] (2pt,0pt) -- (-2pt,0pt) node[left] {$\ytext$};
\end{tikzpicture}

\end{center}
\caption{The  function $\Theta(x)+1$, on a log scale. The four visible peaks, and one that is not apparent on the graph, near $x=5.30278$, are actually asymptotes.}
\end{figure}
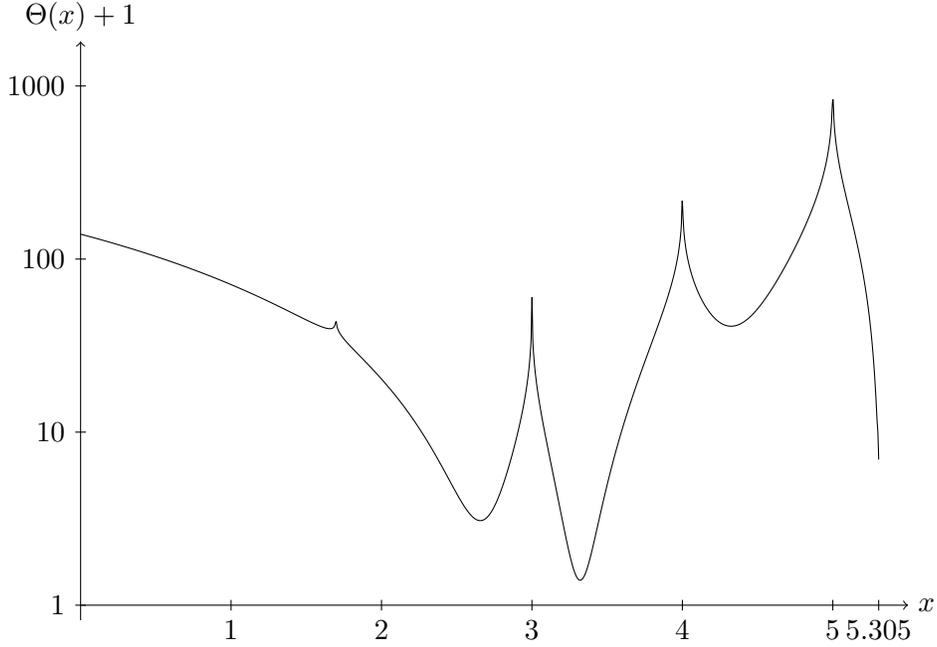

Let $S = \{x_i\}$ be a finite set of real numbers in $[0,(76/33)^2]$ whose average is greater than
 $\kappa = 23/6$.
Then the average of $\kappa - x_i$ is less than zero, and hence
$$0 >  120 \sum (\kappa - x_i) \ge 
36 \sum \log |x_i - 4| + 160 \sum \log |x_i - 5| + 9 \sum \log |x_i - 3|
+ 2 \sum \log |x^2_i - 7 x_i + 9|.$$
Suppose that $S$ consists of the squares of the conjugates of
$\beta \in K = \Q(\beta)$. Since $\ho{\beta} < 76/33$, it follows
that all the $x_i$ lie in $[0,(76/33)^2]$. Since we are assuming that
$\beta^2 \ne 3,4,5$, nor a conjugate of $\alpha^2$ (which is a root of
$x^2 - 7x + 9$), it follows that the norms of $\beta^2 - 3$, $\beta^2 - 4$, and
$\beta^2 - 5$, as well as $\beta^4 - 7 \beta^2 + 9$,
are non-zero algebraic integers. Hence the absolute value of their norms are at least
one. Taking logarithms, we deduce that every sum occurring on the right hand
side of the inequality above is non-negative, which is a contradiction,
and the lemma is established.
\end{proof}

\begin{remark} \emph{ The constants (120,36,160,9,2) chosen in this proof are somewhat
arbitrary and mysterious, and fine tuning would certainly lead to
an improved result. However, to increase $76/33$ substantially one
would need to allow $\M(\beta)$ to increase, which would increase the
combinatorial difficulty of our later arguments.}
\end{remark}

It follows that in order to prove Theorem~\ref{theorem:firstbound}, we may
assume that $\M(\beta) < 23/6$. 

We shall also frequently use the following lemmas:

\begin{lemma}[Cassels' Lemma 2~\cite{MR0246852}] If $\Num(\alpha) \ge 2$, 
then
$\M(\alpha) \ge 3/2$.
\end{lemma}

\begin{lemma}[Cassels' Lemma 3~\cite{MR0246852}] If $\Num(\alpha) \ge 3$, then $\M(\alpha) \ge 2$.
\end{lemma}

\subsection{The case when $p^2 | N$} \label{section:cassels1}

Suppose that $\beta \in \Q(\zeta_N)$, and suppose that $N$ is minimal with
respect to this property.
We start with what Cassels calls the \emph{second case}, that is, the case
when $N$ admits a prime $p$ such that $p^2 | N$.
Explicitly, assume that  $p^{m} \| N$ for an integer $m \ge 2$.
Let $N = p^{m-1} M$, so $p \| M$. 
Let
$\zeta$ be a $p^{m}$th root of unity.
We may write
$$\beta = \sum_S \zeta^i \alpha_i,$$
where $\alpha_i \in \Q(\zeta_M)$.
Here $S$ denotes any set of $p^{m-1}$ integers that are distinct modulo $p^{m-1}$.
After having chosen an $S$, the $\alpha_i$ are determined uniquely by $\beta$.
Since $\beta$ is real, it is invariant under complex conjugation.
It follows that
$$\sum_S \zeta^i \alpha_i = 
\sum_S  \zeta^{-i} \overline{\alpha_i}.$$
If $S$ is odd,
let $S$ denote the set
$\displaystyle{\left\{- \frac{(p^{m-1}-1)}{2}, - \frac{(p^{m-1}-3)}{2}, \ldots, -1,0,1,2, \ldots \frac{(p^{m-1}-1)}{2}\right\}}$.
If $p = 2$, let $S = \{-(2^{m-2} -1), \ldots,-1,0,1,2,\ldots,2^{m-2}\}$.
From the uniqueness of this expansion we deduce, if $p$ is odd, that
$\overline{\alpha_i}   = \alpha_{-i}$
for all $i  \in S$.
If $p = 2$, we deduce that
$\overline{\alpha_i}  = \alpha_{-i}$ if $i < 2^{m-2}$, and
that $\overline{\zeta^{2^{m - 2}} \alpha_{2^{m-2}}} = \zeta^{2^{m - 2}} \alpha_{2^{m-2}}$.

\begin{lemma} There is an equality \label{lemma:Mcase2}
$\displaystyle{\M(\beta) = \sum \M(\alpha_i)}$.
\end{lemma}

\begin{proof} Our proof is essentially that of Cassels (who proves it under extra
hypotheses that are not required for the proof of this particular statement).
We reproduce the proof here.
The  conjugates of $\zeta$ over $\Q(\zeta_M)$ are $\zeta \cdot \zeta^{pn}$
  for $n = 0$ to $p^{m-1} - 1$.
  Let $\M'(\theta)$ denote the average of the conjugates of $|\theta|^2$ over $\Q(\zeta_M)$.
  Then
  $$
  \begin{aligned}
  p^{m-1} \M'(\beta) = & \  \sum_{n} \sum_{i} \zeta^i \alpha_i \zeta^{pni}
  \sum_{j} \zeta^{-j} \overline{\alpha_j} \zeta^{-pnj} \\
  = & \ \sum_n \sum_{i,j} \zeta^{i-j} \zeta^{pn(i-j)} \alpha_i \overline{\alpha_{j}} \\
   = & \ \sum_{i,j}    \zeta^{i-j}  \alpha_i \overline{\alpha_{j}}  \sum_n   \zeta^{pn(i-j)}.
  \end{aligned}
    $$
    Now $i \equiv j \mod p^{m-1}$ if and only if $i = j$, and thus
    $\zeta^{p(i-j)} = 1$ if and only if $i = j$. For all other pairs $(i,j)$, the final sum is a power
    sum of a non-trivial root of unity over a complete set of
    congruence classes, and is thus $0$. Hence, as in Cassels, we find that
     $$\M'(\beta) = \sum  |\alpha_i|^2,$$
  and the result follows upon taking the sum over the conjugates of
    $\Q(\zeta_M)$ over $\Q$.
\end{proof}

Let $X$ denote the number of $\alpha_i$ which are non-zero.
In order the prove Theorem~\ref{theorem:firstbound} in this case,
we must show that if $p^2 \mid N$, then $p = 2$ and $4 \| N$.

\subsection{The case when $X = 1$}
If $p$ is odd, then $\beta  = \alpha = \overline{\alpha}$. In this case
we find that $\beta \in \Q(\zeta_M)$, contradicting the minimality assumption on $N$.
If $p = 2$, then either $\beta  = \alpha = \overline{\alpha}$, or
$$\beta = \zeta^{2^{m - 2}} \alpha_{2^{m-2}} = 
\overline{\zeta^{2^{m - 2}} \alpha_{2^{m-2}}}.$$
By minimality, we deduce that $2^{m-2} =1$, and hence $2^m = 4$.
(The number $\displaystyle{\epl}$ is, in fact, of this form.)

\subsection{The case when $X = 2$}
If $p$ is odd, we deduce that 
 $\beta = \zeta \alpha + \zeta^{-1} \overline{\alpha}$.
 If $\Num(\alpha) \le 2$, then we are done, by Corollary~\ref{corr:useful}.
  If $\Num(\alpha) > 2$, then by Lemma~3 of Cassels,
  $\M(\alpha) \ge 2$, and $\M(\beta) \ge 4$.
  
  If $p = 2$, the same argument applies,
except in this case it could be that
$$\beta  = \alpha_0 + \zeta^{2^{m - 2}} \alpha_{2^{m-2}}.$$
Once more,
since $N$ is minimal with respect to $\beta$, it must be the case
that
$2^{m - 2} = 1$ and $2^m = 4$.

\subsection{The case when $X = 3$}
If $p$ is odd, then, for some primitive $p^{m}$th root of unity $\zeta$,
we have $\beta = \zeta \alpha + \gamma + \zeta^{-1} \overline{\alpha}$.
If $\alpha$ is a root of unity,
  then, by Corollary~\ref{corr:useful}, we may assume that $\Num(\gamma) \ge 3$ and
  hence (by Lemma~3 of Cassels) that $\M(\gamma) \ge 2$, and
  thus $\M(\beta) \ge 1 + 1 + 2 = 4$.
  If $\Num(\alpha)  = 2$, then by Lemma~2 of Cassels,
  $\M(\alpha) \ge 3/2$, and hence $\M(\beta) \ge 3/2 + 3/2 + 1= 4$.
  
  If $p = 2$, there is at least one $i$ such that $\alpha_i \ne 0$ and
  $i \ne 0, 2^{m-2}$. It follows that 
  $\beta = \zeta \alpha + \gamma + \zeta^{-1} \overline{\alpha}$
  for some $\gamma$ such that $\overline{\gamma} = \gamma$,
  and the proof proceeds as above.
  
\subsection{The case when $X \ge 4$} It is immediate that $\M(\beta) \ge 4$.

\section{The case when $p$ exactly divides  $N$} \label{section:cassels2}
We now consider what Cassels calls the \emph{first case}, where $p \mid \mid N$. \label{section:firstcase}
 So suppose that $\beta$ is real, that $\beta \in \Q(\zeta_N)$ with $N$ minimal, that $\Num(\beta) \ge 3$, and $\displaystyle{\ho{\beta} < 76/33}$.  We will show in this section that if $p \mid N$ then $p \leq 7$ or $p=13$ and $\beta=(1+\sqrt{13})/2$. (In particular,
  we may assume that $p$ is odd.)  This will complete the proof of Theorem~\ref{theorem:firstbound}.

Write $N = pM$ once again, and let $\zeta$ be a primitive $p$th root of unity.
The conjugates of $\zeta$ are now $\zeta \cdot \zeta^k$ for any $k$ 
\emph{except} $k \equiv -1 \mod p$. 

We write
$$\beta = \sum_S \zeta^i \alpha_i,$$
where $\alpha_i \in \Q(\zeta_M)$ and $S$ denotes
$\{-(p-1)/2, \ldots, 0,1,\ldots,(p-1)/2\}$.
This expansion is no longer unique; there is ambiguity given by a fixed
constant for each element.
Since $\beta$ is real, it is invariant under complex conjugation.
It follows that there exists a fixed $\lambda \in \Q(\zeta_M)$ such that
$$\overline{\alpha_i} = \alpha_{-i} + \lambda,$$
The element $\lambda$ itself must satisfy $\overline{\lambda} = - \lambda$,
or equivalently, that $\lambda \cdot \sqrt{-1}$ is real.

Let $X$ denote the number of terms occurring in $S$
such that $\alpha_i \ne 0$.

\begin{lemma} If $\lambda \ne 0$, then $X \ge (p+1)/2$. \label{lemma:balanced}
If $\lambda$ is a root of unity, then $\lambda = \pm \sqrt{-1}$.
\end{lemma}

\begin{proof} If $\lambda \ne 0$, then since
$\alpha_{-i} - \overline{\alpha_i} = \lambda$, at least one of
$\{\alpha_i,\alpha_{-i}\}$ must be non-zero. Since there are $(p+1)/2$ such
pairs not containing any common element, the result follows. The second claim
follows from the fact that $\lambda \cdot \sqrt{-1}$ is real.
\end{proof}


\subsection{The case when $X = 1$}
We deduce that $\beta = \alpha = \overline{\alpha}$, contradicting the minimality 
of $N$.

\subsection{The case when $X = 2$}
If $p \ge 7$, by Lemma~\ref{lemma:balanced}, we may assume that $\lambda = 0$, and hence
$$\beta = \zeta \alpha + \zeta^{-1} \overline{\alpha}.$$
If $\alpha$ is a root of unity, then $\Num(\beta) \le 2$. Hence, we may assume
(replacing $\alpha$ by a conjugate) that $|\alpha| \ge \sqrt{2}$.
Note that we may choose $\zeta$ to be primitive, since $N$ was chosen to be minimal
with respect to $\beta$.
Write
 $\zeta \alpha =  |\alpha| e^{2 \pi i \theta}$. The conjugates of
 $\zeta$ are $\zeta \cdot \zeta^k$ where $k$ is any integer such
 that $k \not\equiv -1 \mod p$. We replace $\zeta$ by a conjugate
 to
 make $\theta$ as close to $0$ or $1/2$ as possible. 
  By Dirichlet's box principle, with no constraint on $k$ we could insist that
 $\|\theta\| \le 1/2p$, or, if we liked, that $\|\theta - 1/2\| \ge 1/2p$. Given our single constraint,
 we may \emph{at least} find a conjugate of $\zeta$ such that $\theta$ satisfies one of
 these inequalities.
 In either case, we deduce  that
 $$|\beta| > 2 |\alpha| \cos(\pi/7) \ge 2 \sqrt{2} \cos(\pi/7)  = 2.548324\ldots > 2.303030\ldots = 76/33.$$

\subsection{The case when $X = 3$}

Suppose that $X = 3$, and suppose that $p \ge 11$. By Lemma~\ref{lemma:balanced},
we may assume
that
$\lambda = 0$. We may therefore assume that
$$\beta = \zeta \alpha + \gamma + \zeta^{-1} \overline{\alpha},$$
where $\overline{\gamma} = \gamma$.
After conjugating, we may assume that $|\alpha \gamma| \ge 1$. After possibly negating $\beta$, we may assume that $\gamma$ is positive.
Write $\zeta \alpha =  |\alpha| e^{2 \pi i \theta}$. Now we must insist that $\|\theta\|$ is small rather than $\|\theta - 1/2\|$, and
thus may only deduce that $\|\theta\| \le 1/p$. It follows that
$$\beta \ge  2 |\alpha| \cos(2 \pi/11) + \frac{1}{|\alpha|}
\ge 2 \cdot \sqrt{\frac{2 |\alpha| \cos(2 \pi/11)}{ |\alpha|}} = 2.594229\ldots > 76/33.$$

\subsection{An interlude}
We recall some facts that will be used heavily in the sequel.
There is always a formula:
\begin{equation} \label{equation:one}
(p-1) \M(\beta) = (p-X) \sum \M(\alpha_i) + \sum \M(\alpha_i - \alpha_j),
\end{equation}
(This is Equation $3.9$ of Cassels, his argument is similar to that in Lemma~\ref{lemma:Mcase2}.) We often use this equation in the following way.
Suppose that the $X$ non-zero terms break up into sets of size $X_j$ consisting of
equal terms. Then, since $\M(\alpha_i - \alpha_j) \ge 1$ if $\alpha_i \ne \alpha_j$, we 
deduce that
\begin{equation} \label{equation:two}
\begin{aligned}
(p-1) \M(\beta) \ge & \ (p-X) \sum \M(\alpha_i) + \frac{1}{2} \sum X_j(X - X_j) \\
= & \ (p- X)  \sum \M(\alpha_i)  + \frac{1}{2} \left(X^2 - \sum X^2_j \right). \end{aligned}
\end{equation}
We also note the following lemma, whose proof is obvious.

\begin{lemma} Suppose that \label{lemma:zero} at least $Y$ of the $\alpha_i$
are equal to $\alpha$. Then we may --- after subtracting $\alpha$ from \label{lemma:reduce}
each $\alpha_i$ ---
 assume that $X \le p - Y$.
\end{lemma}

Finally, we note the following.
\begin{lemma} Suppose that $p \ge 13$. Then we may assume that \label{lemma:thirteen}
$\displaystyle{X \le \frac{p-1}{2}}$ and $\lambda = 0$.
\end{lemma}

\begin{proof}
The Corollary to Lemma~1 of Cassels states that if $\M(\beta) < \frac{1}{4}(p+3)$ then at least $\frac{p+1}{2}$ of the $\alpha_i$ are equal to each other.  By Lemma~\ref{lemma:zero} it follows that we can assume that $\displaystyle{X \le \frac{p-1}{2}}$.  Hence, we need only compute that
$$\frac{(p+3)}{4} \ge 4 > 23/6 > \M(\beta).$$
\end{proof}

\subsection{The case when $X = 4$, and $p \ge 11$}

By Lemma~\ref{lemma:balanced}, we may write
$$\beta = \zeta \alpha + \zeta^{-1} \overline{\alpha}  + \zeta^i \gamma +
\zeta^{-i} \overline{\gamma}.$$
If $\alpha$ and $\gamma$ are roots of unity, then we are done by Corollary~\ref{corr:useful}.
Thus, we may assume that $\Num(\alpha) \ge 2$, and hence that
$\M(\alpha) = \M(\overline{\alpha}) \ge 3/2$.  
If $\gamma$ is not equal to $\alpha$ or $\overline{\alpha}$, then
$\{\alpha,\overline{\alpha}\}$
are certainly both distinct from $\{\gamma,\overline{\gamma}\}$.
Hence evaluating $\M$ on the corresponding differences is at least one.
Using Equation~\ref{equation:one}, we deduce that
$$(p-1) \M(\beta) \ge (p - 4)(3/2 \cdot 2 + 2)  + 4,$$
and hence, if $p \ge 11$, that $\M(\beta) \ge 3.9 > 23/6$.  This contradicts
Lemma~\ref{lemma:bound}.
Suppose that $\gamma = \alpha$.
If $\alpha$ is not real, then $\alpha$ and $\gamma$ are distinct from
$\overline{\alpha}$ and $\overline{\gamma}$, and hence
$$(p-1) \M(\beta) \ge (p-4)(3/2 \cdot 4) + 4,$$
from which we deduce a contradiction as above. If $\gamma = \alpha$
\emph{is} real, then
 $$\beta = \alpha \left(\zeta + \zeta^{-1} + \zeta^i + \zeta^{-i}\right).$$
Since $\alpha$ and
 $\left(\zeta + \zeta^{-1} + \zeta^i + \zeta^{-i}\right)$ lie in disjoint Galois extensions, the maximal conjugate
 of $\beta$ is the product of the maximal conjugate of $\alpha$ and the
 maximal conjugate of the second factor. Since $p > 5$, the latter
 factor cannot be written as a sum of a smaller number of roots of unity,
 and hence its maximum is at least
 $(\sqrt{3}+\sqrt{7})/2$, by Corollary~\ref{corr:useful}. Yet, since
 $\M(\alpha) \ge 3/2$,  at least one conjugate of $\alpha$ has absolute value
 $\ge \sqrt{2}$, and hence
 $$\ho{\beta} \ge \frac{\sqrt{14} + \sqrt{6}}{2} = 3.095573\ldots > 76/33.$$
 
\subsection{The case when $X = 5$, and $p \ge 11$}
Once more by Lemma~\ref{lemma:balanced}, we may write
that
$$\beta = \zeta \alpha + \zeta^i \gamma + \delta + \zeta^{-i} \overline{\gamma}
+ \zeta^{-1} \overline{\alpha}.$$
If $\alpha$, $\delta$, and $\gamma$ are roots of unity, then
we are done by Corollary~\ref{corr:useful}.
We break up our argument into various subcases.

\subsubsection{$X = 5$ and $\M(\alpha) = \M(\gamma) = 1$,
$\M(\delta) \ge 3/2$}
If $\alpha = \gamma$ are both real, then, after replacing $\beta$ by $-\beta$
if necessary, they are both one, and
$$\beta = \delta + \left(\zeta + \zeta^{-1} + \zeta^i + \zeta^{-i}\right).$$
We deduce that
$$(p-1) \M(\beta) \ge (p-5)(3/2 + 4) + 4.$$
This implies that $\M(\beta) \ge 4$ if $p \ge 13$. By computation, if $p = 11$,
there exist  two conjugates of the right hand side, one positive and one
negative, both of which have absolute value at least
$$2 \cos(2 \pi/11) + 6 \cos(3 \pi/11) = 1.397877\ldots$$
On the other hand, there exists a conjugate of $\delta$ with absolute value at
least $\sqrt{2}$, and hence there exists  a conjugate of $\beta$ with absolute
value at least
$$\sqrt{2} + 2 \cos(2 \pi/11) + 6 \cos(3 \pi/11)= 2.812090\ldots >  2.303030\ldots = 76/33$$
Thus we may assume that
either $\alpha$ is real and $\gamma$ is not,
or that they are both not real. Thus $\delta$ is distinct from
the four terms $\{\alpha,\overline{\alpha},\gamma,\overline{\gamma}\}$
and either $\{\alpha,\overline{\alpha}\}$ has no intersection
with $\{\gamma,\overline{\gamma}\}$ or
$\{\alpha,\gamma\}$ has no intersection with $\{\overline{\alpha},\overline{\gamma}\}$.
In either case, we deduce that
$$(p-1) \M(\beta) \ge (p-5)(3/2 + 4) + 8,$$
which implies that $\M(\beta) \ge 4.1 > 23/6$.

\subsubsection{$X = 5$ and $\M(\alpha) \ge 3/2$}
We break this case up into further subcases.
\begin{enumerate}
\item $\M(\gamma) = \M(\delta) = 1$:
Clearly the terms involving $\alpha$ are distinct from the other terms,
and hence
$$(p-1) \M(\beta) \ge 6(p-5) + 6,$$
and thus $\M(\beta) \ge 4.2 > 23/6$.
\item $\M(\delta) \ge 3/2$, and $\M(\gamma) = 1$:
In this case, 
$$(p-1)\M(\beta) \ge (p-5)(3/2 \cdot 3 + 2) + 6,$$
which implies that $\M(\beta) \ge 4.5 > 23/6$.
\item $\M(\gamma) \ge 3/2$, $\M(\delta) = 1$:
In this case,
$$(p-1)\M(\beta) \ge (p-5)(3/2 \cdot 4 + 1) + 4,$$
and thus $\M(\beta) \ge 4.6 > 23/6$.
\item $\M(\alpha_i) \ge 3/2$ for all $i$:
In this case,
$$(p-1) \M(\beta) \ge (p-5)(3/2 \cdot 5),$$
and hence $\M(\beta) \ge 4.5 > 23/6$.
\end{enumerate}

\subsection{The case when $X = 6$, $p \ge 11$, and $\lambda = 0$}
If $X = 6$, then Lemma~\ref{lemma:bound} no longer applies when $p = 11$.
We consider this possibility at the end of this subsection. Thus, we assume
that 
$$\beta = \alpha_i \zeta^i + \alpha_j \zeta^j + \alpha_k \zeta^k
+ \overline{\alpha_i} \zeta^{-i} + \overline{\alpha_j} \zeta^{-j}
+ \overline{\alpha_k} \zeta^{-k}.$$
 Again,
we break up into subcases.

\subsubsection{$X = 6$, all the $\alpha_i$ are roots of unity}
If all the $\alpha_i$ are the same, they must be 
(after changing the sign 
of $\beta$ if necessary) equal to $1$. We compute in this case that
$$(p-1) \M(\beta) = (p-6)6.$$
If $p \ne 11,13$, then $\M(\beta) \ge  4.125 > 23/6$.
Otherwise, we may write
$$\beta = 2 \cos(2 \pi i/p) + 2 \cos(2 \pi j/p) + 2 \cos(2 \pi k/p).$$
Note that $(i,p) = (j,p) = (k,p) = 1$.  Without loss of generality, we may
assume that $i = 1$. The smallest value of
$\ho{\beta}$ for $p = 11$ or $p = 13$ of this kind may
easily be computed to be
$$-2 (\cos(4 \pi/11) + \cos(8 \pi/11) + \cos(12 \pi/11)) = 2.397877\ldots$$
$$\frac{1 + \sqrt{13}}{2} = -2 (\cos(4 \pi/13) + \cos(12 \pi/13) + \cos(16 \pi/13)) = 2.302775\ldots$$
the former of which is larger than $76/33$, the
latter which is on our list. The second smallest number for $p = 13$ is
$3.148114\ldots > 76/33$.

Suppose that one of the $\alpha_i$ is not real. Then $\alpha_i$ is
certainly distinct from $\overline{\alpha_i}$, and
either $\alpha_j \ne \overline{\alpha_j}$ or $\alpha_j$ and $\overline{\alpha_j}$
are both distinct from $\alpha_i$ and $\overline{\alpha_i}$, and similarly with
$k$. It follows that there are at least $9$ pairs of numbers which  are distinct,
the minimum occurring when $\alpha_i = \alpha_j = \alpha_k$ 
or when $\alpha_j = \alpha_k = \pm 1$. In either case, we find that
$$(p-1) \M(\beta) \ge (p-6)6 + 9,$$
and hence $\M(\beta) \ge 3.9 > 23/6$.

Finally, suppose that all the $\alpha_i$ are real, but that they are not all equal. Then, up to sign, 
$$\beta =  2 \cos(2 \pi i/p) + 2 \cos(2 \pi j/p) - 2 \cos(2 \pi k/p).$$
In this case, we compute that $(p-1) \M(\beta) \ge (p-6)6 + 8$,
which is larger than $23/6$ if $p \neq 11$. If $p = 11$, we enumerate the possibilities
directly, and find that the smallest value of $\ho{\beta}$ is
$$2 \cos(2 \pi/11) - 2 \cos(8 \pi/11) -  2 \cos(16 \pi/11) = 3.276858\ldots > 76/33.$$

\subsubsection{$X = 6$, and $\M(\alpha_i) \ge 3/2$}
If $\M(\alpha_j) \ge 3/2$ also then
$$(p-1) \M(\beta) \ge (p-6)8,$$
and hence $\M(\beta) \ge 4$. Thus we may assume that 
$\M(\alpha_j) = \M(\alpha_k) = 1$. In this case, there are clearly
at least $8$ distinct pairs, and thus
$$(p-1) \M(\beta) \ge (p-6)7 + 8,$$
and hence $\M(\beta) \ge 4.3 > 23/6$.

\subsection{The case when $X \ge 7$, and $p \ge 11$}
Note that we make no assumptions on $\lambda$ in this case.
Write $\beta = \sum_S \alpha_i \zeta^i$.
From Equation~\ref{equation:two}, we deduce that
$$(p-1) \M(\beta) \ge X(p-X) + \frac{1}{2} \left(X^2 - \sum X^2_j\right).$$
If $p \ge 13$, then by Lemma~\ref{lemma:thirteen}, we may assume
that $X \le (p-1)/2$. In particular, this implies that $p \ge 17$. In this case,
the inequality 
$$(p-1) \M(\beta) \ge X(p-X)$$
already implies that $\M(\beta) \ge 4.375 > 23/6$.
Hence we may reduce to the case when $p = 11$.
By Lemma~\ref{lemma:zero},
 we may assume that $X_j \le 11 - X$.
We consider the various possibilities:

\begin{enumerate}
\item Suppose that $X = 7$. Then $X_j \le 4$, and hence
$\sum X^2_j \le 25$, and
$$10 \M(\beta) \ge 7(11 - 7) + \frac{1}{2} \left(49 - 25\right) = 40,$$
and $\M(\beta) \ge 4 > 23/6$.
\item Suppose that $X = 8$. Then $X_j \le 3$, and hence
$\sum X^2_j \le 22$, and
$$10 \M(\beta) \ge 8(11 - 8) + \frac{1}{2} \left(64 - 22\right) = 45,$$
and $\M(\beta) \ge 4.5 > 23/6$.
\item Suppose that $X = 9$. Then $X_j \le 2$, and hence
$\sum X^2_j \le 17$, and
$$10 \M(\beta) \ge 9(11 - 9) + \frac{1}{2} \left(81 - 17\right) = 50,$$
and $\M(\beta) \ge 5 > 23/6$.
\item Suppose that $X = 10$. Then $X_j \le 1$, and hence
$\sum X^2_j \le 10$, and
$$10 \M(\beta) \ge 10(11 - 10) + \frac{1}{2} \left(100 - 10\right) = 55,$$
and $\M(\beta) \ge 5.5 > 23/6$.
\end{enumerate}

\subsection{The case when $X = 6$, $p = 11$, and $\lambda \ne 0$}
Write $\beta = \sum_S \alpha_i \zeta^i$. Since $\lambda \ne 0$, it
must be the case that either $\alpha_i$ or $\alpha_{-i}$ is non-zero.
Moreover, by cardinality reasons, at least one of these must be zero,
and hence $\lambda = \alpha_ i - \overline{\alpha_{-i}} = \alpha_i$.
Thus, in this case, it must be the case that
$$\beta = \alpha + \lambda \sum_{T}\zeta^i,$$
where $T$ is a subset of $S$ of cardinality $5$ such that
$T \cup \{-T\} \cup \{0\} = S$. Moreover, $\alpha - \overline{\alpha} = \lambda$,
and $\lambda \cdot \sqrt{-1}$ is real.
If $\lambda$ is not a root of unity, then
$$10 \M(\beta) \ge (11 - 6)(3/2 \cdot 5 + 1),$$
and hence $\M(\beta) \ge 4.25 \ge 23/6$.
Hence $\lambda$ is a root of unity, which must be equal (after changing the
sign of $\beta$) to $\sqrt{-1}$.
Clearly $\alpha$ is not equal to $\sqrt{-1}$.
Hence
$$\M(\beta) \ge (11 - 6)(5 + \M(\alpha)) + 5 = 30 + 5 \M(\alpha).$$
It follows that $\M(\alpha) < 8/5 < 2$, and thus $\alpha$ is the sum
of at most two roots of unity.
If $\alpha$ is a root of unity, then $\overline{\alpha} = \alpha^{-1}$,
and hence
$$\alpha -  \alpha^{-1} = \lambda = \sqrt{-1}.$$
This implies that $\alpha = \zeta_{12}$ or $\zeta^5_{12}$.
 In this case we may check every possibility
for $\beta$ (the set of possible $T$ has cardinality $2^5$ since it
requires a choice of one of $\{i,-i\}$ for each non-zero $i \mod 11$),
 and the smallest such (largest conjugate) is:
 $$\zeta_{12} + \zeta_4 \left( \zeta_{11}^{-1} + \zeta_{11}^{2} + \zeta_{11}^{-3} + \zeta_{11}^{-4}
 + \zeta_{11}^{-5} \right) = 2.524337\ldots > 2.303030\ldots = 76/33.$$
 Suppose that $\Num(\alpha) = 2$. Then either
 $\M(\alpha) = 3/2$ and $\alpha$ is a root of unity times
 $(1 + \sqrt{5})/2$, or $\M(\alpha) \ge 5/3 > 8/5$.
 Hence we may now assume that $\alpha = 
 (1 + \sqrt{5})/2 \cdot \xi$ for a root of unity $\xi$. We now obtain the equation
 $$ \left(\frac{1 + \sqrt{5}}{2} \right) (\xi - \xi^{-1}) = \sqrt{-1}.$$
 From this equation we deduce that
 $\xi = \zeta_{20}$ or $\zeta^{9}_{20}$.
  Again, we check the possibilities for $\beta$, the smallest being:
 $$ \left(\frac{1 + \sqrt{5}}{2} \right) \zeta_{20}
 + \zeta_4 \left(\zeta_{11}^{-1} + \zeta_{11}^{2} + \zeta_{11}^{-3} + \zeta_{11}^{-4}
 + \zeta_{11}^{-5} \right) =  3.197154\ldots > 76/33.$$
 
This completes the proof of Theorem~\ref{theorem:firstbound}.
 
  \section{An analysis of the field $\Q(\zeta_{84})$} \label{section:cassels3}
  
In order to progress further, we require some more precise analysis
of certain  elements $\alpha$ in the field $\Q(\zeta_{84})$ with $\M(\alpha)$ small.

\begin{lemma} Suppose that $\alpha \in \Q(\zeta_{7})$ satisfies \label{lemma:bound7}
$\M(\alpha) \le 4$. Then, up to sign and rescaling by a $7$th root of unity,
either:
\begin{enumerate}
\item $\alpha = 0$ or $\alpha = 1$, and $\M(\alpha) = 0$ or $1$.
\item $\alpha = 1 + \zeta_7^i$ with $i \ne 0$, and $\M(\alpha) = 5/3$.
\item $\alpha = 1 - \zeta_7^i$ with $i \ne 0$, and $\M(\alpha) = 7/3$.
\item $\alpha = 1 + \zeta_7^i + \zeta_7^j$ with $(i,j)$ distinct and non-zero, and
$\M(\alpha) = 2$.
\item $\alpha = 1 + \zeta_7^i - \zeta_7^j$ with $(i,j)$ distinct and non-zero, and
$\M(\alpha) = 10/3$.
\item $\alpha = 2$ and $\M(\alpha) = 4$,
\item $\alpha = \zeta_7^i  + \zeta_7^j + \zeta_7^k - 1$ with $(i,j,k)$ distinct and non-zero,
and $\M(\alpha) = 4$.
\end{enumerate}
\end{lemma}

\begin{proof} Write $\alpha = \sum a_i \zeta_7^i$, where $a_i \in \Z$.
We may assume that all the $a_i$ are non-negative, and that at least one
$a_i$ is equal to $0$.
Suppose that $A_i$ of the $a_i$ are equal to $i$. 
Then
$$6 \M(\alpha) = \sum (a_i - a_j)^2 =  \sum (i - j)^2 A_i A_j.$$
Suppose that $\M(\alpha) \le 4$. 
From the inequality $48 \ge 12 \M(\alpha) \ge n^2 A_n A_0$, we deduce that
$A_n = 0$ if $n \ge 7$.
It is easy to enumerate the partitions
of $7 = \sum A_i$ satisfying the inequality $24 \geq \sum(i-j)^2 A_i A_j$. We write $A$ as $(A_0, A_1, \ldots)$, showing only up until the last nonzero value, and find  a strict inequality for
$$A \in \{
  (7),(1,6),(2,5), (3,4), (4,3), (5,2), (6,1),
 (2,4,1), (1,5,1), (1,4,2)\}$$
 (giving cases (1), (1), (2), (4), (4), (2), (1), (5), (3) and (5) of the statement, respectively)
 and equality for
 $A \in \{
 (6,0,1),
(3,3,1), (1,3,3), (1,0,6)\}$ (giving cases (6),(7),(7) and (6) of the statement, respectively).
The result follows.
\end{proof}

\begin{corr} \label{corr:bound7}
Suppose that $\alpha \in \Q(\zeta_{7})$ satisfies $\Num(\alpha) \ge 4$, then $\M(\alpha)  \ge 4$.
\end{corr}


\begin{lemma} Suppose that $\alpha \in \Q(\zeta_{21})$ satisfies \label{lemma:bound21}
$\M(\alpha) < 17/6$. Then, up to sign and a $21$st root of unity,
either:
\begin{enumerate} 
\item $\alpha$ is a sum of at most three roots of unity.
\item $\alpha$ lies in the field $\Q(\zeta_{7})$,
\item $\alpha = \zeta_7^i + \zeta_7^j + \zeta_7^k - \zeta_3$ where
$(i,j,k)$ are distinct and non-zero, and $\M(\alpha) = 5/2$.
\item $\alpha = 1 + \zeta_7^i  - (\zeta_7^j + \zeta_7^k) \zeta_3$ where
$(i,j,k)$ are distinct and non-zero, and $\M(\alpha) = 8/3$.
\item $\alpha = \zeta_7^i + \zeta_7^j + (\zeta_7^j + \zeta_7^k) \zeta_3$ where
$(i,j,k)$ are distinct, and $\M(\alpha) = 8/3$.
\end{enumerate}
\end{lemma}

\begin{proof}
We may write $\alpha = \gamma + \delta \zeta_3$, where
$$\M(\alpha) = \frac{1}{2} ( \M(\gamma) + \M(\delta) + \M(\gamma - \delta)).$$
We may assume that $\gamma \ne \delta$, since otherwise $\alpha = - \zeta_3^2 \gamma$ is,
up to a root of unity, in $\Q(\zeta_{7})$, giving case (2). In general, we note that
$\alpha = (\gamma - \delta) - \delta \zeta_3^2 =
(\delta - \gamma) \zeta_3 - \gamma \zeta_3^2$,
Hence, after re-ordering if necessary, we may assume that
$$\Num(\gamma-\delta) \ge \Num(\gamma) \ge \Num(\delta).$$
Assume that $\M(\alpha) \le 17/6$. 
If $\Num(\delta) \ge 3$, then $\M(\gamma - \delta)$,
$\M(\gamma)$, and $\M(\delta)$ are all $\ge 2$, and thus
$\M(\alpha) \ge 3$, a contradiction.
We consider various other cases.
\begin{enumerate}[(i)]
\item $\Num(\delta) = 1$ and $\Num(\gamma) \le 2$:
In this case, $\Num(\alpha) \le 3$, giving case (1).
\item $\Num(\delta) = 1$ and $\Num(\gamma) = 3$: If
$\Num(\gamma - \delta) \geq 4$, then
$\M(\alpha) \ge (1 + 2 + 10/3)/2 \ge 19/6$.
Thus $\Num(\gamma - \delta) = 3$.
In particular,
$$(\delta - \gamma) + (\gamma) - (\delta) = 0.$$
is a vanishing sum of length $3+1+3$.
The only primitive vanishing sums in $\Q(\zeta_7)$ have length
$7$ or $2$. Thus, the expression above must be a multiple of 
the vanishing sum
$$1 + \zeta_7 + \zeta_7^2 + \zeta_7^3 + \zeta_7^4 + \zeta_7^5 + \zeta_7^6 + \zeta_7^7 = 0.$$
After scaling, we may assume that $\delta = -1$, and thus
$\gamma = \zeta_7^i + \zeta_7^j + \zeta_7^k$ for some triple
$(i,j,k)$ that are all distinct and non-zero. Since
$\delta - \gamma$ is  sum of $3$ distinct $7$th roots of unity in this case,
we deduce that $\M(\gamma) = \M(\delta - \gamma) = 2$, and hence
$\M(\alpha) = 5/2$. We are thus in case (3).
\item $\Num(\delta) = 1$ and $\Num(\gamma) \ge 4$:
It follows immediately that $\M(\alpha) \ge (1 + 10/3 + 10/3)/2 =23/6$, a contradiction.
\item $\Num(\delta) = 2$ and $\Num(\gamma) = 2$:
If $\Num(\delta - \gamma) \geq 4$, then
$\M(\alpha) \ge (5/3 + 5/3 + 10/3) = 20/6$.
If $\Num(\delta - \gamma) = 3$, we obtain a vanishing
sum
$$(\delta - \gamma) + (\gamma) - (\delta) = 0$$
of length $7$, and hence 
$\gamma = \zeta_7^i + \zeta_7^j$ and $\delta = -(\zeta_7^k + \zeta_7^{l})$,
where $(i,j,k,l)$ are all distinct. 
In this case, $\M(\gamma) = \M(\delta) = 5/3$, and
$\gamma - \delta$ is minus a sum of three distinct $7$th
roots of unity, and so $\M(\gamma - \delta) = 2$. It
follows that
$\M(\alpha) =  8/3$ and we are in case (4). If $\Num(\delta - \gamma) = 2$, then
the above sum is a vanishing sum of length $6$. It
follows that it is composed of vanishing subsums of
length $2$, from which it easily follows that
$\delta = \zeta_7^j + \zeta_7^k$ and $\gamma =\zeta_7^i + \zeta_7^j$.
In this case, $\M(\delta) = \M(\gamma) = 5/3$, and
$\M(\delta - \gamma) = 2$, and thus
$\M(\alpha) = 8/3$, giving case (5).
\item $\Num(\delta) = 2$ and
$\Num(\gamma) \ge 3$: It follows immediately that
$\M(\alpha) \ge (5/3 + 2 + 2)/2 = 17/6$, a contradiction.
\end{enumerate}
\end{proof}

\begin{corr} \label{corr:9fourths}
Suppose that $\alpha \in \Q(\zeta_{21})$ satisfies $\M(\alpha) < 9/4$ and $\Num(\alpha) \geq 3$, then $\alpha = 1 + \zeta^i_7 + \zeta^j_7$ where $(i,j)$ are distinct and non-zero and $\M(\alpha) = 2$.
\end{corr}

\begin{lemma}
Suppose that $\alpha \in \Q(\zeta_{21})$ satisfies \label{lemma:bound21b}
$\M(\alpha) < 23/6$, then $\Num(\alpha) \le 5$.
\end{lemma}

\begin{proof}
As before we may write $\alpha = \gamma + \delta \zeta_3$
and we may assume that $\Num(\gamma - \delta) \ge \Num(\gamma) \ge \Num(\delta)$.
If $\Num(\delta) \le 2$, then we are done unless $\Num(\gamma - \delta) \ge \Num(\gamma) \ge 4$.
In this case, we  deduce from Corollary~\ref{corr:bound7} that
$\M(\gamma - \delta) \ge 4$ and $\M(\gamma) \ge 4$, from which it follows
directly that $\M(\alpha) \ge (4+4+1)/2 > 23/6$.
Suppose that $\Num(\delta) \ge 3$.
If  $\Num(\delta- \gamma) \ge 4$, then
$\M(\alpha) \ge (2+2+4)/2 = 4 > 23/6$. Thus, we may assume that
$$\Num(\delta) = \Num(\gamma) = \Num(\delta - \gamma) = 3.$$
Let us consider the resulting vanishing sum
$$(\delta - \gamma) + (\gamma) - (\delta) = 0.$$
It has length $9 = 7 + 2$. After scaling $\alpha$ by a root of unity,
we may assume that this sum is (having re-arranged the order of the roots of unity):
$$(1 + \zeta_7 + \zeta_7^2 + \ldots + \zeta_7^6) + (1 - 1) = 0.$$
At least one of the three terms must be contained within the first sum. 
Furthermore, the $(1-1)$ sum cannot be contained within a single term.
Hence,  we obtain the following two possibilities (up to symmetry):
$$\begin{aligned}
\gamma =  & \ 1 + \zeta_7^i + \zeta_7^j, 
& \ \delta =  1 - \zeta_7^k - \zeta_7^l,  \quad \delta - \gamma = & \ 1 + \zeta_7^m + \zeta_7^n, \\
\gamma = & \  2 + \zeta_7^i,& \ \delta =   1 - \zeta_7^j - \zeta_7^k,
\quad \delta - \gamma  =  & \ \zeta_7^l + \zeta_7^m + \zeta_7^n.\end{aligned}$$
where $(i,j,k,l,m,n)$ are distinct and non-zero.
In the first case, we notice that since $1+\zeta_3 = -\zeta_3^2$, in fact $\Num(\alpha) \leq 5$. 
In the second case, we compute that
$\M(\alpha) = (13/3 + 10/3 + 2)/2  = 29/6>23/6$.
\end{proof}

\begin{lemma}
Suppose that $\alpha \in \Q(\zeta_{21})$ satisfies \label{lemma:bound21c}
 $\Num(\alpha) = 2$, then $\M(\alpha) \ge 2$, or $\M(\alpha) = 5/3$.
\end{lemma}
\begin{proof}
Again we write $\alpha = \gamma + \delta \zeta_3$.  If either $\gamma$ or $\delta$ is zero, then up to a root of unity $\alpha \in \Q(\zeta_7)$ and we can apply Lemma~\ref{lemma:bound7}.  If neither $\gamma$ nor $\delta$ is zero, then they must both be roots of unity, hence, $\M(\alpha) = (2 + \M(\gamma - \delta))/2$.
Notice that $\gamma -\delta$ is not a root of unity, because there are no
vanishing sums $$(\gamma - \delta) + (\delta) - (\gamma) = 0$$
of length $3$ in $\Q(\zeta_7)$.  Since $\alpha$ is not a root of unity, $\gamma \neq \delta$, and hence $\M(\alpha) = (2 + \M(\gamma - \delta))/2 \geq 2$.
\end{proof}

\begin{lemma} \label{lemma:lazier}
Suppose that $\alpha \in \Q(\zeta_{84})$, that $\M(\alpha)<9/4$, and that $\Num(\alpha) \geq 3$, then $\alpha = \zeta_{84}^i(1 + \zeta^j_7 + \zeta^k_7)$.
\end{lemma}
\begin{proof}
Write $\alpha = \gamma + \zeta_4 \delta$.  Since $\Num(\alpha) \geq 3$ it follows that one of $\gamma$ or $\delta$ is not a root of unity.  If $\gamma$ and $\delta$ are both nonzero, then $\M(\beta) \ge 1+3/2 >9/4$, hence $\gamma$ or $\delta$ is zero, and up to a root of unity $\alpha \in \Q(\zeta_{21})$.  The result then follows from Corollary~\ref{corr:9fourths}.
\end{proof}

\begin{lemma}
\label{lemma:bound84}%
The elements $\alpha \in \Q(\zeta_{84})$ such
 that $\M(\alpha) < 17/6$ are, \label{lemma:bound84}
up to roots of unity, either a sum of at most $3$ roots of unity,
or are, up to a root of unity, one of the exceptional forms in $\Q(\zeta_{21})$,
specifically:
\begin{enumerate} 
\item $\alpha = \zeta_7^i + \zeta_7^j + \zeta_7^k - \zeta_3$ where
$(i,j,k)$ are distinct and non-zero, and $\M(\alpha) = 5/2$.
\item $\alpha = 1 + \zeta_7^i  - (\zeta_7^j + \zeta_7^k) \zeta_3$ where
$(i,j,k)$ are distinct and non-zero, and $\M(\alpha) = 8/3$.
\item $\alpha = \zeta_7^i + \zeta_7^j + (\zeta_7^j + \zeta_7^k) \zeta_3$ where
$(i,j,k)$ are distinct and non-zero, and $\M(\alpha) = 8/3$.
\end{enumerate}
Moreover, if $\Num(\alpha) = 2$, then either $\M(\alpha) \ge 2$ or $\M(\alpha)=5/3$.
\end{lemma}

\begin{proof}
If $\alpha = \gamma +  \delta \zeta_4$ with $\gamma, \delta \in \Q(\zeta_{21})$, then $\M(\alpha) = \M(\gamma) + \M(\delta)$.  If $\gamma = 0$ or $\delta = 0$ the problem reduces immediately to Lemma~\ref{lemma:bound21}.  So we may assume that  $\gamma \ne 0$ and $\delta \ne 0$.  By symmetry,  we may assume that $\M(\gamma) \ge \M(\delta) \ge 1$. It follows that
$\M(\gamma) <  11/6 < 2$, and hence $\Num(\gamma) \le 2$.
If $\Num(\delta) = \Num(\gamma) = 2$, then $\M(\alpha) \ge 10/3$.
If $\Num(\alpha) = 2$, then either $\gamma$ and $\delta$ are non-zero, in which
case $\M(\alpha) = 2$, or we may assume that $\alpha \in \Q(\zeta_{21c})$, and
apply Lemma~\ref{lemma:bound21}.
\end{proof}

\begin{lemma} Suppose that $\alpha \in \Q(\zeta_{84})$. 
Then either $\M(\alpha) \ge 23/6$, or $\Num(\alpha) \le 5$. \label{lemma:ubound84}
\end{lemma}

\begin{proof} Assume that $\M(\alpha) < 23/6$.
Write $\alpha = \gamma + \delta \zeta_4$. If $\gamma$ and $\delta$ are both non-zero,
then we may assume that $17/6 > \M(\gamma) \ge \M(\delta) \ge 1$.
Suppose that $\Num(\delta) \ge 2$. Then $\M(\delta) \ge 5/3$, and
hence
$\M(\gamma) \le 13/6 < 5/2$, from which we deduce from
Lemma~\ref{lemma:bound84} that $\Num(\gamma) \le 3$,
and hence $\Num(\alpha) \le 5$.
Suppose that $\Num(\delta) = 1$. Since $\M(\gamma) \le 17/6$, we see that
$\Num(\gamma) \le 4$ and $\Num(\alpha) \le 5$.
Thus we may assume that one of $\gamma$ or $\delta$ is zero,
and hence, up to  a root of unity, $\alpha \in \Q(\zeta_{21})$.
The result follows by Lemma~\ref{lemma:bound21b}.
\end{proof}

\begin{corr} Suppose that $\beta \in \Q(\zeta_{84})$ is real. \label{corr:cor84}
Then either  $\ho{\beta} \ge 76/33$, $\Num(\beta) \le 2$, or $\beta$
is either a conjugate of  $\ep$ or  $1 + 2 \cos(2 \pi/7)$.
\end{corr}

\begin{proof}
The result is an immediate
consequence of Lemma~\ref{lemma:ubound84},  combined with
Corollary~\ref{corr:useful} and Lemma~\ref{lemma:bound}.
\end{proof}

\setcounter{theorem}{0}
\section{Final reductions} \label{section:finalreduction} \label{section:cassels4}
  In this section, we complete the proof of Theorem~\ref{theorem:main}
  by proving the following.
  
  \begin{theorem}
  If $\beta$ is a real cyclotomic integer such that $\beta \in \Q(\zeta_{420})$, $\Num(\beta) \ge 3$,  and $\ho{\beta} < 76/33$, then either $\beta \in \Q(\zeta_{84})$,
  or $\ho{\beta} = \sqrt{5}$ or $(1 + \sqrt{5})/\sqrt{2}$.
  \end{theorem}
  
The technique used in this section is to apply the style of arguments from Cassels ``first case" which we used in Section~\ref{section:cassels2} applied to the prime $5$.  The arguments are much more detailed than those in Section~\ref{section:cassels2} and we exploit our understanding of small numbers in $\Q(\zeta_{84})$.  As in Section~\ref{section:cassels2} we will use $\zeta$ to denote an arbitrary $p$th root of unity, and in this section $p=5$.  Recall that, on the other hand, $\zeta_5$ denotes the particular $5$th root of unity $e^{2 \pi i /5}$.


Note that if $\Num(\beta) \le 5$,
the result follows from Corollary~\ref{corr:useful}.
We consider various cases in turn.

\subsection{The case when $X = 1$ and $p = 5$}
The same proof in \S\ref{section:firstcase} holds verbatim.

\subsection{The case when $X = 2$ and $p = 5$}
Since $p = 5$, we may assume by Lemma~\ref{lemma:balanced} that $\lambda = 0$,
and hence $\beta = \zeta \alpha + \zeta^{-1} \overline{\alpha}$.
Suppose that $\ho{\alpha} \ge \sqrt{3}$.
Then,
as in \S\ref{section:firstcase}, we deduce that
$$\ho{\beta} \ge 2 \ho{\alpha} \cos(\pi/5) \ge 2 \sqrt{3} \cos(\pi/5) =  2.802517\ldots > 2.303030\ldots
= 76/33.$$
It follows immediately from
 Lemma~6 of Cassels~\cite{MR0246852} that  if $\ho{\alpha} < \sqrt{3}$, then
either $\Num(\alpha) \le 2$, or $\alpha$ is a root of unity times one
of
$$ \frac{1}{2} \left(1 + \sqrt{-7} \right),  
\qquad \frac{1}{2} \left( \sqrt{-3} + \sqrt{5} \right).$$
If $\Num(\alpha) \le 2$, then $\Num(\beta) \le 4$ and we are done.
Suppose that, up to a root of unity,
$\alpha$ is one of the two exceptional cases.
Since  $\alpha \in \Q(\zeta_{84})$, only the first possibility may occur.
Writing $\alpha$  
as a root of unity times $(1 + \sqrt{-7})/2$
and enumerating all possibilities,
the smallest possible element thus obtained is
$$ \left|    \frac{\sqrt{7} + \sqrt{-1}}{2} \cdot \zeta_5^2
+    \frac{\sqrt{7} - \sqrt{-1}}{2} \cdot \zeta_5^{-2} \right|
= \frac{1}{2} \sqrt{13 + 3 \sqrt{5} + \sqrt{14 (5 + \sqrt{5})}} = 2.728243\ldots > 76/33.$$

\subsection{The case when $X = 3$, $p = 5$, and $\lambda = 0$}
We have that $\beta = \zeta \alpha + \gamma + \zeta^{-1} \overline{\alpha}$.
From equation~\ref{equation:two}, we deduce that
$$\begin{aligned}
4 \M(\beta) = & \  2 \M(\alpha) + 2 \M(\overline{\alpha}) + 2 \M(\gamma)
+ \M(\alpha - \overline{\alpha}) + \M(\alpha - \gamma) + \M(\overline{\alpha} - \gamma) \\
 = & \ 4 \M(\alpha) + 2 \M(\gamma) + 2 \M(\alpha - \gamma) + 
 \M(\alpha - \overline{\alpha}). \end{aligned}$$
We consider various subcases.

\subsubsection{$X = 3$, $p = 5$, $\lambda = 0$, and $\alpha = \gamma$}
We deduce that $\alpha$ is real, and hence
$\beta = \alpha (\zeta + 1 + \zeta^{-1})$. It follows that
$$\ho{\beta} = \ho{\alpha} \cdot 
\hr{{1 + \zeta + \zeta^{-1}}} = 2 \cos(\pi/5) \ho{\alpha} > 76/33$$
if $\ho{\alpha} \ge 2$. Thus $\ho{\alpha} = 2 \cos(\pi/n)$ for some $n|84$,
and we quickly determine that the only $\ho{\beta}$ in
the range $[2,76/33]$ is $(\sqrt{5}+1)/\sqrt{2}$.

\subsubsection{$X = 3$, $p = 5$, $\lambda = 0$, $\alpha \ne \gamma$, 
$\Num(\gamma) \le 2$, 
$\Num(\alpha) \ge 3$, and $\alpha$ is not real}
Since $\alpha$ is not real, $\M(\alpha - \overline{\alpha}) \ge 1$.
Since $\Num(\alpha) \ge 3$, 
if $\Num(\gamma) = 1$ then $\Num(\alpha - \gamma) \ge 2$, whereas
if $\Num(\alpha - \gamma) = 1$ then $\Num(\gamma) \ge 2$.
Thus
$$4\M(\beta) \ge 4\M(\alpha) + 2 \left( \frac{5}{3} + 1 \right) + 1,$$
and hence $\M(\alpha) <  9/4$. It follows from Lemma~\ref{lemma:lazier}
(and the fact that $\alpha \in \Q(\zeta_{84})$) that
$\alpha = \zeta_{84}^i(1 + \zeta^j_7 + \zeta^k_7)$. Moreover, we may assume
that either $\gamma = 1$ or $\gamma = \zeta_{84}^l + \zeta_{84}^{-l}$ for some $l$.
Enumerating all possibilities with $\alpha = \zeta_{84}^i(1 + \zeta^j_7 + \zeta^k_7)$
(without the assumption that $\alpha$ is not real), we find that the smallest largest conjugate
is:
$$ 2 \cos(\pi/5)(1 + 2 \cos(2 \pi/7)) -  1 = 2.635689\ldots > 76/33.$$

\subsubsection{$X = 3$, $p = 5$, $\lambda = 0$, $\alpha \ne \gamma$,
$\Num(\gamma) \le 2$, 
$\Num(\alpha) \ge 3$, and $\alpha$ is real}
Suppose that $\Num(\gamma)$ and $\Num(\alpha - \gamma)$ are both
at least two. It follows from Lemma~\ref{lemma:lazier}  that
$\alpha = \zeta_{84}^i(1 + \zeta^j_7 + \zeta^k_7)$, which was considered
above.
Thus, we may assume that at least one of $\Num(\gamma)$ or $\Num(\alpha - \gamma)$
equal to one. We show that $\Num(\alpha) \le 4$.
If $\Num(\gamma) = 1$, and $\Num(\alpha) \ge 5$, then $\Num(\alpha - \gamma) \ge 4$,
and thus $\M(\alpha)$ and $\M(\alpha - \gamma)$ are $\ge 8/3$ by
Lemma~\ref{lemma:bound84}. Yet then
$$\M(\beta) \ge 8/3 +  (8/3 + 1)/2 = 9/2 > 23/6.$$
Conversely, if $\Num(\alpha - \gamma) = 1$, then by assumption, $\Num(\gamma)  \le 2$,
and so $\Num(\alpha) \le 3$.
It follows by Lemma~\ref{lemma:84} that we assume that $\alpha$ is one of the following forms, up to sign:
\begin{enumerate}
\item $1 + \zeta_{84}^{i} + \zeta_{84}^{-i}$,
\item $\zeta_{84}^{i} + \zeta_{84}^{-i} + \zeta_{84}^{j} + \zeta_{84}^{-j}$,
\item $\zeta_{84}^{-9} + \zeta_{84}^{-7} + \zeta_{84}^{3} + \zeta_{84}^{15}$,
\item $\zeta_{84}^{-9} + \zeta_{84}^{-7} + \zeta_{84}^{3} + \zeta_{84}^{27}$,
\end{enumerate}
whereas we may assume that $\gamma = \zeta_{84}^k + \zeta_{84}^{-k}$.
(Here we are using the fact that $\alpha \in \Q(\zeta_{84})$ to eliminate some
of the other exceptional possibilities in Lemma~\ref{lemma:84}.) In cases $3$ and $4$ every $\beta$
has a conjugate of absolute value 
at least $3$. In the first two cases, $\sqrt{5}$ occurs as a (degenerate) possibility for
$\beta$. The second smallest largest conjugate is also degenerate, and occurs with $\alpha = 2$ and
$\gamma = 1$, where $\ho{\beta} = 2 + 2 \cos(2 \pi/5) = 2.618033\ldots > 76/33$.

\subsubsection{$X =3$, $p = 5$,  $\lambda = 0$, $\alpha \ne \gamma$,
$\Num(\gamma) \le 2$, and $\Num(\alpha) \le 2$} We
may let $\alpha = \zeta_{84}^i + \zeta_{84}^j$ and $\gamma = \zeta_{84}^k + \zeta_{84}^{-k}$.
The smallest such largest conjugate (besides a degenerate  $\sqrt{5}$) is
$$   4 \cos(\pi/5) \cos(3 \pi/7) + 2 \cos(\pi/7) = 2.522030\ldots >
2.303030\ldots = 76/33.$$

\subsubsection{$X = 3$, $p = 5$, $\lambda = 0$, and $\Num(\gamma) \ge 3$}
By Corollary~\ref{corr:cor84}, we may assume that
either $\gamma = \ep$,  $1 + 2 \cos(2 \pi/7)$, 
 or  $\gamma  = \ho{\gamma} \ge 76/33$. 
In the latter case,  we choose a conjugate
of $\zeta$ such that $\zeta \alpha + \zeta^{-1} \overline{\alpha} > 0$, and then
$\beta > \gamma > 76/33$.
Since $\M(\ep) = 5/2$ and $\M(1 + 2 \cos(2 \pi/7)) = 2$, we may deduce that
$\M(\gamma) \ge 2$. Thus
$$4 \M(\beta) \le 4 \M(\alpha) + 4 + 2 \M(\alpha - \gamma) + \M(\alpha - \overline{\alpha}).$$
The case $\gamma = \alpha$ has already been considered. Thus
$\M(\alpha - \gamma) \ge 1$, and hence, since $\M(\beta) < 23/6$, we deduce
that $\M(\alpha) < 7/3$. By Lemma~\ref{lemma:bound84},
it  follows that  $\Num(\alpha) \le 3$.
Enumerating over all $\alpha$ with $\Num(\alpha) \le 3$ and
$\gamma = \ep$ or $1 + 2 \cos(2 \pi/7)$, 
all the smallest conjugates (with $\alpha \ne 0$) are at
least $3$, except
for
$$1 + 2 \cos(2 \pi/7) + 2 \cos(2 \pi/5) = 2.865013\ldots > 76/33.$$

\subsection{The case when $X = 3$, $p = 5$, and $\lambda \ne 0$}
It follows, choosing $\zeta$ appropriately, that
$$\beta = \alpha + \lambda (\zeta + \zeta^2),$$
where, as usual, $\alpha - \overline{\alpha} = \lambda$. 
We do a brute force computation for all $\alpha$ with
$\Num(\alpha) \le 3$. Note that if $\Num(\alpha) = 3$,
we may assume that
$\alpha = \zeta_{84}^i + \zeta_{84}^j + \zeta_{84}^k$ where $i$ is a divisor of $84$.
The smallest resulting largest conjugate that
arises is
$$\zeta_{84}^{7} + (\zeta_{84}^{7} - \zeta_{84}^{-7})(\zeta^{3} + \zeta^{4}) = 2 \cos(\pi/30) +
2 \cos(13 \pi/30) = 2.404867\ldots \ge 2.303030\ldots = 76/33.$$
We note that
$$4 \M(\beta) = (5 - 3)(\M(\alpha) + 2 \M(\lambda)) + 2 \M(\alpha -\lambda).$$
Since $\alpha - \lambda = \overline{\alpha}$, we may write this as
$$\M(\beta) = \M(\alpha) + \M(\lambda).$$
Since $\lambda \ne 0$, it follows that $\M(\beta) < 17/6$.
We deduce by Lemma~\ref{lemma:bound84}
that either $\Num(\alpha) \le  3$, or $\alpha$ is one of three
specific forms given in that lemma, that is, we may assume that
$\alpha$ is, up to a root of unity, one of the following:
\begin{enumerate}
\item $\alpha = \zeta_{84}^{n} (\zeta^i_7 + \zeta^j_7 + \zeta^k_7 - \zeta_3)$,
where
$(i,j,k)$ are distinct and non-zero modulo $7$.
\item $\alpha = \zeta_{84}^{n}(1 + \zeta^i_7  - (\zeta^j_7 + \zeta^k_7) \zeta_3)$, where
$(i,j,k)$ are distinct and non-zero modulo $7$.
\item $\alpha = \zeta_{84}^{n}(\zeta^i_7 + \zeta^j_7 + (\zeta^j_7 + \zeta^k_7) \zeta_3)$, where
$(i,j,k)$ are distinct modulo $7$.
\end{enumerate}
We compute in all cases that the
smallest  $\alpha + (\alpha - \overline{\alpha})(\zeta + \zeta^2)$
which occur are all $\ge 3.5$, or  $\alpha$ real and $\lambda = 0$.

\subsection{The case when $X = 4$, $p = 5$, and $\lambda \ne 0$ }

Since $X = 4$, by Lemma~\ref{lemma:reduce}, we may assume that
all the $\alpha_i$ are distinct.
We are assuming that $\lambda \ne 0$.
Then $\alpha - \overline{\alpha} = \lambda$.
Write 
$$\beta = \alpha + \alpha_1 \zeta + \alpha_2 \zeta^2 + \alpha_3 \zeta^3.$$
Then $\alpha_1 = \lambda$, $\alpha_2 - \overline{\alpha_3} = \lambda$.
Hence
$$\beta = \alpha + (\alpha - \overline{\alpha}) \zeta + 
(\overline{\gamma} + \alpha - \overline{\alpha}) \zeta^2 +
\gamma \zeta^3.$$ 
There is some symmetry in this expression.
If we let $\gamma = 
\overline{\theta} + \alpha - \overline{\alpha}$, then
$$\overline{\gamma} + \alpha - \overline{\alpha} = \theta.$$
This  sends the pair
$(\alpha - \gamma, \overline{\gamma} + \alpha - \overline{\alpha})
\mapsto (\overline{\alpha} - \overline{\theta},\theta)$.
It follows that the two terms $\gamma$ and $\theta$ can be interchanged
in various arguments.
We compute that
$$
\begin{aligned}
4 \M(\beta) = & \ \M(\alpha) + \M(\alpha - \overline{\alpha}) +
\M(\overline{\gamma} + \alpha - \overline{\alpha}) + \M(\gamma)
+ \M(\alpha) \\
+ & \  \M(\alpha - \gamma) + \M(\alpha - \gamma)
+ \M(\gamma) + \M(\overline{\gamma} + \alpha - \overline{\alpha})
+ \M(\alpha - \overline{\alpha})\\
= &  \ 2 \M(\alpha) + 2 \M(\gamma)
+ 2 \M(\alpha - \overline{\alpha}) + 2 \M(\alpha -
\gamma) + 
2 \M(\overline{\gamma} + \alpha - \overline{\alpha})
\end{aligned}$$
If $\alpha = \gamma$ then not every term is distinct, which is a contradiction, and
hence all the five terms in the sum above are non-zero.

\begin{lemma} \label{lemma:bigcomp1}
At least one of $\Num(\alpha)$ and $\Num(\gamma)$ is $\ge 3$.
\end{lemma}

\begin{proof}
We compute all  numbers such that $\Num(\alpha) \le 2$ or
$\Num(\gamma) \le 2$. We carry out the calculation as follows.
Suppose that $\alpha = \zeta_{84}^i + \zeta_{84}^j$ and $\gamma = \zeta_{84}^k + \zeta_{84}^{l}$.
Then we may assume that $l \ge k$, and that either:
\begin{enumerate}
\item $i = 1$,
\item $i = 3$ and $3 | j$,
\item $i = 4$ and $2 | j$,
\item $i = 7$ and $7 | j$,
\item $i = 12$ and $6  | j$,
\item $i = 21$ and $21 | j$,
\item $i = 28$ and $14 | j$.
\item $i = 84$ and $42 | j$.
\end{enumerate}
We remark that this computation also covers the cases where
$\Num(\alpha) = 1$ or $\Num(\gamma) = 1$, since $\zeta_{84}^k = \zeta_{84}^{k-14} + \zeta_{84}^{k + 14}$.
The smallest largest conjugate which occurs is $\sqrt{5}$, which occurs
in case $7$, and the second smallest largest conjugate is
$2 \cos(\pi/30) + 2 \cos(13 \pi/30)$, in case $4$.
Thus we have shown that at least one of $\Num(\alpha)$ or $\Num(\gamma)$
is $\ge 3$.
By symmetry, the same argument also proves that at least
one of $\Num(\alpha)$ or $\Num(\theta)$ is $\ge 3$.
\end{proof}

\begin{lemma} Either at least three of the terms $\M(\alpha), \M(\gamma), \M(\alpha - \overline{\alpha}), \M(\alpha -
\gamma)$ and $\M(\overline{\gamma} + \alpha - \overline{\alpha})$ above are roots of unity, \label{lemma:2or1}
or at least two terms are roots of unity and at least two other terms are the
sum of at most two roots of unity.
\end{lemma}

\begin{proof} If there is at most one root of unity, then,
 by Lemma~\ref{lemma:bound84},
 $$\M(\beta) \ge 1/2(5/3 \cdot 4 + 1) = 23/6.$$
If there are only two roots of unity, and only one other term which can
be expressed as the sum of exactly two roots of unity, then
$$\M(\beta) \ge 1/2(2 \cdot 2 +  5/3 + 1 + 1) = 23/6.$$
\end{proof}

We now consider possible pairs of terms which are roots of unity.
\begin{enumerate}
\item $\alpha$ and $\gamma$:  The result follows
from Lemma~\ref{lemma:bigcomp1}.
\item $\alpha$ and $\alpha - \overline{\alpha}$: The
latter is, up to a sign that we fix, $\sqrt{-1} = \zeta_{84}^{21}$, the former is therefore,
up to conjugation, $\zeta_{84}^7$. By Lemma~\ref{lemma:2or1}, either
one of the other terms is a root of unity, or at least two terms are the sum
of at most two roots of unity. If $\gamma$ is a root of unity, we reduce immediately
to case $1$. If $\theta = \overline{\gamma} + \alpha - \overline{\alpha}$ is a root of unity,
we also reduce to case $1$,  by symmetry.
 If $\alpha - \gamma$ is a root of unity, then
$\Num(\gamma) \le 2$.
On the other hand, if at least two terms are the sum of at most
two roots of unity, then either $\Num(\theta)$ or $\Num(\gamma)$ is $\le 2$, and by
symmetry, we may assume that $\Num(\gamma) \le 2$, and we are done
by Lemma~\ref{lemma:bigcomp1}.
\item $\alpha$ and $\alpha - \gamma$: We deduce immediately
that $\Num(\gamma) \le 2$, and hence, we are done by Lemma~\ref{lemma:bigcomp1}.
\item $\alpha$ and $\theta = \overline{\gamma} + \alpha - \overline{\alpha}$: This
reduces to case $1$ by symmetry.
\item $\gamma$ and $\alpha - \overline{\alpha}$:
The latter, after changing the sign of $\beta$,
 is $\sqrt{-1} = \zeta_{84}^{21}$. By
 Lemma~\ref{lemma:2or1}, either one
 of the other terms is a root of unity, or at least two
 terms are the sum of at most two roots of unity. Note that
 $\theta= \overline{\gamma}  + \alpha - \overline{\alpha}$
 is equal to
$\overline{\gamma} + \zeta_{84}^{21}$.   Suppose there is another root of unity.
We consider
various subcases:
\begin{enumerate}
\item $\theta$ is a root of unity:
From the three term vanishing sum
$\theta - \overline{\gamma} - \zeta_{84}^{21} = 0$ we deduce that
$\gamma = \zeta_{84}^{49}$ or $\zeta_{84}^{77}$. After conjugating
we may assume it is the first. Then
$$\beta = \alpha + \zeta_{84}^{21} \zeta + \zeta_{84}^{35} \zeta^2 + 
\zeta_{84}^{49} \zeta^3.$$
Now
$$\Num(\beta)  = 3/2 + \M(\alpha)/2 + \M(\alpha - \zeta_{84}^{49})/2.$$
Either $\M(\alpha) \le 23/10$ or $\M(\alpha - \zeta_{84}^{49}) \le 23/10$.
Since $23/10 < 5/2$, 
it follows from Lemma~\ref{lemma:bound84}
that either $\Num(\alpha) \le 3$ or 
$\Num(\alpha - \zeta_{84}^{49}) \le 3$.
Enumerating over all $\alpha$ with $\Num(\alpha) = 3$, we find that
the smallest value of the expression above is 
$$|(1 + \zeta_{84}^{42} + \zeta_{84}^{49}) +  \zeta_{84}^{21} \zeta_5^3 + \zeta_{84}^{35} \zeta_5 + \zeta_{84}^{49} \zeta_5^4 |
= \sqrt{1 + 4 \cos^2(\pi/15)} = 2.1970641\ldots$$
however,  the $\beta$ occuring here is not real, since we did not impose
the condition (in our computation) that $\alpha - \overline{\alpha} = \zeta_{84}^{21}$.
The second smallest value that occurs
is
$$|(1 + \zeta_{84}^{35} + \zeta_{84}^{42}) +  \zeta_{84}^{21} \zeta_5^4 + \zeta_{84}^{35} \zeta_5^{3} + \zeta_{84}^{49} \zeta_5^{2} |
= \sqrt{1 + 4 \cos^2(\pi/30)} = 2.226273\ldots$$
which is also not real. The third smallest value that occurs
is $2.574706\ldots > 2.303030\ldots = 76/33$.
 If $\Num(\alpha - \zeta_{84}^{49}) = 3$,
the smallest value thus obtained is
$$|(\zeta_{84}^{49} + 1 + \zeta_{84}^{28} + \zeta_{84}^{56}) +  \zeta_{84}^{21} \zeta_5^3 + \zeta_{84}^{35} \zeta_5 + \zeta_{84}^{49} \zeta_5^4 |
= \sqrt{1 + 4 \cos^2(\pi/15)},$$
the second smallest value is, as above, $\sqrt{1 + 4 \cos^2(\pi/30)}$, and the
third smallest value is (once more) $2.574706\ldots > 2.303030\ldots = 76/33$.
\item $\alpha$ is a root of unity: Since $\alpha$
and $\gamma$ are roots of unity, we are reduced to case $1$.
\item $\alpha - \gamma$ is a root of unity:  If $\gamma$ and $\alpha - \gamma$
are roots of unity, then $\Num(\alpha) \le 2$, and we are done
by Lemma~\ref{lemma:bigcomp1}.
 \end{enumerate}
Hence we may assume that all other terms are not roots of unity,
and hence there are at least two terms which are sums of at most
$2$ roots of unity.  We consider various possibilities:
\begin{enumerate}
\item Suppose
 that $\Num(\alpha) = 2$. Then we are done by
 Lemma~\ref{lemma:bigcomp1}.
 \item We may assume that $\gamma - \alpha$ and $\theta$
 are both at most the sum of two roots of unity. Write
 $\gamma = \zeta_{84}^i$ and $\alpha = \zeta_{84}^i + \zeta_{84}^j + \zeta_{84}^k$ with
 $i \le j \le k$. After conjugating, we may assume that $i$ divides $84$.
 Enumerating all the possibilities, we find that the smallest
 number of this form is
 $2 \cos(\pi/30) + 2 \cos(13 \pi/30) = 2.404867\ldots > 2.303030\ldots = 76/33$.
 \end{enumerate}
 \item $\gamma$ and $\alpha - \gamma$: Since $\Num(\alpha) \le 2$
 and $\Num(\gamma) = 1$, we are done by Lemma~\ref{lemma:bigcomp1}.
 \item $\gamma$ and $\theta: = \overline{\gamma} + (\alpha - \overline{\alpha})$:
 If $\Num(\alpha - \overline{\alpha}) = 1$ then we are back in case $5$.
 If $\Num(\alpha) =1$ we are back in case $1$.  If $\Num(\alpha - \gamma) = 1$  
 we are back in case $6$. Thus, by Lemma~\ref{lemma:2or1} it
 follows that at least one of $\Num(\alpha)$ or $\Num(\alpha - \gamma)$
 is equal to $2$. In the first case, we are done by Lemma~\ref{lemma:bigcomp1}.
 In the second case, we may let $\gamma = \zeta_{84}^i$ with $i | 84$ and
 $\alpha = \zeta_{84}^i + \zeta_{84}^j + \zeta_{84}^k$, and we are reduced to the computation in
 the final section of part $5$.
\item $\alpha - \overline{\alpha}$ and $\alpha - \gamma$:
If  either $\Num(\alpha) = 1$ or $\Num(\gamma) = 1$, then
the other is the sum of at most two roots of unity, and we are done
by Lemma~\ref{lemma:bigcomp1}. If $\theta$ is a root of unity, then by
symmetry we can reduce to case $5$. Thus, by Lemma~\ref{lemma:2or1},
we may assume that at least two of $\gamma$, $\alpha$ and $\theta$
are the sums of at most two roots of unity. By Lemma~\ref{lemma:bigcomp1},
we are done unless $\Num(\gamma) = \Num(\theta) = 2$, and $\Num(\alpha) \ge 3$.
Since $\alpha - \gamma$ is a root of unity, it must be the case that
$\Num(\alpha) = 3$.
Since $\alpha - \overline{\alpha}$ is a purely imaginary root of unity,
it must be $\pm \sqrt{-1}$. Changing the sign of $\beta$ if neccessary,
we may assume that $\alpha - \overline{\alpha} = \zeta_{84}^{21}$.
It follows that
$$(\alpha - \zeta_{84}^{7}) - \overline{(\alpha -  \zeta_{84}^{7})}= 0,$$
and hence $\alpha - \zeta_{84}^{7}$ is real. Since $2 \le \M(\alpha -\zeta_{84}^{7}) \le 4$,
and $\alpha$ lies in $\Q(\zeta_{84})$, 
it follows that $\alpha - \zeta_{84}^{7}$ is of the form:
\begin{enumerate}
\item $\zeta_{84}^{i} + \zeta_{84}^{-i}$
\item $ \zeta_{84}^{i} + \zeta_{84}^{-i} + 1$
\item $\zeta_{84}^{i} + \zeta_{84}^{-i} - 1$
\item  $\zeta_{84}^{i} + \zeta_{84}^{-i} + \zeta_{84}^{j} + \zeta_{84}^{-j}$
\item Galois conjugate to $\zeta_{84}^{-9} + \zeta_{84}^{-7} + \zeta_{84}^{3} + \zeta_{84}^{15}$
or  $\zeta_{84}^{-9} + \zeta_{84}^{-7} + \zeta_{84}^{3} + \zeta_{84}^{27}$.
\end{enumerate}
In all five cases, we let $\gamma =\zeta_{84}^j$ and enumerate
all possibilities. The smallest largest conjugate is a relatively
gargantuan $ 2.989043\ldots$
\item $\alpha -\overline{\alpha}$ and $\theta$: By symmetry, we
are reduced to case $5$.
\item $\alpha - \gamma$ and $\theta$: By symmetry, we are reduced
to case $6$.
\end{enumerate}

\subsection{The case when $X = 4$, $p = 5$, and $\lambda = 0$}
We have
$$\beta = \zeta \alpha + \zeta^{-1} \overline{\alpha} + \zeta^2 \gamma + \zeta^{-2} 
\overline{\gamma}.$$
Note that every term is distinct. We have
$$4 \M(\beta) = 2 \M(\alpha) + 2 \M(\gamma) + 2 \M(\alpha - \gamma)
+ 2 \M(\alpha - \overline{\gamma}) + \M(\alpha - \overline{\alpha})
+ \M(\gamma - \overline{\gamma}).$$
\begin{lemma} At least one of $\Num(\alpha)$ or $\Num(\gamma)$ is
at least \label{lemma:bigcomp2} $3$.
\end{lemma}

\begin{proof}
We compute all  numbers such that $\Num(\alpha) \le 2$ or
$\Num(\gamma) \le 2$. We carry out the calculation as follows.
Suppose that $\alpha = \zeta_{84}^i + \zeta_{84}^j$ and $\gamma = \zeta_{84}^k + \zeta_{84}^{l}$.
Then we may assume that $l \ge k$, and that either:
\begin{enumerate}
\item $i = 1$,
\item $i = 3$ and $3 | j$,
\item $i = 4$ and $2 | j$,
\item $i = 7$ and $7 | j$,
\item $i = 12$ and $6  | j$,
\item $i = 21$ and $21 | j$,
\item $i = 28$ and $14 | j$.
\item $i = 84$ and $42 | j$.
\end{enumerate}
We remark that this computation also covers the cases where
$\Num(\alpha) = 1$ or $\Num(\gamma) = 1$, since $\zeta_{84}^k = \zeta_{84}^{k-14} + \zeta_{84}^{k + 14}$.
We find that the smallest largest conjugates are $\sqrt{5}$, which is on our list, and
$2 \cos(\pi/30) + 2 \cos(13 \pi/30) \ge 76/33$.
\end{proof}

We note there is a symmetry between
$(\alpha,\gamma)$ and $(\overline{\alpha},\overline{\gamma})$.
Without loss of generality, we assume that
$\Num(\alpha) \ge \Num(\gamma)$, and that $\Num(\alpha) \ge 3$.

\begin{lemma}
At least one of the following holds: \label{lemma:dup}
\begin{enumerate}
\item At least two of $\{\gamma,\alpha - \gamma,\overline{\alpha} - \gamma\}$
are roots of unity.
\item Both $\alpha - \overline{\alpha}$ and $\gamma - \overline{\gamma}$ are roots of unity,
and every element in $\{\gamma,\alpha - \gamma,\overline{\alpha} - \gamma\}$ is
a sum of at most two roots of unity.
\end{enumerate}
\end{lemma}

\begin{proof} Note that $\Num(\alpha) \ge  3$, and so
$\M(\alpha) \ge 2$. Suppose that $\alpha - \overline{\alpha}$ and $\gamma - \overline{\gamma}$ are not both roots of unity, and at most one of
$\{\gamma,\alpha - \gamma,\overline{\alpha} - \gamma\}$ is a root of unity.
then
$$\M(\beta) \ge (1 + 5/3 + 5/3 + 2)/2 + (1 + 5/3)/4 =23/6.$$
Conversely, if $\alpha - \overline{\alpha}$ and
$\gamma - \overline{\gamma}$ are both roots of unity,
 at most one of $\{\gamma,\alpha - \gamma,\overline{\alpha} - \gamma\}$ is a root
 of unity, and at most two of  $\{\gamma,\alpha - \gamma,\overline{\alpha} - \gamma\}$
 are the sum of $2$ roots of unity, then
$$\M(\beta) \ge (1 + 5/3 + 2 + 2)/2 + (1+1)/4 = 23/6.$$
\end{proof}

\setcounter{subsubsection}{2}
\subsubsection{$X = 4$, $p = 5$, $\lambda = 0$, and two
of $\{\gamma,\alpha - \gamma,\overline{\alpha} - \gamma\}$ are roots of unity.}

If $\gamma$ is a root of unity, then so is $\overline{\gamma}$.
Since at least one of $\alpha - \gamma$
and $\alpha - \overline{\gamma}$ is also a root of unity, we deduce that
$\Num(\alpha) \le 2$, and we are done by Lemma~\ref{lemma:bigcomp2}.
Thus we may assume that $\Num(\alpha - \gamma) = \Num(\overline{\alpha} - \gamma) = 1$.
Recall that by Lemma~\ref{lemma:zero}, we may assume that
$\alpha$ and $\gamma$ are distinct from their conjugates.
Write $\alpha - \gamma = \zeta_{84}^i$ and $\overline{\alpha} - \gamma = \zeta_{84}^j$.
We deduce that $\alpha - \overline{\gamma} = \zeta_{84}^{-j}$.
Thus
$$\alpha - \overline{\alpha} = (\alpha - \gamma) - (\overline{\alpha} - \gamma)
= \zeta_{84}^i - \zeta_{84}^j$$
and
$$\gamma - \overline{\gamma} = (\alpha - \overline{\gamma}) - (\alpha - \gamma)
= \zeta_{84}^{-j} - \zeta_{84}^i$$
are purely imaginary. Since $\zeta_{84}^i - \zeta_{84}^j$ is purely imaginary, 
it follows that
$$\zeta_{84}^i - \zeta_{84}^j + \zeta_{84}^{-i} - \zeta_{84}^{-j} = 0.$$
This is a vanishing sum of length four, so it must be comprised of
two subsums of length $2$. If $\zeta_{84}^i = \zeta_{84}^j$ then $\alpha - \overline{\alpha} = 0$,
which is a contradiction. If $\zeta_{84}^i = \zeta_{84}^{-j}$, then $\gamma - \overline{\gamma} = 0$,
which is also a contradiction. Thus $\zeta_{84}^i = - \zeta_{84}^{-i}$ and $\zeta_{84}^j = - \zeta_{84}^{-j}$.
It follows that $\zeta_{84}^i = \pm \sqrt{-1}$ and $\zeta_{84}^j = \pm \sqrt{-1}$.  Yet, for each
of these possibilities, it is the case that $\zeta_{84}^i$ is equal to $\zeta_{84}^j$ or $\zeta_{84}^{-j}$,
and hence either $\alpha = \overline{\alpha}$ or $\gamma = \overline{\gamma}$,
a contradiction.

\subsubsection{$X = 4$, $p = 5$, $\lambda = 0$, at most one
of $\{\gamma, \alpha - \gamma,\overline{\alpha} - \gamma\}$ is a root of unity.}
It follows from Lemma~\ref{lemma:dup} that either
$\Num(\gamma) + \Num(\alpha - \gamma) \le 3$ or
$\Num(\gamma) + \Num(\alpha - \overline{\gamma}) \le 3$.
If $\Num(\gamma) = 1$, then we let
$\gamma = \zeta_{84}^i$, and $\alpha = \zeta_{84}^i +\zeta_{84}^j +\zeta_{84}^k$ and enumerate,
or $\alpha = \zeta_{84}^{-i} + \zeta_{84}^j + \zeta_{84}^k$ and enumerate.
If $\Num(\gamma) = 2$, we let
$\gamma = \zeta_{84}^{i} +\zeta_{84}^{j}$, and $\alpha = \zeta_{84}^{i} + \zeta_{84}^j + \zeta_{84}^k$,
or $\zeta_{84}^{-i} + \zeta_{84}^{-j} + \zeta_{84}^{k}$.
Enumerating over all such possibilities, we find that the smallest
largest conjugates that arise are $\sqrt{5}$ and
$2 \cos(\pi/30) + 2 \cos(13 \pi/30)$.

  \subsection{The case when $X = 5$ and $p = 5$}
  In this case, by Lemma~\ref{lemma:zero}, we can reduce to the
case that $X < 5$.
This completes the proof of Theorem~\ref{theorem:main}

\section{\texorpdfstring{$\M(\beta)$}{$M(b)$} is discrete in an interval beyond $2$} \label{section:discrete}

We have seen that the values of $\ho{\beta}$ for real cyclotomic integers
are discrete in
$[0,76/33]$ away from a limit point (from below) at $2$. In this section, we show
(now for all cyclotomic integers) that $\M(\beta)$ is discrete in $[0,9/4]$, away
from a limit point (from both sides) at $2$. This is an easy consequence of the
following theorem.

\begin{theorem} Let $\beta$ be a cyclotomic integer, and suppose that
$\M(\beta) < 9/4$. Then, up to a root of unity, either: \label{theorem:lazy}
\begin{enumerate}
\item $\beta = 0$ or $\beta = 1$.
\item $\beta$ is a sum of two roots of unity.
\item $\beta = 1 + \zeta^i_7 + \zeta^j_7$, where $(i,j)$ are distinct and non-zero.
\item $\beta = \zeta_3^{\pm 1} -  (\zeta^i_5 + \zeta^j_5)$ where $(i,j)$ are distinct and non-zero.
\end{enumerate}
\end{theorem}

\begin{proof}
Our proof follows the same lines as the arguments in 
sections~\ref{section:cassels1}--\ref{section:cassels4}, although it
is much easier. Assume that $\M(\beta) < 9/4$.
Suppose that $\beta \in \Q(\zeta_N)$, where $N$ is the conductor of
$\Q(\beta)$, and suppose that $\beta$ is \emph{minimal}, that is, no
root of unity times $\beta$ lies in a field  of smaller conductor.
Let $p^m \| N$, and
write $\beta = \sum \alpha_i \zeta^i$ where $\zeta$ is a $p^m$th root of unity
and the $\alpha_i \in \Q(\zeta_M)$, for $N = pM$.
If $p^2 | N$, then
$\displaystyle{\beta = \sum \M(\alpha_i)}$. If this sum consists of at least
three non-zero terms, then $\M(\beta) \ge 3$. If this sum consists of two 
non-zero terms,
and at least one of the $\alpha_i$ is not a root of unity, then
$\M(\beta) \ge 1 + 3/2 > 9/4$. Hence, either $\beta$ is the sum of two roots of unity,
or there is only one non-zero term, contradicting minimality.
Thus we may suppose that $N$ is squarefree. 

Suppose that $p|N$ for $p > 7$.
Since
$$\M(\beta) = 9/4 < \frac{11 + 1}{4},$$
by Lemma~1 of~\cite{MR0246852} we deduce that one can write $\beta$ as a sum of
$X \le (p-1)/2$ non-zero terms. Suppose that $X \ge 3$.
It follows from equation~\ref{equation:two} that
$$(p-1) \M(\beta) \ge X(p-X) \ge 3(p-3),$$
from which it follows that $\M(\beta) \ge 12/5 > 9/4$.
Thus we may assume that  $X = 2$, and $\beta = \alpha + \zeta \gamma$.
If $\alpha$ and $\gamma$ are roots of unity, then $\beta$ is a sum of two roots of unity.
If  at least one of $\alpha$ or $\gamma$ is not a root of unity,  then
$$(p-1) \M(\beta) \ge (p - 2)(1 + 3/2),$$
and hence $\M(\beta) \ge 9/4$, a contradiction.
Thus, we may assume that $N$ divides $105$.

 Now let us consider $\beta \in \Q(\zeta_{105})$.
 Write $\beta = \sum \alpha_i \zeta^i$, and suppose there are $X$ non-zero terms.
 We consider the various possible values of $X$, as in \S\ref{section:cassels4}.
 \begin{enumerate}
 \item If $X = 1$, then $\beta \in \Q(\zeta_{21})$.  Hence the result follows from Corollary~\ref{corr:9fourths}.
 \item If $X = 2$, then $\beta = \alpha + \gamma \zeta$, and
 $$4 \M(\beta) = 3 \M(\alpha) + 3 \M(\gamma) + \M(\alpha - \gamma).$$
 If $\alpha$ and $\gamma$ are roots of unity, then $\beta$ is a sum 
 of two roots  of unity.
 If $\alpha = \gamma$ is not a root of unity, then 
  $\M(\beta) \ge 9/4$. If $\alpha$ and $\gamma$ are distinct, and at least one
 is not a root of unity, then
 $$4 \M(\beta) \ge 3 (1 + 5/3) + 1,$$
 and it follows easily that $\M(\beta) \ge 9/4$.
 \item If $X = 3$, $\beta = \sum \alpha_i \zeta^i$, then
 we may assume that not all the $\alpha_i$ are the same, since otherwise
 we may subtract $\sum \zeta^i \alpha$ from $\beta$ and assume that $X = 2$.
 Thus, at least two of the $\alpha_i - \alpha_j$ are non-zero, and  hence
 $$4 \M(\beta) \ge 2 \sum \M(\alpha_i) + 2.$$
 If at least one of the $\alpha_i$ is not a root of unity, then $\M(\beta) \ge 7/3 > 9/4$.
 Thus, we may assume that all the $\alpha_i$ are roots of unity. Moreover, at least
 two of the $\alpha_i$ must coincide, since otherwise
 $4 \M(\beta) \ge 6 + 3$ and thus $\M(\beta) \ge 9/4$. We may therefore assume,
 after multiplying by a root of unity, that
 $$\beta = \alpha + \zeta^i + \zeta^j$$
 where $(i,j)$ are distinct and non-zero, and $\alpha$ is a root of unity.
 Since 
 $$4 \M(\beta) = 6 + 2 \M(\alpha - 1),$$
 we find that $\M(\beta) \ge 9/4$ unless $\alpha - 1$ is also a root of unity.
 If $\alpha - 1$ and $\alpha$ are both roots of unity then $\alpha = - \zeta_3^{\pm 1}$.  Hence, up to a root of unity, $\beta =  \zeta_3^{\pm 1} -(\zeta^i + \zeta^j)$.
 \item If $X = 4$, then we may assume that all the $\alpha_i$ are distinct.
 Then
 $$4 \M(\beta) \ge \sum \M(\alpha_i - \alpha_j) \ge 10,$$
 and $\M(\beta) \ge 5/2 > 9/4$.
 \item If $X =5$, we may subtract a multiple of $1+\zeta+\zeta^2+\zeta^3+\zeta^4 = 0$ to reduce to a previous case.
 \end{enumerate}
 \end{proof}
 
 \begin{remark} \emph{The exceptional values (with $\M(\alpha) = 2$)
 occurring in Theorem~\ref{theorem:lazy}
 were already noticed by Cassels \cite[Lemma~3]{MR0246852}.}
 \end{remark}
 
 The discreteness of $\M(\beta)$ away from $2$ follows from the fact that, given an $n$th root of unity
 $\zeta$, we have
 $$\M(1 + \zeta) = 2 \left(1 + \frac{\mu(n)}{\varphi(n)} \right),$$
 where $\mu(n)$ is the M\"{o}bius $\mu$-function and $\varphi(n)$ is
 Euler's totient function --- as $n$ increases this converges to $2$.
 
 We deduce the following:
 
\begin{corr} Let $\beta$ be a real cyclotomic integer, and suppose that \label{corr:94}
$\M(\beta) < 9/4$. Then, up to sign, either:
\begin{enumerate}
\item $\beta$ is conjugate to $2 \cos(2 \pi/n)$ for some integer $n$.
\item $\beta$ is conjugate to $1 + 2 \cos(2 \pi/7)$.
\item $\beta$ is conjugate to $\eta:=\zeta_{12} + \zeta_{20} + \zeta^{17}_{20} =  2 \cos(\pi/30) + 2 \cos(13 \pi/30)$.
\end{enumerate}
\end{corr}

\begin{proof} We use the fact (Lemma~\ref{lemma:84}) that if $\beta$ is totally real and 
$\Num(\beta) \le 3$,
then, up to sign, $\beta = 0$, $1$, $\eta$, $\zeta^{i} + \zeta^{-i}$, or
$1 + \zeta^i + \zeta^{-i}$ (the sign is unnecessary in the first, third, or forth cases).
 \end{proof}
 
\subsection{A general sparseness result on the set of values of $\M(\beta)$ for $\beta$ a cyclotomic integer}
\begin{theorem} \label{theorem:closedsubset} Let $\LL \subset \R$ denote the closure of the set of \label{theorem:closure}
real numbers of the form $\M(\beta)$ for cyclotomic integers $\beta$.
Then $\LL$ is a closed subset of $\Q$.
\end{theorem}

\begin{proof} If $U \subset \R$ is a set, let $U^n$ for any positive integer $n$
denote the set of sums of at most $n$ elements of $U$. If $U$ is closed, then
so is $U^n$.
Let $\LL(d) \subset \LL$ denote $\LL \cap [0,d]$.
Since $\LL(1) = \{0,1\}$,  it suffices to show that there exists
an integer $m$ (depending on $d$) such that
$$\LL(d + 1/2) \subset \LL(d)^m \cup \Q,$$
since then the result follows by induction.
Let $\gamma$ denote a point
in $\LL(d+1/2)$. There exists a sequence $\beta_k$ of
cyclotomic integers with $\M(\beta_k) = \gamma_k$ such that 
$\displaystyle{\lim_{\rightarrow} \gamma_k = \gamma}$. 
We note the following theorem of  Loxton \cite[\S 6.1, p.81]{MR0309896}:

\begin{theorem}[Loxton] There exists a continuous \label{theorem:loxton}
 increasing unbounded function
$g(t)$ such that $\M(\beta) \ge g(\Num(\beta))$. In particular, any bound on $\M(\beta)$
yields  an upper bound on $\Num(\beta)$. 
\end{theorem}

Since $\M(\beta_k) = \gamma_k$ converges to $\gamma \le d + 1/2$, it follows
that $\gamma_k$ is bounded above by $d + 1$ for sufficiently large $k$.
Without loss of generality, we may assume this bound holds for all $k$.
It follows from Loxton's theorem that the
 $\beta_k$ can be written as the sum of at most $m = m(d)$
roots of unity for some $m$.
Let $N_k$ denote the conductor of $\beta_k$. 
Recall that $\M(\alpha)\cdot [\Q(\alpha):\Q] \in \Z$.
If the $N_k$ are bounded, then the fields $\Q(\beta_k)$ are of bounded
degree, and hence the  $\M(\beta_k) = \gamma_k$ have bounded denominators,
and  $\M(\beta) \in \Q$. Hence, we may assume that the conductors
$N_k$ grow without bound.
Let $p^{n_k}_k$ denote the largest prime power divisor of $N_k$.
For each $k$, we may write
$\beta_k = \sum \alpha_i \zeta^i$, where the sum runs over a set of cardinality
$m$ (allowing some of the $\alpha_i$ to be zero). Assuming that $\beta_k$ is
minimal (which we may do without changing the value of $\M(\beta_k)$) we may
assume that there are at least two non-zero $\alpha_i$. We consider two cases:
\begin{enumerate}
\item Suppose that $n_k > 1$ for infinitely many $k$. For such $k$, we have
$$\M(\beta_k) = \sum \M(\alpha_i)$$
Since at least two of the $\alpha_i$ are non-zero, 
$\M(\alpha_i) \le \gamma_k - 1 < d$. Thus $\M(\alpha_i) \in \LL(d)$, and
$\M(\beta_k) \in \LL(d)^m$. Since the latter is closed, we deduce that
$\M(\beta) \in \LL(d)^m$.
\item Suppose that $n_k = 1$ for infinitely many $i$. We deduce that
$$(p_k - 1) \M(\beta_k) = (p_k - m) \sum \M(\alpha_i) +  \sum \M(\alpha_i - \alpha_j).$$
Since at least two of the $\alpha_i$ are non-zero, we deduce that
$$\M(\alpha_i) \le  \left(\frac{p_k-1}{p_k-m} \right) \cdot \gamma_k - 1 < d,$$
the last inequality holding for sufficiently large $k$ (equivalently, $p_k$).
Thus
$\M(\alpha_i) \le d$ for sufficiently large $k$. 
From the AM-GM inequality, we deduce that
$$\sum \M(\alpha_i - \alpha_j) \le 
2  \sum  \M(\alpha_i) + 2 \sum  \M(\alpha_j)
\le 4d  \binom{m}{2}.$$
As $k$ increases, therefore, the contribution of this term
to $\M(\beta_k)$ converges to zero, and hence
$$\M(\beta) = \lim_{\rightarrow} \M(\beta_i) = \lim_{\rightarrow} \sum \M(\alpha_i),$$
and thus $\gamma = \M(\beta)$ lies in the closure of $\LL(d)^m$.
Since $\LL(d)$ is closed, $\gamma \in \LL(d)^m$.
\end{enumerate}
\end{proof}

\begin{remark} \emph{Since closed subsets of $\Q$ are very far from being dense, we see that this result is in stark contrast to the analogous
set constructed out of $\M(\beta)$ for totally real integers $\beta$, which is
dense in $[2,\infty)$.}
\end{remark} 

\setcounter{theorem}{0}

\section{Galois groups of graphs} \label{section:graphs}

Let $\Gamma$ be a connected
graph with $|\Gamma|$ vertices. Fix a vertex $v$ of $\Gamma$, and
let $\Gamma_n$ denote the  sequence of graphs obtained by adding a $2$-valent tree
of length $n - |\Gamma|$ to $\Gamma$ at $v$.
Let $M_n$ denote the adjacency matrix of $\Gamma_n$, and let $P_n(x)$ denote the characteristic polynomial of $M_n$.
 By construction, $\Gamma_n$ has $n$ vertices, and thus
the degree of $P_n(x)$ is $n$.  The main result of this section is the following:

\begin{theorem} For any $\Gamma$, there exists an effective constant
$N$ such that for all $n \ge N$, either: \label{theorem:galois}
\begin{enumerate}
\item  All the eigenvalues of $M_n$ are of the form
$\zeta + \zeta^{-1}$ for some root of unity $\zeta$, and the graphs $\Gamma_n$
are the  Dynkin diagrams $A_n$ or $D_n$.
\item There exists at least one eigenvalue
$\lambda$ of $M_n$ of multiplicity one such
that  $\Q(\lambda^2)$ is not abelian.
\end{enumerate}
\end{theorem}

\begin{remark} \emph{We shall also prove a stronger version of this result which only looks at the largest eigenvalue
(Theorem~\ref{theorem:galois2}). We include this result because, although
Theorem~\ref{theorem:galois2} is also (in principle) effective, the bound
on $n$ arising in Theorem~\ref{theorem:galois} is easily computed, and all
our intended applications satisfy the conditions of
Theorem~\ref{theorem:galois}.}
\end{remark} 

\begin{corr}
For any $\Gamma$, there exists an effective constant $N$ such that for all $n \ge N$, either:
\begin{enumerate}
\item  $\Gamma_n$ is the  Dynkin diagram $A_n$ or $D_n$.
\item  $\Gamma_n$ is not the principal graph of a subfactor.
\end{enumerate}
\end{corr}
\begin{proof}
This is an immediate consequence of Theorem~\ref{theorem:galois} and
 Lemma~\ref{lemma:graphcyclo}.
\end{proof}

\subsection{Adjacency matrices}

We begin by recalling some basic facts about the eigenvalues of $M_n$.
\begin{lemma}
Let $x = t + t^{-1}$, and write $P_n(x) = F_n(t) \in \Z[t,t^{-1}]$.
\begin{enumerate}
\item The matrix $M_n$ is symmetric and the roots of $P_n(x)$ are all real.
\item The polynomials $P_n$ satisfy the recurrence:
$$P_n(x) = x P_{n-1}(x) - P_{n-2}(x).$$
\item There is a fixed Laurent polynomial $A(t) \in \Z[t,t^{-1}]$ such that:
$$F_n(t) \left(t - \frac{1}{t} \right)  = t^{n} \cdot A(t) - t^{-n} \cdot A(t^{-1}).$$
\end{enumerate}
\end{lemma}

We are particularly interested in the roots of $P_n(x)$ of absolute value larger than $ 2$,
or, equivalently, the real roots of $F_n(t)$ of absolute value larger than $1$. The following
facts will be useful to note.
\begin{lemma} Denote the roots of $P_n(x)$ by  \label{lemma:useful}
$\lambda_i$ for $i = 1$ to $n$.
\begin{enumerate}
\item If the roots of
$P_{n-1}(x)$
are $\mu_i$ for $i = 1$ to $n-1$, then, with the natural ordering of the roots,
$$\lambda_1 \le \mu_1 \le \lambda_2 \le \mu_2 \ldots
\le \mu_{n-1} \le \lambda_n.$$
\item The number of roots of $P_n(x)$ of absolute value larger than $ 2$ are bounded.
\item The largest real
root of  $P_n(x)$ is bounded.
\item For sufficiently large $n$, the real roots of $P_n(x)$ of
absolute value larger than $2$ are bounded uniformly away from $2$.
\end{enumerate}
\end{lemma}

\begin{proof} The first claim is the interlacing theorem;
see (\cite{MR1829620}, Theorem 9.1.1). By Descartes' rule of signs,
the polynomial $F_n(t)$ has a bounded number of real roots, which implies the second
claim. 
The largest real root of $F_n(t)$ converges to the largest real root
$\mut_{\infty}$ of $A(t)$
(compare Lemma~12 of~\cite{MR2169524})  and hence the largest real root of $P_n(x)$ converges to
$\lambda_{\infty} = \mut_{\infty} + \mut^{-1}_{\infty}$. The final claim follows
immediately from the first two. 
\end{proof}

We use the letter $\lambda$ to refer to a root of $P_n(x)$, and the letter $\rho$ to refer to the corresponding
roots of $F_n(t)$, where $\lambda = \rho + \rho^{-1}$.

\begin{lemma} There exists a polynomial $B(t)$ \label{lemma:repeat} such that
for for $n$ larger that some effectively computable constant, every repeated root of $F_n(t)$ on the unit circle is a root of $B(t)$.
\end{lemma}

\begin{proof} The polynomial $A(t)$ is monic. In particular, if $A(t)$ has a root
on the unit circle, then $A(t)$ has a factor $B(t)$ which is a reciprocal polynomial.
It follows that we can write
$$t^n \cdot F_n(t) \left( t - \frac{1}{t} \right) = B(t)  \left( t^{2n} \cdot C(t) -  C(t^{-1}) \right)$$
where $A(t) = B(t)C(t)$ and $C(t)$ has no roots on the unit circle.
Suppose that $F_n(t)$ has a repeated root $\mut$ on the unit circle.
Then either $\mut$ is a root of $B(t)$, or it is a root of
$t^{2n} C(t) - C(t^{-1})$. Yet the absolute value of the derivative
of this expression is, by the triangle inequality, greater than
$$ 2n |C(t)| - |C'(t)| - |C'(t^{-1})|.$$
Since $C(t)$ has no roots on the unit circle, for
all $n$ larger than some effectively computable constant this expression is positive.
\end{proof}

\begin{lemma} For all sufficiently large $n$, there exists a \label{lemma:suf2}
constant $K(\Gamma)$ such
that $$\sum (\lambda^2 - 2)^2 = 2n + K(\Gamma).$$
\end{lemma}

\begin{proof} Clearly $(\lambda^2 - 2)^2 = \rho^4 + 2 + \rho^{-4}$.  Since there is a pair of inverse roots of $F_n(t)$ corresponding to every  root $\lambda$ of $P_n(x)$, it follows that $\sum (\lambda^2 - 2)^2 = 2n + \sum \rho^4$.
The sum of the $4$th powers of the roots of $F_n(t)$ depends only on the first
four coefficients of $F_n(t)$, which is clearly independent of $n$, when $n$
is sufficiently large compared to $\deg(A)$.
\end{proof}

Recall that $\eta:=2 \cos(\pi/30) + 2 \cos(13 \pi/30)$ has degree
$8$ over $\Q$.
\begin{lemma} The polynomials  
$\displaystyle{\prod_{1,2,4} (x^2 - 3 - 2 \cos(2 \pi k/7))}$ 
and $\displaystyle{\prod_{i=1}^{8} (x^2 - 2 - \sigma_i \eta)}$  divide $P_n(x)$
a uniformly bounded and effectively computable number of times.
\end{lemma}
\begin{proof} Since the polynomials in question have at
least one real root larger than $2$, the number of factors of $P_n(x)$
of this form is clearly at most the number of real roots of $P_n(x)$ of size larger than $2$.
\end{proof}

Let us now complete the proof of Theorem~\ref{theorem:galois}.
By Lemma~\ref{lemma:repeat}, we deduce that for $n$ sufficiently large,
there are a uniformly bounded (with multiplicity) number of roots
which have multiplicity $\ge 2$.
Moreover, if $\Gamma_n$ is not $A_n$ or $D_n$, then the number
 roots of the form $\zeta + \zeta^{-1}$ is also uniformly and effectively bounded,
 by the main theorem of~\cite{MR2516970}. Finally, the number of roots $\lambda$ such
 that $\lambda^2 - 2 = 1 + 2 \cos(2 \pi/7)$ or $\eta$ is also uniformly bounded.
Let $\Ss$ denote the set of roots in any of these three categories. Clearly, we have
 $$\sum_{\lambda \notin \Ss} (\lambda^2 - 2)^2 \le 2n + K(\Gamma).$$
 On the other hand, by assumption, each $\lambda^2 - 2$ with
 $\lambda \notin \Ss$ is a cyclotomic integer.
 If $\lambda^2 - 2 = \zeta + \zeta^{-1}$, then $\lambda = \zeta^{1/2} -
 \zeta^{-1/2}$ lies in $\Ss$. 
  If $\lambda^2 - 2 = 1 + 2 \cos(2 \pi/7)$ or $\lambda^2 - 2 = \eta$, then $\lambda$ also lies in $\Ss$.
Thus,
 by Corollary~\ref{corr:94}, $\M(\lambda^2 - 2) \ge 9/4$ for all $\lambda \notin \Ss$.
 Hence
    $$2n + K(\Gamma) \ge \sum_{\lambda \notin \Ss} (\lambda^2 - 2)^2 \ge \frac{9(n - |\Ss|)}{4}.$$
   Combining these two inequalities, we obtain a contradiction
 whenever $n \ge 4 K(\Gamma) + 9 |\Ss|$, as long as $n$ is big enough
 for the conclusions of Lemma~\ref{lemma:repeat} and~\ref{lemma:suf2}
 to hold.
 
 \begin{remark} \emph{In practice, one can improve the bound on $n$
 by noting that the cyclotomic factors and repeated factors (that one
 knows explicitly) contribute to the sum $\sum (\lambda^2 - 2)^2$, thus
 enabling one to obtain a smaller bound on
 $\sum_{\notin \Ss} (\lambda^2 - 2)^2$.}
 \end{remark}
 
 \begin{remark} \emph{Suppose that $A(t)$ has exactly one root of absolute value larger than $1$.
  Then the polynomials $P_n(x)$ have a unique root  larger than $2$, and $P_n(x)$ factors
 as a Salem polynomial times a product of cyclotomic polynomials.
 (A Salem polynomial is an irreducible polynomial with a unique root of absolute value larger than $1$.)
  Similarly, if $\Gamma$
 is bipartite, and $A(t)$ has a pair of roots (equal up to sign) of absolute value larger than $1$, then
 $P_n(x)$ factors into cyclotomic polynomials and a factor $S(x^2)$ where
 $S(x)$ is a Salem polynomial --- in particular, in these cases, $P_n(x)$ will
 never have repeating roots that are not cyclotomic.}
 \end{remark}

\begin{remark}
\emph{In practice, the limiting factor in applying this argument is the bound coming from Gross-Hironaka-McMullen \cite{MR2516970} for roots of the form $\zeta_N + \zeta_N^{-1}$.  The argument in \cite{MR2516970} proceeds in two steps.  First, there is a uniform bound on $N$.  Second, for each fixed $N$ the $P_n$ which have such a root are precisely those in certain classes modulo $N$.  Let $\widetilde{A}$ be $A$ divided by all its cyclotomic factors, let $\ell(\widetilde{A})$ be the number of nonzero coefficients of $\widetilde{A}$.  The argument in \cite{MR2516970} shows that if $\zeta_N + \zeta_N^{-1}$ is a root of $P_n(x)$ for some $n$ such that $\zeta_N$ is not a root of $A_n(t)$, then $N$ divides $m \prod_{p\leq 2 \ell(\widetilde{A})} p$ for some integer $m \leq 4 \deg \widetilde{A}$ (this is not the exact statement of \cite[Thm 2.1]{MR2516970}, but the proof is the same).  It seems in the cases that we have looked at that there is a much stronger bound on $N$,  and proving an improved bound would substantially increase the effectiveness of our technique.}
\end{remark}

  \begin{example} \label{ex:threecases} \emph{We compute three applications of
  Theorem~\ref{theorem:galois}. Consider the graphs $\Gamma_{i,n}$ for $i = 1$, $2$, $3$,
  where the graphs $\Gamma_i$ are given below: 
   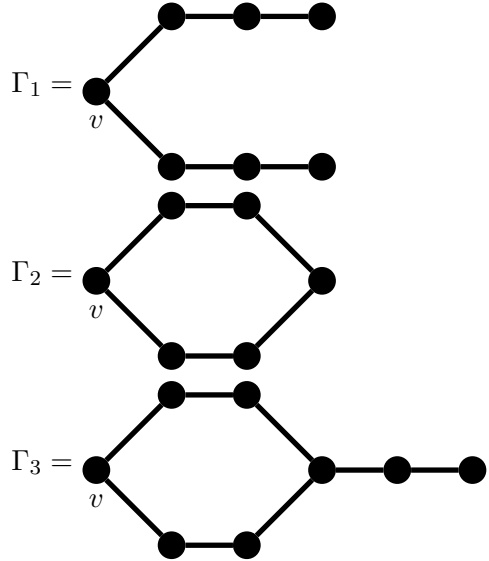
\begin{figure}[!h]
\begin{align*}
\Gamma_1 & = \begin{tikzpicture}[baseline, inner sep=3pt, line width=2pt]
\node[fill, draw, circle, label={below:$v$}](0) at (0,0) {};
\node[fill, draw, circle](1a) at (1,1) {};
\node[fill, draw, circle](2a) at (2,1) {};
\node[fill, draw, circle](3a) at (3,1) {};
\draw (0)--(1a)--(2a)--(3a);
\node[fill, draw, circle](1b) at (1,-1) {};
\node[fill, draw, circle](2b) at (2,-1) {};
\node[fill, draw, circle](3b) at (3,-1) {};
\draw (0)--(1b)--(2b)--(3b);
\end{tikzpicture} \\
\Gamma_2 & = \begin{tikzpicture}[baseline, inner sep=3pt, line width=2pt]
\node[fill, draw, circle, label={below:$v$}](0) at (0,0) {};
\node[fill, draw, circle](1a) at (1,1) {};
\node[fill, draw, circle](2a) at (2,1) {};
\node[fill, draw, circle](3) at (3,0) {};
\draw (0)--(1a)--(2a)--(3);
\node[fill, draw, circle](1b) at (1,-1) {};
\node[fill, draw, circle](2b) at (2,-1) {};
\draw (0)--(1b)--(2b)--(3);
\end{tikzpicture} \\
\Gamma_3 & = \begin{tikzpicture}[baseline, inner sep=3pt, line width=2pt]
\node[fill, draw, circle, label={below:$v$}](0) at (0,0) {};
\node[fill, draw, circle](1a) at (1,1) {};
\node[fill, draw, circle](2a) at (2,1) {};
\node[fill, draw, circle](3) at (3,0) {};
\draw (0)--(1a)--(2a)--(3);
\node[fill, draw, circle](1b) at (1,-1) {};
\node[fill, draw, circle](2b) at (2,-1) {};
\draw (0)--(1b)--(2b)--(3);
\node[fill, draw, circle](4) at (4,0) {};
\node[fill, draw, circle](5) at (5,0) {};
\draw (3)--(4)--(5);
\end{tikzpicture}
\end{align*}
\caption{The  graphs $\Gamma_{i}$.}
\end{figure}
The graphs $\Gamma_{1,n}$ and $\Gamma_{2,n}$ are  the two
infinite families
which arise
in the classification of Haagerup \cite{MR1317352}. 
It was shown by Bisch  \cite{MR1625762} (using a fusion ring argument)
that none of the $\Gamma_{2,n}$ are the principal graph of a subfactor.
The corresponding result for $\Gamma_{1,n}$ and $n > 10$
was proved by 
Asaeda--Yasuda \cite{MR2472028} was proved using number theoretic
methods. 
The family $\Gamma_{3,n}$ is one of several families arising in ongoing
work of V.~Jones, Morrison, Peters, Penneys, and Snyder, extending the classification
of Haagerup beyond $3+\sqrt{3}$.
We compute that
$$K(\Gamma_1) = 2, \qquad K(\Gamma_2) = 4, \qquad K(\Gamma_3) = 8,$$
where Lemma~\ref{lemma:suf2} applies for
$n \ge 8$,  $n \ge 7$, and $n \ge 11$ respectively.
Similarly, we find that
the cyclotomic factors of $P_n(x)$ depend (for $n \ge 11)$ 
 only on  $n \mod 24$, $n \mod 12$, and $n \mod 24$ for $i = 1$, $2$, $3$, 
and have degree at most $9$, $6$, and $8$ respectively.
The polynomials $A(t)$ are given as follows:
$$
\begin{aligned}
A_1(t)  = & \  (t^2 + 1)(t^4 + 1)(t^6 - t^4 - t^2 - 1) t^{-11} \\
A_2(t) = & \  (t^2 - t +  1)(t^2 + t + 1)(t^6 - 2 t^4 - 1) t^{-9} \\
A_3(t) = & \  (t^2 - t +  1)(t^2 + t + 1)(t^{10} - 2 t^8 - t^6 - t^4 - 1) t^{-13} \\
\end{aligned}
$$
In each case, we deduce that the only repeated factors of $F_n(t)$  on the unit
circle can occur at roots of unity. In all cases, the graphs $\Gamma_{i,n}$
are bipartite, and, moreover, the polynomials $A_i(t)$ have a unique pair
of roots of absolute value larger than $1$. It follows that $P_n(x)$ can be written as the product of
cyclotomic factors and a factor $S(x^2)$, where $S(x)$ is a Salem polynomial.
From this we can directly eliminate the possible occurrence of a root
$\lambda$ of $P_n(x)$ of the form $\lambda^2 - 2 = 1 + 2 \cos(2 \pi/7)$ or $\lambda^2 - 2 = \eta$ whenever
the degree of $S(x)$ is greater than $7$, or when $n \ge 16$.
It follows that $\Gamma_{n,i}$ does not correspond to
a subfactor whenever $n \ge N$, where
$$
\begin{aligned}
N(\Gamma_1) = \ & 9 \cdot \Ss(\Gamma_1)  + 4 \cdot K(\Gamma_1) =
 9 \cdot 9 + 4 \cdot 2 = 89, \\
N(\Gamma_2) = \ & 9 \cdot \Ss(\Gamma_2)  + 4 \cdot K(\Gamma_2) = 
9 \cdot 6 + 4 \cdot 4 = 70, \\
N(\Gamma_3) = \ & 9 \cdot \Ss(\Gamma_3)  + 4 \cdot K(\Gamma_3) = 
9 \cdot 8 + 4 \cdot 8 = 104.
\end{aligned}
$$
We may explicitly enumerate the polynomials for smaller $n$, and our results
are as follows:
\begin{corr} The graphs $\Gamma_{i,n}$ are not the principal graphs
of subfactors for all $(i,n)$ with the possible exception of the pairs
$(i,n) = (1,7)$, $(1,8)$, $(1,10)$, $(1,14)$, $(2,6)$, $(2,7)$, $(2,8)$, $(2,9)$, $(2,11)$
and $(3,8)$.
In these cases, we observe the following possibilities:
\begin{enumerate}
 \item $\Gamma_{1,7} = A_7$, and $\|\Gamma\| = \lambda^2 =
 (2 \cos(\pi/8))^2 = 2 + \sqrt{2}$.
  \item $\Gamma_{1,8} = \widetilde{E_7}$, the extended Dynkin diagram
 of $E_7$, and $\|\Gamma\| = \lambda^2 = 4$.
 \item $\Gamma_{1,10}$ corresponds to the Haagerup subfactor \cite{MR1686551},
 and
 $\|\Gamma\| = \displaystyle{\lambda^2 = \frac{5 + \sqrt{13}}{2}}$.
 \item $\Gamma_{1,14}$ corresponds to the
 extended Haagerup subfactor \cite{0909.4099}, and
 $$\|\Gamma\| = 
 \displaystyle{\lambda^2 = 3 + \zeta + \zeta^{-1} + \zeta^{3} + \zeta^{-3}
 + \zeta^{4} + \zeta^{-4}}, \quad \text{with} \quad \zeta^{13} = 1.$$
 \item $\Gamma_{2,6} = \widetilde{A_5}$, the extended Dynkin
 diagram of $A_5$, and $\|\Gamma\| = \lambda^2 =
4$.
   \item $\Gamma = \Gamma_{2,7}$, and $\|\Gamma\| = \lambda^2 = 
  3 + \sqrt{2}$. 
  \item $\Gamma = \Gamma_{2,8}$, and $\|\Gamma\| = \lambda^2 =
 (5 + \sqrt{17})/2$.
  \item $\Gamma = \Gamma_{2,9}$, and $\|\Gamma\| = \lambda^2 =
   (7 + \sqrt{5})/2$.
   \item $\Gamma = \Gamma_{2,11}$, and 
   $\displaystyle{\|\Gamma\| = \lambda^2 = 2 - \zeta^4 - \zeta^{-4} - \zeta^6 - \zeta^{-6}}$
   for $\zeta^{13} = 1$.
   \item $\Gamma = \Gamma_{3,8} = \Gamma_{2,8}$
    \end{enumerate}
         In each of the cases $\Gamma_{2,7}$, $\Gamma_{2,8} = \Gamma_{3,8}$,
         $\Gamma_{2,9}$, and $\Gamma_{2,11}$, we may rule out the existence
   of a corresponding subfactor for each choice of fixed leaf by computing the global dimension $\Delta$ and checking that, for some Galois automorphism $\sigma$, the ratio $\sigma(\Delta)/\Delta$ is not an algebraic integer \cite{0810.3242}. 
     \end{corr}}
\end{example}

\setcounter{theorem}{0}

\section{An Extension of Theorem~\ref{theorem:galois}}  \label{section:graphs2}

In this section, we prove the following extension of Theorem~\ref{theorem:galois}.

\begin{theorem} For sufficiently large $n$, either: \label{theorem:galois2}
\begin{enumerate}
\item  All the eigenvalues of $M_n$ are of the form
$\zeta + \zeta^{-1}$ for some root of unity $\zeta$, and the graphs $\Gamma_n$
are the  Dynkin diagrams $A_n$ or $D_n$.
\item The largest eigenvalue $\lambda$ is greater than $2$,
and the field $\Q(\lambda^2)$ is not abelian.
\end{enumerate}
\end{theorem}

\begin{remark} \emph{The proof of this theorem was found before the proof of Theorem~\ref{theorem:galois}.  In our intended applications, all the
conditions of
Theorem~\ref{theorem:galois} are met, however, this generalization may still 
be of interest.}
\end{remark}

\begin{df} Let $\Cb_m(x)$ be the polynomial such that if $x = t + t^{-1}$
then $\Cb_m(x) = t^{m} + t^{-m}$.
\end{df}

\begin{remark} The polynomials $\Cb_{m}(x)$ are the Chebyshev polynomials, appropriately
scaled so that all their roots are contained in the interval $[-2,2]$. If $m$ is even, then
$\Cb_{m}(x)$ is a polynomial in $x^2$.
\end{remark}

 \subsection{Heights and algebraic integers}
 
 The goal of this section is to show that the fields  $\Q(\rho)$ for any real root
 $\rho > 1$ of $F_n(t)$ have degree asymptotically bounded below by a linear function in 
 $n$.

 Recall that the Weil height of an algebraic number $\gamma = \alpha/\beta$ such that
 $K = \Q(\gamma)$ is defined to be
 $$h(\gamma):=\frac{1}{[K:\Q]} \sum_{v} \log \max \{|\alpha|_v,|\beta|_v\}.$$
 If $\lambda_{\infty} \le 2$ then every root of $P_n(x)$ has absolute value
 at most $2$, and thus every root $\rho$ of $F_n(t)$ has absolute value $1$.
 Yet then $h(\rho) = 0$ for all roots $\rho$ of $F_n(t)$.
 A theorem of Kronecker says that $h(\gamma) > 0$ unless $\gamma$ is zero or a root of unity.
Hence, in this case, we are in the first case of Theorem~\ref{theorem:galois2}.

 The following lemma is well known, and is a consequence of the triangle inequality.
 
 \begin{lemma} If $\phi: \mathbf{P}^1 \rightarrow
  \mathbf{P}^1$ is a homomorphism of finite degree, then
 $h(\phi(P)) \ge \deg(\phi) \cdot h(P) + C(\phi)$, for some constant $C(\phi)$ depending only on $\phi$.
 \end{lemma}
 
 Using this, we may deduce the following:
  
\begin{lemma} There exists an explicit constant $c$  depending only on $\Gamma$
such that for sufficiently
large $n$, and for every root $\mut$ of $F_n(t)$ there is an inequality:
$$h(\mut) \le \frac{c}{n}.$$
\end{lemma}

\begin{proof}
Consider the rational map $\phi: \mathbf{P}^1 \rightarrow \mathbf{P}^1$ defined by sending $t$ to
$\displaystyle{\frac{A(t^{-1})}{A(t)}}$. Since $\phi(\mut) = \mut^{2n}$, we deduce that
$$2n \cdot h(\mut) = h(\mut^{2n}) = h(\phi(\mut)) \le \deg(\phi) \cdot h(\mut) + C(\phi).$$
The lemma follows, taking $c = C(\phi)$ and $n \ge \deg(\phi)$.
\end{proof}

\begin{lemma} There exists a constant $a$ such that
 if $\mut$ is a root of $F_n(t)$, then either $\mut$ is a root of unity  \label{lemma:degree}
 or $[\Q(\mut):\Q] \ge a \cdot n$ for sufficiently large $n$.
\end{lemma}

\begin{proof} For sufficiently large $n$, the real roots of
absolute value larger than $1$ of $F_n(t)$ are bounded away from $1$,
by Lemma~\ref{lemma:useful} (4).
If $\rho$ is a root of $F_n(t)$ that is not a root of unity, then it has at least one conjugate
of absolute value larger than $1$, by Kronecker's theorem.
It follows from the definition of height that for sufficiently large $n$,
$$[\Q(\rho):\Q] \cdot h(\mut) \ge d.$$
 for some absolute constant  $d$. In light of the previous lemma,
this suffices to prove the result with $a = d/c$.
\end{proof}

Note that if $\lambda = \rho + \rho^{-1}$, then $[\Q(\rho):\Q(\lambda)] \le 2$, 
and so the same
result (with a different $d$) applies to $[\Q(\lambda):\Q]$.

\begin{lemma} Fix an integer $m$. For sufficiently large $n$, if \label{lemma:amplification}
$\lambda$ is a root of $P_n(x)$, then
$$\frac{1}{[\Q(\lambda):\Q]} \sum \Cb^2_m(\sigma \lambda) \le 5,$$
where the sum runs over all conjugates of $\lambda$.
\end{lemma}

\begin{proof} If $|x| \le 2$ then $\Cb^2_m(x) \le 4$. If $\lambda = \rho + \rho^{-1}$
and $\rho$ is a root of unity the result is obvious. Thus we may assume (after conjugation if
necessary) that $\rho > 1$.  Suppose that $\lambda$ has $R$ conjugates of
absolute value larger than $2$. Each of these roots is bounded by $\lambda_{\infty}$,
and the number of such roots is also uniformly bounded, by
Lemma~\ref{lemma:useful}. Note that
$$\frac{1}{[\Q(\lambda):\Q]} \sum \Cb^2_m(\sigma \lambda) \le 4
+ R \cdot \frac{ \Cb^2_m(\lambda_{\infty})- 4}{[\Q(\lambda):\Q]}.$$
Since $[\Q(\lambda):\Q]$ becomes arbitrarily large by Lemma~\ref{lemma:degree}, the right hand side is
bounded by $5$ for sufficiently large $n$.
\end{proof}

The following result is an immediate consequence
of Loxton's theorem (Theorem~\ref{theorem:loxton}) quoted previously:

\begin{corr} If $\beta$ is a cyclotomic integer such that $\M(\beta) \le 5$, then $\Num(\beta)$ is bounded \label{cor:lox}
by some absolute constant,
which we denote by $C$.
\end{corr}

\subsection{Proof of Theorem~\ref{theorem:galois2}}

If $\lambda_{\infty} \le 2$ then
the first claim follows from~
\cite[Theorem 2]{MR0266799}. 
We may assume that
$\lambda_{\infty}  > 2$. By Lemma~\ref{lemma:useful} (4), we may assume that 
for all $n$, 
$P_n(x)$ has no roots in the interval $(2,\alpha)$ for some $\alpha > 2$.
Choose an even integer $m$ such that $\Cb_{m}(\alpha) > C$, where $C$
is to be chosen later.
By Lemma~\ref{lemma:amplification}, we deduce that if $n$ is sufficiently large, then
for any root $\lambda$ of $P_n(x)$, 
$$\M(\Cb_m(\lambda)) =
\frac{1}{[\Q(\lambda):\Q] } \sum \Cb^2_m(\sigma \lambda) \le 5.$$
We assume that $\Q(\lambda^2)$ is abelian for some $\lambda > 2$ and
derive a contradiction. Since $m$ is even,
$\beta = \Cb_m(\lambda) \in \Q(\lambda^2)$, and hence $\beta$ is cyclotomic.
Moreover, $\M(\beta) \le 5$.

Choosing $C$ to be as in the above corollary, we deduce that  $\Num(\beta) \le C$.
Since $\lambda > 2$, however,  $\lambda \ge \alpha$ and hence
$\beta > C$. Yet the sum of $C$ roots of unity has absolute value at most $C$,
by the triangle inequality.
This completes the proof of Theorem~\ref{theorem:galois2}.

\appendix 

\section{A pseudo-unitary fusion category with an object of dimension $\frac{\sqrt{3}+\sqrt{7}}2$.}\label{appendix}
\begin{center}
by Victor Ostrik
\end{center}

\subsection{} The goal of this Appendix is to construct a fusion category $\V$ over $\BC$ 
with an object $\bV$ such that $\FP(\bV)=\frac{\sqrt{3}+\sqrt{7}}2$ (notice that since 
$\frac{\sqrt{3}+\sqrt{7}}2<1+\sqrt{2}$, the object $\bV$ is automatically simple). We do not attempt 
to classify all fusion categories generated  by such an object. 

The category we construct is pseudo-unitary (i.e. it is endowed with a spherical
structure and $\FP(X)=\dim(X)$ for any object $X$); moreover all the categories considered in this
Appendix are pseudo-unitary as well.

\subsection{Preliminaries} In this section we collect necessary definitions and results.
We refer the reader to \cite{MR2183279, 0906.0620} for a general theory of fusion and braided fusion categories.

Let $\Cat$ be a pre-modular fusion category, see e.g. \cite[Definition 2.29]{0906.0620}. Following \cite{MR1936496}
we will consider
commutative associative unital algebras $A\in \Cat$ satisfying the following assumptions:

(i) $\dim \Hom(\be,A)=1$;

(ii) the pairing $A\ot A\to \be$ defined as a composition of the multiplication $A\ot A\to A$ and a non-zero
morphism $A\to \be$ is non-degenerate and $\dim (A)\ne 0$;

(iii) the balance isomorphism $\theta_A=\id_A$.

In \cite{MR1936496} the algebras $A$ satisfying these conditions were called ``rigid $\Cat-$algebras with 
$\theta_A=\id_A$''; to abbreviate we will call such algebras ``$\Cat-$algebras'' here.

Given a pre-modular fusion category $\Cat$ and a $\Cat-$algebra $A\in \Cat$ one considers the category
$\Cat_A$ of right $A-$modules. The category $\Cat_A$ has a natural structure of spherical fusion category,
see \cite[Theorem 3.3, Remark 1.19]{MR1936496}. 
It contains a full fusion subcategory $\Cat_A^0$ of {\em dyslectic} modules, see \cite[Definition 1.8]{MR1936496}. 
The category $\Cat_A^0$ has a natural structure of pre-modular category. If $\Cat$
is pseudo-unitary the same is true for $\Cat_A$ and $\Cat_A^0$.

For a braided fusion category $\Cat$ let $\Cat^{op}$ denote the {\em opposite} category ($\Cat^{op}=\Cat$ as
a fusion category and the braiding in $\Cat^{op}$ is the inverse of the braiding in $\Cat$).
 Let $\Center(\A)$ denote the Drinfeld center of a fusion category $\A$. 
 
 \begin{theorem} \label{dims} {\em (cf. \cite[Theorem 4.5]{MR1936496}, \cite[Remark 4.3]{0704.0195}, \cite[Theorem 2.15]{MR2183279})}
 
 Assume that the  category $\Cat$ is {\em modular}. We have
 
 {\em (i)} $\dim \Cat_A=\frac{\dim \Cat}{\dim(A)}$ and $\dim \Cat_A^0=\frac{\dim \Cat}{\dim(A)^2}$;
 
 {\em (ii)} the category $\Cat_A^0$ is modular;
 
 {\em (iii)} there is a braided equivalence $\Center(\Cat_A)=\Cat \boxtimes (\Cat_A^0)^{op}$. $\square$
\end{theorem}

Recall (see e.g. \cite[\S 2.12]{0906.0620}) that a braided fusion category $\E$ is called {\em Tannakian} 
if it is braided equivalent to the representation category $\Rep(G)$ of a finite group $G$. Let
$\E$ be a Tannakian subcategory of a braided fusion category $\Cat$. Recall (\cite[\S 5.4.1]{0906.0620})
that in this situation one defines a {\em fiber category} $\E'_\Cat \boxtimes_\E \Ve$.

\begin{theorem} \label{gr} {\em (\cite[Theorem 1.3]{0809.3031})}
Let $\Cat$ be a modular category with Tannakian subcategory $\E=\Rep(G)$.
Assume that $\E'_\Cat \boxtimes_\E \Ve \simeq \Center(\A)$ for a fusion category $\A$. Then $\Cat \simeq
\Center(\B)$ where $\B=\bigoplus_{g\in G}\B_g$ is a faithfully $G-$graded fusion category with neutral
component $\B_1$ equivalent to $\A$. $\square$
\end{theorem}

\subsection{Affine Lie algebras and conformal embeddings} \label{afce}
Let $\g$ be a finite dimensional simple Lie algebra and let $\hat \g$ be the corresponding affine Lie algebra, see e.g. \cite[\S 7.1]{MR1797619}. For $k\in \BZ_{>0}$ let $\Cat(\g,k)$ denote the category of integrable highest weight $\hat \g-$modules of level $k$ (this category is denoted by
$\cO_k^{int}$ in {\em loc. cit.}). 
It is well known that the category $\Cat(\g,k)$ has a natural structure of pseudo-unitary 
modular tensor category, see e.g. \cite[Theorem 7.0.1]{MR1797619}.
The unit object of the category $\Cat(\g,k)$ is the {\em vacuum $\hat \g-$module}
of level $k$.

Let $\g \subset \g'$ be an embedding of simple (or, more generally, semisimple) Lie algebras. It
defines an embedding $\hat \g \subset \hat \g'$. This embedding does not preserve the level; we will
write $(\hat \g)_k\subset (\hat \g')_{k'}$ if the pullback of a $\hat \g'-$module of level $k'$ under this 
embedding is a $\hat \g-$module of level $k$ (it is clear that $k$ is uniquely determined by $k'$).
Recall (see e.g. \cite{MR1424041}) that a {\em conformal embedding} $(\hat \g)_k\subset (\hat \g')_{k'}$ is an embedding as above
such that the pullback of any module from $\Cat(\g',k')$ is a {\em finite} direct sum of modules from
$\Cat(\g,k)$. Let $(\hat \g)_k\subset (\hat \g')_{k'}$ be a conformal embedding. Then the pullback of
the vacuum $\hat \g'-$module of level $k'$ is an object $A$ of $\Cat(\g,k)$ which has a natural structure 
of $\Cat(\g,k)-$algebra, see \cite[Theorem 5.2]{MR1936496}. Moreover, there is a natural
equivalence $\Cat(\g,k)_A^0\simeq \Cat(\g',k')$, see {\em loc. cit.}

\begin{example} \label{toy}
\emph{The following is a toy version of our main construction.
There exists a conformal embedding $(\hat sl_2)_4\subset (\hat sl_3)_1$, see 
e.g. \cite{MR1424041}.
Let $A_0\in \Cat(sl_2,4)$ be the corresponding $\Cat(sl_2,4)-$algebra.
Recall (cf. \cite[\S 3.3]{MR1797619}) that the category $\Cat(sl_2,4)$ has 5 simple objects of dimensions
$1,\sqrt{3},2,\sqrt{3},1$; in particular $\dim \Cat(sl_2,4)=12$. The category $\Cat(sl_3,1)$ is pointed
with underlying group $\BZ/3\BZ$; in particular $\dim \Cat(sl_3,1)=3$. We deduce from 
Theorem \ref{dims} (i) that $\dim(A_0)=2$ and $\dim \Cat(sl_2,4)_{A_0}=6$. Notice that the category
$\Cat(sl_2,4)_{A_0}$ contains an object of dimension $\sqrt{3}$ since its center does (see Theorem
\ref{dims} (iii)); this object is automatically simple. It follows that
the category $\Cat(sl_2,4)_{A_0}$ has precisely 4 simple objects: 3 from the subcategory 
$\Cat(sl_2,4)_{A_0}^0$ and one more of dimension $\sqrt{3}$. Furthermore this implies that the
category $\Cat(sl_2,4)_{A_0}$ is a Tambara-Yamagami category associated to $\BZ/3\BZ$ \cite{MR1659954}.
In particular, $\Cat(sl_2,4)_{A_0}$ is $\BZ/2\BZ-$graded with neutral component 
$\Cat(sl_2,4)_{A_0}^0=\Cat(sl_3,1)$.}

\emph{We now show that this example is an illustration of Theorem \ref{gr}.
Since $\dim(A_0)=2$, we see that $A_0$ is a direct sum of two invertible objects. It follows that the
subcategory $\E$ of $\Cat(sl_2,4)$ generated by the invertible objects is Tannakian and is equivalent
to $\Rep(\BZ/2\BZ)$ (see also \cite[Theorem 6.5]{MR1936496}). It follows from the definitions that in this case 
$\E'_{\Cat(sl_2,4)}\boxtimes_\E \Ve=\Cat(sl_2,4)_{A_0}^0=\Cat(sl_3,1)$, 
see e.g. \cite[Proposition 4.56 (i)]{0906.0620}.
Notice that $\E=\E \boxtimes \be$ can be considered as a subcategory of $\Cat(sl_2,4)\boxtimes
\Cat(sl_3,1)^{op}$. Clearly we have
$$\E'_{\Cat(sl_2,4)\boxtimes \Cat(sl_3,1)^{op}}\boxtimes_\E \Ve=
(\E'_{\Cat(sl_2,4)}\boxtimes_\E \Ve)\boxtimes \Cat(sl_3,1)^{op}=
\Cat(sl_3,1)\boxtimes \Cat(sl_3,1)^{op}.$$
Since $\Cat(sl_3,1)\boxtimes \Cat(sl_3,1)^{op}=\Center(\Cat(sl_3,1))$ (see e.g. \cite[Proposition 3.7]{0906.0620}), 
Theorem \ref{gr} says that 
$\Cat(sl_2,4)\boxtimes \Cat(sl_3,1)^{op}=\Center(\B)$ where $\B$ is a $\BZ/2\BZ-$graded category with
neutral component $\Cat(sl_3,1)$. This is indeed so since by Theorem \ref{dims} (iii) 
$$\Center(\Cat(sl_2,4)_{A_0})=\Cat(sl_2,4)\boxtimes (\Cat(sl_2,4)_{A_0}^0)^{op}=
\Cat(sl_2,4)\boxtimes \Cat(sl_3,1)^{op}.$$}
\end{example}

\subsection{Izumi-Xu category $\IX$} We will consider here another example for the formalism from
\S \ref{afce}. Let $\g_{G_2}$ and $\g_{E_6}$ be the simple Lie algebras of type $G_2$ and $E_6$. There exists a conformal embedding $(\hat \g_{G_2})_3\subset (\hat \g_{E_6})_1$, see e.g. \cite{MR1424041}.
Let $A_1\in \Cat(\g_{G_2},3)$ be the corresponding $\Cat(\g_{G_2},3)-$algebra. 

\begin{proposition} \label{IzX}
The category $\Cat(\g_{G_2},3)_{A_1}$ has precisely 4 simple objects 
$\be,\bg,\bg^2$ and $\bX$.
The subcategory generated by $\be,\bg,\bg^2$ is pointed with underlying group $\BZ/3\BZ$. 
The remaining fusion rules are
$$\bg\ot \bX=\bg^2\ot \bX=\bX\ot \bg=\bX\ot \bg^2=\bX; \; \bX \ot \bX=\be \oplus \bg\oplus \bg^2\oplus 3\bX.$$
\end{proposition}

\begin{proof} The category $\Cat(\g_{E_6},1)$ is pointed with underlying group $\BZ/3\BZ$. Hence
the category $\Cat(\g_{G_2},3)_{A_1}$ contains a pointed subcategory with underlying group $\BZ/3\BZ$,
namely $\Cat(\g_{G_2},3)_{A_1}^0\simeq \Cat(\g_{E_6},1)$. We will denote the simple objects of this
subcategory by $\be$ (the unit object), $\bg$ and $\bg^2$.

Using \cite[Theorem 7.0.2, Theorem 3.3.20]{MR1797619} one computes 
$$\dim \Cat(\g_{G_2},3)
=\frac{147}{(64\sin(\frac{\pi}{21})\sin(\frac{4\pi}{21})\sin(\frac{5\pi}{21})\sin(\frac{\pi}{7})
\sin(\frac{2\pi}{7})\sin(\frac{3\pi}{7}))^2}
=3\left(\frac{7+\sqrt{21}}2\right)^2.$$ 
Since $\dim \Cat(\g_{G_2},3)_{A_1}^0=3$, we deduce from Theorem \ref{dims} (i) that 
$\dim(A_1)=\frac{7+\sqrt{21}}2$ and $\dim \Cat(\g_{G_2},3)_{A_1}=\frac{21+3\sqrt{21}}2$. 
The sum of squares $\sum_id_i^2$ of the dimensions of simple objects of the category 
$\Cat(\g_{G_2},3)_{A_1}$ not
lying in $\Cat(\g_{G_2},3)_{A_1}^0$ is $\frac{15+3\sqrt{21}}2$. Notice that every $\alpha=d_i^2$ is a totally positive
algebraic integer satisfying $\ho{\alpha}=\alpha$. The proof of the following result is left to the reader:

\begin{lemma} \label{decom}
There are precisely three decompositions of $\frac{15+3\sqrt{21}}2$ into a sum of 
totally positive algebraic integers $\alpha$ satisfying $\ho{\alpha}=\alpha$, namely
\begin{enumerate}
\item \label{eqn:decom1}
$\frac{15+3\sqrt{21}}2=\frac{15+3\sqrt{21}}2;$
\item \label{eqn:decom2}
$\frac{15+3\sqrt{21}}2=\frac{5+\sqrt{21}}2+(5+\sqrt{21});$
\item \label{eqn:decom3}
$\frac{15+3\sqrt{21}}2=\frac{5+\sqrt{21}}2+\frac{5+\sqrt{21}}2+\frac{5+\sqrt{21}}2.\;$
\end{enumerate}

\end{lemma}

Notice that in cases (\ref{eqn:decom2}) and (\ref{eqn:decom3}) the abelian subgroup $\BZ\oplus \bigoplus_i\BZ d_i\subset \BC$ is not closed under multiplication. 
Hence the only possibility is the decomposition (1); thus the category $\Cat(\g_{G_2},3)_{A_1}$
has precisely one simple object $\bX$ that is not in $\Cat(\g_{G_2},3)_{A_1}^0$; moreover
$\dim (\bX)=\sqrt{\frac{15+3\sqrt{21}}2}=\frac{3+\sqrt{21}}2$. The result follows.
\end{proof}

A fusion category with fusion rules as in Proposition \ref{IzX} was constructed by
Izumi in \cite{MR1832764}. The construction presented here is due to Feng Xu \cite{Xu} (note that it is not clear
whether the two constructions produce equivalent categories). Thus we call the category
$\Cat(\g_{G_2},3)_{A_1}$ the {\em Izumi--Xu category} and denote it by $\IX$.

\begin{remark} \label{ZIX}
Both categories $\Cat(sl_3,1)$ and $\Cat(\g_{E_6},1)$ are pointed
with underlying group $\BZ/3\BZ$. One observes (using \cite[Theorem 3.3.20]{MR1797619}) that 
these categories are opposite to each other. In particular, Theorem \ref{dims} (iii) implies that
$$\Center(\IX)\simeq \Cat(\g_{G_2},3)\boxtimes \Cat(\g_{E_6},1)^{op}\simeq \Cat(\g_{G_2},3)\boxtimes 
\Cat(sl_3,1).$$
\end{remark}

\subsection{Main result} 

\begin{theorem} There exists a pseudo-unitary fusion category $\V$ such that

{\em (i)} $\Center(\V)\simeq \Cat(\g_{G_2},3)\boxtimes \Cat(sl_2,4)$;

{\em (ii)} $\V=\V_0\oplus \V_1$ is $\BZ/2\BZ-$graded with neutral component $\V_0$ equivalent to the 
Izumi-Xu category $\IX$;

{\em (iii)} $\V_1$ contains three simple objects of dimensions $\frac{\sqrt{3}+\sqrt{7}}2$ and a simple
object of dimension $\sqrt{3}$.
\end{theorem}

\begin{proof} We recall that the category $\Cat(sl_2,4)$ contains a Tannakian subcategory 
$\E\simeq \Rep(\BZ/2\BZ)$ such
that $\E'_{\Cat(sl_2,4)}\boxtimes_\E \Ve \simeq \Cat(sl_3,1)$, see Example \ref{toy}. Now we consider $\E=\be \boxtimes \E$ as a subcategory
of $\Center:=\Cat(\g_{G_2},3)\boxtimes \Cat(sl_2,4)$. Clearly, 
$\E'_\Center\boxtimes_\E \Ve \simeq \Cat(\g_{G_2},3)\boxtimes \Cat(sl_3,1)$. Thus Theorem \ref{gr} and Remark
\ref{ZIX} imply that $\Center \simeq \Center(\V)$ where $\V$ is $\BZ/2\BZ-$graded fusion category with
neutral component $\IX$. Thus (i) and (ii) are proved.

To prove (iii) we observe that the category $\Center$ contains an object of dimension $\sqrt{3}$; hence
the category $\V$ contains an object $\bM$ of dimension $\sqrt{3}$. The object $\bM$ is
automatically simple and is contained in $\V_1$. Obviously, $\bM \ot \bM=\be\oplus \bg \oplus \bg^2$. 
Hence $\bM \simeq \bM^*$ and $\Hom(\bM,\bX \ot \bM)=\Hom(\bM \ot \bM^*,\bX)=0$. Furthermore,
$\Hom(\bX \ot \bM,\bX \ot \bM)=\Hom(\bM,\bX^*\ot \bX \ot \bM)=\BC^3$. Thus, $\bX \ot \bM\in \V_1$ 
is a direct
sum of three distinct simple objects $\bV_1, \bV_2, \bV_3$, none of which is isomorphic to $\bM$. 
Since $\dim \V_1=\dim \V_0=\frac{21+3\sqrt{21}}2$, we get that
$$\dim (\bV_1)^2+\dim (\bV_2)^2+\dim (\bV_3)^2=\frac{15+3\sqrt{21}}2.$$
Using Lemma \ref{decom}, we see that 
$$\dim (\bV_1)=\dim (\bV_2)=\dim (\bV_3)=\sqrt{\frac{5+\sqrt{21}}2}=\frac{\sqrt{3}+\sqrt{7}}2.$$ Thus the theorem is proved.
\end{proof}
 
 \subsection{Fusion rules of the category $\V$} In this section we compute the fusion rules of the
 category $\V$ following a suggestion of Noah Snyder.
 
 First, at least one of the objects $\bV_1, \bV_2, \bV_3$ is self dual; we assume that $\bV_1$ is self dual and use notation $\bV :=\bV_1$. The dimension count shows that
 $$\bV\ot \bV\simeq \bV_2\ot \bV_2^*\simeq \bV_3\ot \bV_3^*\simeq \be \oplus \bX.$$
 It follows that $\bg \ot \bV\not \simeq \bV$ and $\bg^2 \ot \bV\not \simeq \bV$; 
 thus we can (and will) assume that $\bV_2=\bg \ot \bV$ and $\bV_3=\bg^2 \ot \bV$.
 
 We claim that $\bV \ot \bg \not \simeq \bg \ot \bV$. Assume for the sake of contradiction that $\bV \ot \bg \simeq
 \bg \ot \bV$. It follows that the Grothendieck ring $K(\V)$ is commutative (since it is generated by
 the classes $[\bg]$ and $[\bV]$). Thus \cite[Lemma 8.49]{MR2183279} implies that the map
 $K(\Center(\V))\ot \BQ \to K(\V)\ot \BQ$ is surjective.  But this is impossible since any object of 
 $\Center(\V)=\Cat(\g_{G_2},3)\boxtimes \Cat(sl_2,4)$ is self dual and $(\bg)^*=\bg^2\not \simeq \bg$.
 
 It follows that $\bV \ot \bg \simeq \bg^2\ot \bV$.
 The remaining fusion rules are easy to determine from the known information. We have
 
 \begin{proposition} The simple objects of the category $\V$ are $\be, \bg, \bg^2, \bX, \bM, \bV,
 \bg \bV:=\bg \ot \bV, \bg^2\bV:=\bg^2\ot \bV$. The fusion rules are uniquely determined by
 Proposition \ref{IzX} and
 $$\bV \ot \bg=\bg^2\bV,\; 
 \bX \ot \bM=\bM \ot \bX=\bV\oplus \bg\bV\oplus \bg^2\bV,\;
 \bX \ot \bV=\bV \ot \bX=\bM\oplus \bV\oplus \bg\bV\oplus \bg^2\bV,$$
$$ \bM\ot \bM=\be\oplus \bg\oplus \bg^2,\; \bM \ot \bV=\bV \ot \bM=\bX,\; \bV \ot \bV=\be \oplus \bX.\;
\square$$
 \end{proposition}
 
\newcommand{\arxiv}[1]{\href{http://arxiv.org/abs/#1}{\tt arXiv:\nolinkurl{#1}}}
\newcommand{\doi}[1]{\href{http://dx.doi.org/#1}{{\tt DOI:#1}}}
\newcommand{\euclid}[1]{\href{http://projecteuclid.org/getRecord?id=#1}{{\tt #1}}}
\newcommand{\mathscinet}[1]{\href{http://www.ams.org/mathscinet-getitem?mr=#1}{\tt #1}}
\newcommand{\googlebooks}[1]{(preview at \href{http://books.google.com/books?id=#1}{google books})}
\newcommand{\zentralblatt}[1]{\href{http://www.zentralblatt-math.org/zmath/en/advanced/?q=#1}{{\tt Zentralblatt #1}}}

\bibliographystyle{plain}
\bibliography{bibliography/bibliography.bib}

\newcommand{\noopsort}[1]{}\def\cprime{$'$} \def\cprime{$'$} \def\cprime{$'$}
\begin{thebibliography}{10}

\bibitem{MR2307421}
Marta Asaeda.
\newblock Galois groups and an obstruction to principal graphs of subfactors.
\newblock {\em Internat. J. Math.}, 18(2):191--202, 2007.
\newblock \mathscinet{MR2307421} \doi{10.1142/S0129167X07003996}
  \arxiv{math.OA/0605318}.

\bibitem{MR1686551}
Marta Asaeda and Uffe Haagerup.
\newblock Exotic subfactors of finite depth with {J}ones indices
  {$(5+\sqrt{13})/2$} and {$(5+\sqrt{17})/2$}.
\newblock {\em Comm. Math. Phys.}, 202(1):1--63, 1999.
\newblock \mathscinet{MR1686551} \doi{10.1007/s002200050574}
  \arxiv{math.OA/9803044}.

\bibitem{MR2472028}
Marta Asaeda and Seidai Yasuda.
\newblock On {H}aagerup's list of potential principal graphs of subfactors.
\newblock {\em Comm. Math. Phys.}, 286(3):1141--1157, 2009.
\newblock \mathscinet{MR2472028} \doi{10.1007/s00220-008-0588-0}
  \arxiv{0711.4144}.

\bibitem{MR1797619}
Bojko Bakalov and Alexander Kirillov, Jr.
\newblock {\em Lectures on tensor categories and modular functors}, volume~21
  of {\em University Lecture Series}.
\newblock American Mathematical Society, Providence, RI, 2001.
\newblock \mathscinet{MR1797619}.

\bibitem{0909.4099}
Stephen Bigelow, Kim Morrison, Emily Peters, and Noah Snyder.
\newblock Constructing the extended haagerup planar algebra, 2009.
\newblock \arxiv{0909.4099}.

\bibitem{MR1625762}
Dietmar Bisch.
\newblock Principal graphs of subfactors with small {J}ones index.
\newblock {\em Math. Ann.}, 311(2):223--231, 1998.
\newblock \mathscinet{MR1625762} \doi{http://dx.doi.org/10.1007/s002080050185}.

\bibitem{MR0246852}
J.~W.~S. Cassels.
\newblock On a conjecture of {R}. {M}. {R}obinson about sums of roots of unity.
\newblock {\em J. Reine Angew. Math.}, 238:112--131, 1969.
\newblock \mathscinet{MR0246852}.

\bibitem{MR0422149}
J.~H. Conway and A.~J. Jones.
\newblock Trigonometric {D}iophantine equations ({O}n vanishing sums of roots
  of unity).
\newblock {\em Acta Arith.}, 30(3):229--240, 1976.
\newblock \mathscinet{MR0422149}.

\bibitem{MR1424041}
Philippe Di~Francesco, Pierre Mathieu, and David S{\'e}n{\'e}chal.
\newblock {\em Conformal field theory}.
\newblock Graduate Texts in Contemporary Physics. Springer-Verlag, New York,
  1997.
\newblock \mathscinet{MR1424041}.

\bibitem{0704.0195}
Vladimir Drinfeld, Shlomo Gelaki, Dmitri Nikshych, and Victor Ostrik.
\newblock Group-theoretical properties of nilpotent modular categories, 2007.
\newblock \arxiv{0704.0195}.

\bibitem{0906.0620}
Vladimir Drinfeld, Shlomo Gelaki, Dmitri Nikshych, and Victor Ostrik.
\newblock On braided fusion categories i, 2009.
\newblock \arxiv{0906.0620}.

\bibitem{0809.3031}
Pavel Etingof, Dmitri Nikshych, and Victor Ostrik.
\newblock Weakly group-theoretical and solvable fusion categories, 2008.
\newblock \arxiv{0809.3031}.

\bibitem{MR2183279}
Pavel Etingof, Dmitri Nikshych, and Viktor Ostrik.
\newblock On fusion categories.
\newblock {\em Ann. of Math. (2)}, 162(2):581--642, 2005.
\newblock \mathscinet{MR2183279} \doi{10.4007/annals.2005.162.581}
  \arxiv{math.QA/0203060}.

\bibitem{MR1829620}
Chris Godsil and Gordon Royle.
\newblock {\em Algebraic graph theory}, volume 207 of {\em Graduate Texts in
  Mathematics}.
\newblock Springer-Verlag, New York, 2001.
\newblock \mathscinet{MR1829620}.

\bibitem{MR2516970}
Benedict~H. Gross, Eriko Hironaka, and Curtis~T. McMullen.
\newblock Cyclotomic factors of {C}oxeter polynomials.
\newblock {\em J. Number Theory}, 129(5):1034--1043, 2009.
\newblock \mathscinet{MR2516970}.

\bibitem{MR1317352}
Uffe Haagerup.
\newblock Principal graphs of subfactors in the index range
  {$4<[M:N]<3+\sqrt2$}.
\newblock In {\em Subfactors ({K}yuzeso, 1993)}, pages 1--38. World Sci. Publ.,
  River Edge, NJ, 1994.
\newblock \mathscinet{MR1317352} available at
  \url{http://tqft.net/other-papers/subfactors/haagerup.pdf}.

\bibitem{MR0296043}
H.~Iwaniec.
\newblock On the error term in the linear sieve.
\newblock {\em Acta Arith.}, 19:1--30, 1971.
\newblock \mathscinet{MR0296043}.

\bibitem{MR1832764}
Masaki Izumi.
\newblock The structure of sectors associated with {L}ongo-{R}ehren inclusions.
  {II}. {E}xamples.
\newblock {\em Rev. Math. Phys.}, 13(5):603--674, 2001.
\newblock \mathscinet{MR1832764} \doi{10.1142/S0129055X01000818}.

\bibitem{MR0125047}
Ernst Jacobsthal.
\newblock \"{U}ber {S}equenzen ganzer {Z}ahlen, von denen keine zu {$n$}
  teilerfremd ist. {I}, {II}, {III}.
\newblock {\em Norke Vid. Selsk. Forh. Trondheim}, 33:117--124, 125--131,
  132--139 (1961), 1961.
\newblock \mathscinet{MR0125047}.

\bibitem{MR0224587}
A.~J. Jones.
\newblock Sums of three roots of unity.
\newblock {\em Proc. Cambridge Philos. Soc.}, 64:673--682, 1968.
\newblock \mathscinet{MR0224587}.

\bibitem{MR696688}
Vaughan F.~R. Jones.
\newblock Index for subfactors.
\newblock {\em Invent. Math.}, 72(1):1--25, 1983.
\newblock \mathscinet{MR696688} \doi{10.1007/BF01389127}.

\bibitem{MR0209247}
Hans-Joachim Kanold.
\newblock \"{U}ber eine zahlentheoretische {F}unktion von {J}acobsthal.
\newblock {\em Math. Ann.}, 170:314--326, 1967.
\newblock \mathscinet{MR0209247}.

\bibitem{MR1936496}
Alexander Kirillov, Jr. and Viktor Ostrik.
\newblock On a {$q$}-analogue of the {M}c{K}ay correspondence and the {ADE}
  classification of {$\mathfrak{sl}_2$} conformal field theories.
\newblock {\em Adv. Math.}, 171(2):183--227, 2002.
\newblock \mathscinet{MR1936496} \arxiv{math.QA/0101219}
  \doi{10.1006/aima.2002.2072}.

\bibitem{an:053.1389cj}
L.~Kronecker.
\newblock {Zwei S\"atze \"uber Gleichungen mit ganzzahligen Coefficienten.}
\newblock {\em J. Reine Angew. Math.}, 53:173--175, 1857.
\newblock \zentralblatt{an:053.1389cj}.

\bibitem{MR0309896}
J.~H. Loxton.
\newblock On the maximum modulus of cyclotomic integers.
\newblock {\em Acta Arith.}, 22:69--85, 1972.
\newblock \mathscinet{MR0309896}.

\bibitem{MR2169524}
James McKee and Chris Smyth.
\newblock Salem numbers, {P}isot numbers, {M}ahler measure, and graphs.
\newblock {\em Experiment. Math.}, 14(2):211--229, 2005.
\newblock \mathscinet{MR2169524} \euclid{euclid.em/1128100133}.

\bibitem{0810.3242}
Victor Ostrik.
\newblock On formal codegrees of fusion categories.
\newblock {\em Math. Research Letters}, 16(5):895--901, 2009.
\newblock \arxiv{0810.3242}.

\bibitem{MR1612877}
Bjorn Poonen and Michael Rubinstein.
\newblock The number of intersection points made by the diagonals of a regular
  polygon.
\newblock {\em SIAM J. Discrete Math.}, 11(1):135--156 (electronic), 1998.
\newblock \mathscinet{MR1612877} \arxiv{math/9508209}.

\bibitem{MR0012092}
Carl~Ludwig Siegel.
\newblock The trace of totally positive and real algebraic integers.
\newblock {\em Ann. of Math. (2)}, 46:302--312, 1945.
\newblock \mathscinet{MR0012092}.

\bibitem{MR0266799}
John~H. Smith.
\newblock Some properties of the spectrum of a graph.
\newblock In {\em Combinatorial {S}tructures and their {A}pplications ({P}roc.
  {C}algary {I}nternat. {C}onf., {C}algary, {A}lta., 1969)}, pages 403--406.
  Gordon and Breach, New York, 1970.
\newblock \mathscinet{MR0266799}.

\bibitem{MR736460}
C.~J. Smyth.
\newblock The mean values of totally real algebraic integers.
\newblock {\em Math. Comp.}, 42(166):663--681, 1984.
\newblock \mathscinet{MR736460}.

\bibitem{MR1659954}
Daisuke Tambara and Shigeru Yamagami.
\newblock Tensor categories with fusion rules of self-duality for finite
  abelian groups.
\newblock {\em J. Algebra}, 209(2):692--707, 1998.
\newblock \mathscinet{MR1659954}.

\bibitem{Xu}
Feng Xu.
\newblock Unpublished notes, 2001.

\end{thebibliography}

\end{document}